\documentclass{amsart}
\usepackage{amsmath,amssymb,mathrsfs,euscript,color,mathabx,stmaryrd,mathabx}
\usepackage[colorlinks]{hyperref}
\usepackage[shortlabels]{enumitem}
\usepackage{pdfsync}
\usepackage[backgroundcolor=green!40, linecolor=green!40]{todonotes}
\usepackage[utf8]{inputenc}                                    
\usepackage[T1]{fontenc}

\input{xypic}
\usepackage{tikz-cd}

\usepackage[margin=1.2in]{geometry}

\newcounter{makeconstant}
\newenvironment{makeconstant}%
{\refstepcounter{makeconstant}}%
{}

\newcounter{nonumber}

\theoremstyle{plain}
\newtheorem{theorem}[equation]{Theorem}
\newtheorem{lemma}[equation]{Lemma}
\newtheorem{corollary}[equation]{Corollary}
\newtheorem{proposition}[equation]{Proposition}

\newtheorem{conjecture}[equation]{Conjecture}

\theoremstyle{definition}
\newtheorem{example}[equation]{Example}
\newtheorem{definition}[equation]{Definition}

\newtheorem{remark}[equation]{Remark}
\newtheorem{remarks}[equation]{Remarks}

\newtheorem*{remark*}{Remark}

\numberwithin{equation}{section} 
\def\A{{\rm A}}
\def\B{{\rm B}}
\def\C{{\rm C}}
\def\D{{\rm D}}
\def\E{{\rm E}}
\def\F{{\rm F}}
\def\G{{\rm G}}
\def\H{{\mathrm{H}}}
\def\I{{\rm I}}
\def\J{{\rm J}}
\def\K{{\rm K}}
\def\L{{\rm L}}
\def\M{{\rm M}}
\def\N{{\rm N}}
\def\O{{\rm O}}
\def\P{{\rm P}}

\def\R{{\rm R}}

\def\T{{\rm T}}
\def\U{{\rm U}}
\def\V{{\rm V}}
\def\W{{\rm W}}

\def\Z{{\rm Z}}

\def\Cc{\mathscr{C}}

\def\Ee{\mathscr{E}}

\def\Qq{\mathscr{Q}}

\def\ee{\mathrm{e}}

\def\kk{\mathfrak{k}}

\def\so{\mathsf{o}}

\def\b{\beta}

\def\e{\varepsilon}

\def\vphi{\varphi}

\def\k{\kappa}
\def\l{\lambda}

\def\o{\mathfrak{o}}
\def\p{\mathfrak{p}}
\def\s{\sigma}
\def\t{\theta}

\def\La{\Lambda}

\def\Th{\Theta}

\def\({\left(}
\def\){\right)}
\def\>{\geqslant}
\def\<{\leqslant}
\def\le{\leqslant}

\def\Hom{\operatorname{Hom}}
\def\End{\operatorname{End}}
\def\Aut{\operatorname{Aut}}

\def\GL{\operatorname{GL}}
\def\Gal{\operatorname{Gal}}

\def\id{\operatorname{id}}

\def\Ind{\operatorname{Ind}}

\def\ind{\operatorname{ind}}

\def\dim{\operatorname{dim}}

\def\cusp{\operatorname{cusp}}

\def\diag{\operatorname{diag}}

\def\gcd{\operatorname{gcd}}

\def\st{\operatorname{st}}

\def\Sp{\operatorname{Sp}}

\def\SO{\operatorname{SO}}

\def\SL{\operatorname{SL}}

\def\tG{{\widetilde{\G}}}

\def\tM{{\widetilde{\M}}}
\def\tP{{\widetilde{\P}}}

\definecolor{dorange}{RGB}{255,140,0}
\definecolor{dgreen}{RGB}{0,205,10}

\def\ignore#1{\relax}

\def\be{{\mathrm e}}

\def\rmf{{\mathrm f}}
\def\ee{{\mathbf e}}
\def\ff{{\mathbf f}}
\def\bdots{\mathinner{\mkern1mu\raise1pt\hbox{.}\mkern2mu\raise4pt\hbox{.}
           \mkern2mu\raise7pt\vbox{\kern7pt\hbox{.}}\mkern1mu}}

\def\presuper#1#2%
  {\mathop{}%
   \mathopen{\vphantom{#2}}^{#1}%
   \kern-\scriptspace%
   #2}
\def\presub#1#2%
  {\mathop{}%
   \mathopen{\vphantom{#2}}_{#1}%
   \kern-\scriptspace%
   #2}

\def\ft{{\mathfrak{t}}} %%%% Or other notation?
\newcommand{\bext}{\mathrm{beta}}
\newcommand{\ext}{\mathrm{ext}}

\usepackage{bbm}

\def\cV{\mathscr{V}}
\def\ft{\mathfrak{t}}

\def\fs{\mathfrak{s}}
\def\kk{\mathrm{k}}
\def\BK{\mathrm{BK}}
\def\cE{\mathscr{E}}
\def\Rep{\mathrm{Rep}}

\def\cW{\mathscr{W}}
\def\presuper#1#2%
  {\mathop{}%
   \mathopen{\vphantom{#2}}^{#1}%
   \kern-\scriptspace%
   #2}
\newcommand{\SAut}{\text{SAut}}
\def\cSE{\mathscr{S\!E}}
\def\endo{\operatorname{endo}}

   \theoremstyle{definition}

\def\Div{\mathtt{Div}}
\def\Herm{\mathtt{Herm}}

\usepackage{subfiles}

\title{Cuspidal endo-support and strong beta extensions}
\date{\today}
\author{David Helm}
\address{David Helm, Department of Mathematics, Imperial College, London, SW7 2AZ, United Kingdom.}
\email{d.helm@imperial.ac.uk }
\author{Robert Kurinczuk}
\address{Robert Kurinczuk, School of Mathematics and Statistics, University of Sheffield, Sheffield, S3 7RH, United Kingdom.}
\email{robkurinczuk@gmail.com}
\author{Daniel Skodlerack}
\address{Daniel Skodlerack, Institute of Mathematical Sciences, ShanghaiTech University,
201210, Pudong New District, Shanghai, China}
\email{dskodlerack@shanghaitech.edu.cn}
\author{Shaun Stevens}
\address{Shaun Stevens, School of Mathematics, University of East Anglia, Norwich, NR4 7TJ, United~Kingdom.}
\email{shaun.stevens@uea.ac.uk}

\begin{document}
\begin{abstract}
Let~$\G$ be an inner form of a general linear group or classical group over a non-archimedean local field of residual characteristic~$p$, {assumed odd in the classical case}. We prove that every smooth representation of~$\G$ over an algebraically closed field~$\R$ of characteristic~$\ell\neq p$ contains a maximal semisimple character, i.e., one for which the point in the building of the corresponding centralizer is a vertex. Further, for every endo-parameter adapted to~$\G$, we define its support, which leads also to the notion of cuspidal endo-support of an irreducible representation, and we relate this to its cuspidal support. We also introduce beta extensions for strong facets in the building of a centralizer, and {show these are sufficient for the construction of types}. These results are used in a subsequent paper to decompose the category of smooth~$\R$-representations of~$\G$.
\end{abstract}
\maketitle

\setcounter{tocdepth}{1}
\tableofcontents

\section{Introduction}

{Over the last thirty years, for inner forms of general linear and classical groups over a non-archimedean local field, }the theory of semisimple characters and the more conceptual theory of endo-classes and endo-parameters has been developed: starting with simple characters {of general linear groups} in~\cite{BK93} by Bushnell--Kutzko, extended to 
lattice sequences in~\cite{BH96} by Bushnell--Henniart, {where the authors introduce endo-classes of simple characters}; semisimple characters for classical groups {with $p$ odd} were introduced in~\cite{St05} by Stevens {with their endo-classes and endo-parameters considered by Kurinczuk--Skodlerack--Stevens in \cite{KSS}}; for inner forms of classical groups this theory developed {with $p$ odd} in~\cite{SkodInnerFormII} by Skodlerack; and for inner forms of general linear groups by S\'echerre in~\cite{SecherreCh}, {Broussous--S\'echerre--Stevens \cite{MR2889743},} and 
Skodlerack in~\cite{SkodInnerFormI}.  {Semisimple characters for the exceptional group~$\G_2$ over a non-archimedean local field of residual characteristic greater than~$3$ have also been introduced by Blasco--Blondel \cite{Blasco2012}.  }

Semisimple characters have {an inductive} arithmetic construction which makes it possible to compute intertwining between them,  a property crucial for their use in {the construction} of types for cuspidal representations \cite{BK93,SecherreStevensIV,St08,SkodlerackCuspQuart}
and Bernstein blocks \cite{MR1711578,SecherreStevensVI,MiSt,SkodlerackYe} in the cases of forms of general linear groups, and inner forms of classical groups { with~$p$ odd}.  
{They play an important role in work towards describing an explicit local Langlands correspondence for general linear groups, {following Bushnell--Henniart's seminal work, including in \cite{BH96,MR1990931,MR3664814}}; and have many applications, including to: the Jacquet--Langlands correspondence {as considered in \cite{MR4000000}}; and to studying distinguished representations{, see for example \cite{MR4275472,MR4009674,MR4867166}}. }

In this work we prepare the technical requisites for the application of semisimple characters to block decompositions for classical groups over~$\mathbb{Z}[1/p]$-algebras. We fix~$\R$ an algebraically closed field of characteristic~$l$ and~$\F$ a non-archimedian local field of residue characteristic~$p\neq l$. Let~$\G$ be 
an inner form of a general linear group defined over~$\F$, or the connected component of a classical group defined~$\F$ or a quadratic subextension of~$\F$ (in the unitary case) with odd~$p$. 

%%%%% Analyse category R_R(G) using endo-parameters %%% Construction of endo-factor
\subsection{$m$-realizations and exhaustion in endo-factors}
{We use semisimple characters to} analyse the category~$\Rep_\R(\G)$ of smooth~$\R$-representations of~$\G$.   
Given an endo-parameter~$\ft$ for~$\G$, i.e., an intertwining class of a semisimple $\R$-valued character defined on a compact open subgroup of~$\G$,  we introduce the full subcategory~$\Rep_\R(\ft)$ of~$\Rep_\R(\G)$ consisting of all representations which are generated 
by semisimple characters of class~$\ft$, called endo-factor for~$\ft$ in~$\Rep_\R(\G)$. We call such representations~\emph{of class} $\ft$. 

Note that~$\Rep_\R(\ft)$ is not obviously {a factor of~$\Rep_\R(\G)$ (and this is not needed in this paper)}, but we will see in {an accompanying paper \cite{HKSSBlocks}} in a more general setting that~$\Rep_\R(\ft)$ is indeed a factor of~$\Rep_\R(\G)$. 

The main theorem of the current paper states that we can find a nice collection of semisimple characters which exhausts~$\Rep_\R(\G)$; this is important for constructing finitely generated projective generators of~$\Rep_{\R}(\ft)$. 
These characters are chosen more {sophisticatedly}: A semisimple character~$\theta$ is parametrized by a pair~$(\Lambda,\beta)$ consisting of a semisimple element~$\b$ of the Lie algebra of~$\G$ and a lattice sequence~$\Lambda$ which corresponds to a point in the
 Bruhat-Tits building~$\mathfrak{B}(\G)$ of~$\G$ which lies in the image of a canonical embedding~$j_\b:\ \mathfrak{B}(\G_\b)\hookrightarrow\mathfrak{B}(\G)$ of the building of the centralizer~$\G_\b$. 
 We show we can choose the characters~$\theta$ such that its point in~$\mathfrak{B}(\G_\b)$ defines a maximal parahoric subgroup in~$\G_\b$. 
 Such semisimple characters are called~\emph{$m$-realizations of~$\ft$}. 
 
 \begin{theorem}[see Theorem~\ref{thm43}]\label{thmExhaustionEndofactor}
 Let~$\ft$ be an endo-parameter for~$\G$. Then, every representation in~$\Rep_\R(\ft)$ is generated by the sum of the isotypic components of all~$m$-realizations of~$\ft$. 
 \end{theorem} 
 
 {For classical and general linear groups, this refines a Theorem of Dat \cite{Dat09}, which says the category is generated by the sum of the isotypic components of all semisimple characters of endo-parameter~$\ft$.}  The main advantage {of this improvement to}~$m$-realizations, is that {$m$-realizations} satisfy an important finiteness property. In general, there are infinitely many~$\G$-conjugacy classes of~$m$-realizations of~$\ft$, but if we take for each~$m$-realization an Iwahori decomposition into account 
 then we can obtain a coarser equivalence relation to obtain a finite number of equivalence classes:
  The idempotents of~$\F[\b]$ define a Levi subgroup~$\M(\b,\G)$ of~$\G$. We call two~$m$-realizations~$\theta,\ \theta'$ of~$\ft$ \emph{essentially $\G$-conjugate} if their restrictions~$\theta|_{\M(\b,\G)}$ and~$\theta'|_{\M(\b',\G)}$ are conjugate in~$\G$. 
We obtain the following Theorem:

\begin{theorem}[see Theorem \ref{thmEssentiallyConjugateClasses}]\label{thmFiniteness}
Let~$\ft$ be an endo-parameter for~$\G$. There are only finitely many essential~$\G$-conjugacy classes of~$m$-realizations of~$\ft$. 
\end{theorem}

%%%%%%%% Cuspidal support for endp-parameters
\subsection{Cuspidal support for endo-parameters}
From the theory of covers, an irreducible cuspidal representations~$\pi$ of~$\G$ only contains semisimple characters in~$\G$ for which the point in the building of the centralizer~$\G_\b$ is a vertex in the strong simplicial structure and the split rank of the 
centre of~$\G$ is equal to the split rank of the centre of~$\G_\b$. The first property {can be realized for any} endo-parameter, but the second {cannot; those endo-parameters for which it is possible} are precisely the ones {whose} endo-factor contains an irreducible cuspidal representation. 
We call them \emph{cuspidal endo-parameters}. 

For a general endo-parameter~$\ft$ we describe the endo-factor~$\Rep_\R(\ft)$ using a cuspidal endo-parameter on a Levi subgroup of~$\G$ as follows:
Given a realization~$\theta$ of~$\ft$, as~$\b$ generates a product~$\F[\b]$ of field extensions, it defines a Levi subgroup~$\M_c$ of~$\G$, which is determined up to conjugacy by~$\ft$. 
We then obtain a cuspidal endo-parameter~$\ft_{\M_c}$ on~$\M_c$ via restriction of~$\theta$ to~$\M_c$. The~$\G$-conjugacy class of~$(\M_c,\ft_{\M_c})$ is called \emph{cuspidal endo-support~$\cusp(\endo(\pi))$} of any irreducible representation~$\pi$ of~$\Rep_\R(\ft)$. 

We call {the~$\G$-conjugacy class~$(\M,\fs)_\G$ of} a Levi subgroup~$\M$ of~$\G$ and an endo-parameter on~$\M$ a \emph{sub-endo-parameter} for~$\G$. There is a natural ordering on the set of sub-endo-parameters for~$\G$, see Definition~\ref{defSubEndopar}.  We collect the results of Section~\ref{secEndoSupport} in the following theorem:

\begin{theorem}\label{thmCuspidalEndosupport}
{
 Let~$\pi\in\Rep_{\R}(\ft)$. 
 \begin{enumerate}
 \item The cuspidal endo-support~$\cusp(\ft)$ of~$\pi$ is a maximal cuspidal sub-endo-parameter for~$\G$. 
 \item Every maximal cuspidal sub-endo-parameter for~$\G$ lies in the cuspidal endo-support of some irreducible~$\R$-representation of~$\G$. 
 \item Writing~$\cusp(\pi)$ for the inertial cuspidal support of~$\pi$, we have
 \[
 \endo(\cusp(\pi))\le \cusp(\endo(\pi)),
 \]
 where~$\endo(\cusp(\pi))$ is the sub-endo-parameter for~$\G$ determined by the endo-parameter of~$\cusp(\pi)$.
 \end{enumerate} 
}
\end{theorem}

%%%%%%% beta extensions for the strong simplicial structure 
\subsection{beta extensions for the strong simplicial structure}\label{secBetaStrong}
{An important step in the construction of cuspidal representations and types, are the beta extensions associated to a semisimple character parametrized by~$(\Lambda,\beta)$.  In earlier works,~\cite{St08} and~\cite{SkodlerackCuspQuart}, beta }extensions have been constructed for classical groups using the weak simplicial structure of the building~$\mathfrak{B}(\G_\b)$. But to unify results for types with those in depth zero for reductive groups 
it is important to establish the theory of~$\b$-extensions using the strong simplicial structure of~$\mathfrak{B}(\G_\b)$. For a given semisimple character~$\theta$ of~$\G$, parametrized by~$(\Lambda,\b)$, a~$\b$-extension~$\kappa$ for~$\theta$ is an extension of the Heisenberg representation
of~$\theta$.  Not every extension can be chosen:
\begin{itemize}
\item For constructing a type,~$\kappa$ needs to have large intertwining, for example if~$\Lambda$ comes from a strong vertex of~$\mathfrak{B}(\G_\b)$ one requires that the restriction of~$(\J,\kappa)$ to any pro-$p$-Sylow subgroup of its domain~$\J$ is intertwined by the whole centralizer~$\G_\b$. 
\item To be able to compare different types, a~$\b$-extension for~$\Lambda$ should be constructed compatibly to a~$\b$-extension of a lattice sequence~$\Lambda_0$ coming from a strong vertex of~$\mathfrak{B}(\G_\b)$ in a chamber containing~$j_\b^{-1}(\Lambda)$.
\end{itemize}
We verify the construction of~beta extensions in Section~\ref{subsecBeta} and we recall the construction of types for Bernstein components, {showing these types can be constructed with beta extensions defined using the strong simplicial structure}.  

%%%%%%%%%%%%%%%% Structure and remarks on the proof
\subsection{Structure of the paper} We now explain the structure of the paper and we give remarks on the proof of Theorem~\ref{thmExhaustionEndofactor}. In Section~\ref{secPreliminaries} we mainly introduce the description of~$\G$ and its parahoric subgroups. In Section~\ref{secEPs}
we recall the notion of semisimple characters and we recall the notion of endo-parameter, i.e., the invariant for the intertwining class of a semisimple character. In Section~\ref{Section:m-realizations} we introduce~$m$-realizations of an endo-parameter~$\ft$ and we introduce an equivalence relation 
called~\emph{essential-$\G$-conjugacy} on the set of~$m$-realizations of~$\ft$ and  we prove that the number of essential-$\G$-conjugacy classes of~$m$-realizations of~$\ft$ is finite, see Theorem~\ref{thmEssentiallyConjugateClasses}. This uses the numerical rigidity of endo-parameters, related to the classification of pearl embeddings by Broussous--Grabitz~\cite{BroussousGrabitz}. In Section~\ref{secESing} we define endo-factors and we prove the finiteness theorem {(Theorem~\ref{thmExhaustionEndofactor})}. {The main idea is that,} starting with a semisimple character of class~$\ft$ in an irreducible representation~$\pi$, we go along a path in~$\mathfrak{B}(\G_\b)$ and use compatibility to show that~$\pi$ contains a transfer of~$\theta$ to a strong vertex. 
After introducing beta extensions for the strong structure and recalling results on types for Bernstein blocks in Section~\ref{secBetaStrong}, we introduce the concept of cuspidal endo-support in Section~\ref{secEndoSupport}.  

%%%%%%%%%%%%%%%%%%%%%%%%%%%%%
\subsection{Acknowledgements}
The first author was partially supported by EPSRC New Horizons grant EP/V018744/1. The second author was supported by EPSRC grant EP/V001930/1 and the Heilbronn Institute for Mathematical Research. The third author was supported by a Shanghai 2021 `Science and Technology Innovation Action Plan' Natural Science Foundation Project grant.  The fourth author was supported by EPSRC grants EP/H00534X/1 and EP/V061739/1.  %We thank Jean-Fran\c cois Dat, Johannes Girsch, Thomas Lanard, Vanessa Miemietz, Peter Schneider, and Vincent S\'echerre for useful conversations.

%%%%%%%%%%%%%%%%%%%%%%%%%%%%%
\section{Preliminaries}\label{secPreliminaries}
%%%%%%%%%%%%%%%%%%%%%%%%%%%%%

%%%%%%%%%%%%%%%%%%%%%%%%%%%%%
\subsection{Notation}
For a non-archimedean local skew-field~$\D$  we write~$\o_\D$ for the ring of integers of~$\D$,~$\p_\D$ for the unique maximal ideal of~$\o_\D$, and~$\kk_\D$ for the residue field~$\o_\D/\p_\D$.  

%%%%%%%%%%%%%%%%%%%%%%%%%%%%%
\subsection{Smooth~$\R$-representations}\label{abstractsmoothreps}
Let~$\R$ be an algebraically closed field of characteristic~$\ell$ and~$\H$ a locally profinite group.  We call a smooth representation of~$\H$ on an~$\R$-module an \emph{$\R$-representation}.  The basic theory of~$\R$-representations is developed in Vign\'eras' book \cite{Vig96}.  In particular,~$\R$-representations of~$\H$ form an abelian category, we denote by~$\Rep_{\R}(\H)$, \cite[I 4.2]{Vig96}; and for~$\R$-representations~$\pi_1,\pi_2$ of~$\H$ we write~$\Hom_{\R[\H]}(\pi_1,\pi_2)$ for the~$\R$-module of morphisms~$\pi_1\rightarrow\pi_2$ in this category.  We call an~$\R$-representation \emph{irreducible} if its only~$\R$-subrepresentations are~$0$ and itself; in particular,~$0$ is \emph{not} considered an irreducible~$\R$-representation.  A simple application of Zorn's lemma shows that every non-zero finitely generated~$\R$-representation of~$\H$ has an irreducible quotient.

%%%%%%%%%%%%%%%%%%%%%%%%%%%%%
\subsection{Forms of classical groups}\label{subsecClassicalGroups}
We fix~$\F/\F_\so$, a Galois extension of non-archimedean local fields of degree one or two with residual characteristic~$p$ different from~$l$, and~$\varepsilon\in\{\pm,0\}$, such that
\newcommand{\caseGL}{\textbf{(\I)}}
\newcommand{\caseClassical}{\textbf{(\I\I)}}
\begin{itemize}
 \item[\caseGL]~$\F=\F_\so$ if~$\varepsilon=0$ and
 \item[\caseClassical]~$2\not\mid p$ if~$\varepsilon\neq 0$.
\end{itemize}
We write~$\overline{\phantom{ll}}$ for the generator of the of the Galois group of~$\F/\F_\so$. 

We denote by~$\Div(\F/\F_\so,\varepsilon)$ the set of pairs~$(\D,(\overline{\phantom{ll}})_\D)$ consisting of a skew-field~$\D$ of finite degree~$d$ with center~$\F$ and an~$\F_\so$-linear endomorphism~$(\overline{\phantom{ll}})_\D$ of~$\D$ extending~$\overline{\phantom{ll}}$ such that
\begin{itemize}
 \item $(\overline{\phantom{ll}})_\D$ is~$\id_\D$ in Case~\caseGL~and
 \item $(\overline{\phantom{ll}})_\D$ is an orthogonal or unitary anti-involution on~$\D$ in Case~\caseClassical. Note that in this case~$\D$ has at most degree~$2$, and if~$\D$ has degree~$2$ then~$\F=\F_\so$. 
\end{itemize}
We will still write~$\overline{\phantom{ll}}$ for~$(\overline{\phantom{ll}})_\D$ if there is no cause of confusion. 

A~\emph{Hermitian space for~$(\F/\F_\so,\varepsilon)$} is a pair~$(\V,h)$ together with a pair~$(\D,(\overline{\phantom{ll}})_\D)\in\Div(\F/\F_\so,\varepsilon)$ such that~$\V$ is a finite dimensional right~$\D$-vector space and~$h$ is an~$\varepsilon$-hermitian form~$h:\V\times\V\rightarrow \D$ with respect to~$\overline{\phantom{ll}}$, in particular~$h$ is the zero map in Case~\caseGL~and non-degenerate in Case~\caseClassical.
We write~$\Herm(\F/\F_\so,\varepsilon)$ for the set of Hermitian spaces for~$(\F/\F_\so,\varepsilon)$. 

We fix a pair~$(\D,(\overline{\phantom{ll}})_\D)\in\Div(\F/\F_\so,\varepsilon)$ and a Hermitian space~$(\V,h)$ for~$(\F/\F_\so,\varepsilon)$. 

We consider the following subgroups of~$\tG:=\GL_\D(\V)$:
the group
\[\U(\V,h)=\{g\in\GL_\D(\V):h(gv,gw)=h(v,w)\text{ for all }v,w\in\V\}\]
of isometries of~$(\V,h)$.  
and
\[\G:=\left\{\begin{array}{ll}
\U(\V,h)\cap\SL_\F(\V) & \text{if}\ h\neq 0\ \text{and}\ \D=\F=\F_\so\ \text{and}\ \varepsilon=+\\
\U(\V,h)& \text{otherwise.}
           \end{array}\right.
\]
 
Therefore~$\G$ will be the set of rational points of an~$\F_\so$-form of a general linear, symplectic, or special orthogonal group. 
Note that in the Case~\caseClassical~if~$\D\neq\F$ then every element of~$\U(\V,h)$ has reduced norm~$1$ over~$\F$. 
 
We let~$\Sigma=\langle \sigma\rangle$ denote an abstract cyclic group of order~$2$, which we will let act on various objects. The element~$\sigma$  acts trivially on~$\tG$ if~$h=0$, and as the inverse of the adjoint anti-involution of~$h$ if~$h$ is non-zero.  In particular,
\[
\U(\V,h)=\tG^\Sigma.
\]
We denote the set of~$\D$-endomorphisms of~$\V$ by~$\A$. The square root of the~$\F$-dimension of~$\A$ is called the~\emph{$\F$-degree} 
(or for short just the degree) of~$\A$, denoted by~$\deg_\F(\A)$.

%%%%%%%%%%%%%%%%%%%%%%%%%%%%%
\subsection{Parahoric subgroups for forms of classical groups}\label{parahoricssec}
For inner forms of (products of) classical $p$-adic groups and general linear groups, we use the (self-dual)~$\o_{\F}$-lattice function model of the Bruhat--Tits building and use the following notation for parahoric subgroups:  for~$\Lambda$ an~$\o_{\F}$-lattice function in~$\V$ we write~$\P(\Lambda)$ for the compact open subgroup stabilizing~$\Lambda$, and write~$\P(\Lambda)^{\circ}$ for the parahoric subgroup associated to~$\Lambda$ with pro-$p$ unipotent radical~$\P_1(\Lambda)$. We write~$\P^{\st}(\Lambda)$ for the fixator of the minimal facet~$\overline{\Lambda}$ containing~$\Lambda$, with respect to the strong simplicial structure.  

{Associated to an~$\o_{\F}$-lattice function in~$\V$, we also have the hereditary~$\o_{\F}$-order~$\mathfrak{a}_0(\Lambda)$ in~$\A$, with the filtration~$(\mathfrak{a}_n)_{n\in\mathbb Z}$ of~$\A$ by~$\o_{\F}$-lattices, where~$\mathfrak{a}_1(\Lambda)$ is the Jacobson radical. Then~$\P^1(\Lambda)=(1+\mathfrak{a}_1(\Lambda))\cap\G$.}

%%%%%%%%%%%%%%%%%%%%%%%%%%%%%
\section{Endo-parameters}\label{secEPs}
%%%%%%%%%%%%%%%%%%%%%%%%%%%%%

We recall the theory of self-dual semisimple characters, endo-equivalence, and endo-parameters.  We adopt slightly different notation and terminology to \cite{KSS} and \cite{SkodInnerFormI}, because we wish to later consider inner forms of $p$-adic classical groups and general linear groups simultaneously and we can by-pass some of the technicalities as we will not need the full theory (we only use the special case of \emph{full} semisimple characters).  

%%%%%%%%%%%%%%%%%%%%%%%%%%%%%
\subsection{Semisimple characters and intertwining for~$\tG$, following~\cite{KSS} and~\cite{SkodInnerFormI}}\label{semisimplechartG}
We fix an algebraic closure~$\overline{\F}$ of~$\F$ {and an additive character~$\psi_{\F}$ of conductor~$\mathfrak{p}_{\F}$.}

\begin{definition}
Let~$\beta=\sum_{i\in \I}\beta_i$ be a finite sum with~$\beta_i\in\overline{\F}$ such that the minimal polynomials  of~$\b_i$ and~$\b_j$ over~$\F$ differ for~$i\neq j$.  We call~$\beta$ \emph{full semisimple} if it satisfies a technical condition: its \emph{critical exponent}~$k_\F(\beta)<0$ is negative (for the definition of~$k_\F(\beta)$ in this generality see \cite[\S{8.1}]{KSS}).   
\end{definition}

Let~$\beta=\sum_{i\in \I}\beta_i$ be a full semisimple element. We set~$\E=\F[\b]$, a semisimple~$\F$-algebra,~$\E=\bigoplus_{i\in \I}\E_i$ with~$\E_i=\F[\beta_i]$ fields.  Let~$\mathscr{Q}(\beta)$ denote the class of all triples~$(\V,\varphi,\Lambda)$ where
\begin{enumerate}
\item $\V$ is a right~$\D$-vector space;
\item $\varphi:\E\hookrightarrow \A$ is an~$\F$-embedding;
\item $\Lambda$ is a~${\o_{\varphi(\E)}}$-$\o_\D$-lattice sequence.  
\end{enumerate}

Following the approach of Bushnell, Henniart, and Kutzko, as in \cite{KSS}, we associate to~$(\V,\varphi,\Lambda)\in \mathscr{Q}(\beta)$ a compact open pro-$p$ subgroup~$\widetilde{\H}^{1}(\varphi(\beta),\Lambda)$ of~$\tG=\Aut_\D(\V)$ and a set~$\Cc(\Lambda,\varphi(\beta))$ of characters of~$\widetilde{\H}^{1}(\varphi(\beta),\Lambda)$ which we call \emph{semisimple characters for}~$\tG$.   Note that, this is a special case of the construction; these characters are called \emph{full semisimple characters} in \cite{KSS} and~\cite{SkodInnerFormI}.   

We let~$\Cc(\beta)=\bigcup_{(\V,\varphi,\Lambda)\in \mathscr{Q}(\beta)}\Cc(\Lambda,\varphi(\beta))$ and, for fixed~$\V$, we put~$\Cc(\beta,\V)=\bigcup_{(\V,\varphi,\Lambda)\in \mathscr{Q}(\beta)}\Cc(\Lambda,\varphi(\beta))$. Then the collection of all semisimple characters for~$\tG=\Aut_\D(\V)$ is~$\Cc(\V)=\bigcup_{\beta}\Cc(\beta{,\V})$, where the union is over all full semisimple elements~$\beta$.

\begin{remark}
In~\cite[Definition 6.6]{SkodInnerFormI}, the third author defines the endo-equivalence class~$\mathfrak{E}$ of a semisimple \emph{stratum}. If one of the strata in~$\mathfrak{E}$ is of the form~$[\Lambda,n,0,\varphi(\beta)]$, for~$(\V,\varphi,\Lambda)\in\Qq(\b)$ then for every stratum~$\Delta=[\Lambda',n,0,\gamma]\in\mathfrak{E}$ there is a tuple~$(\V',\varphi',\Lambda')\in\Qq(\beta)$ such that~$\Delta$ is equivalent (as strata) to~$[\Lambda',n,0,\varphi'(\beta)]$, by~\cite[Proposition 4.30]{SkodInnerFormI} and~\cite[Theorem 6.6]{SkodInnerFormII}. 
\end{remark}

Corresponding to the decomposition~$\beta=\sum_{i\in \I}\beta_i$, we have:
\begin{enumerate}[$\bullet$]
\item a decomposition~$\V=\bigoplus_{i\in\I} \V_i$ (called the \emph{associated splitting} of~$\vphi(\b)$) 
\item a decomposition~$\Lambda=\bigoplus_{i\in\I}\Lambda_i$, where~$\Lambda_i(k)=\Lambda(k)\cap\V_i$ (we say that the decomposition of~$\V$ \emph{splits}~$\Lambda$);
\item endomorphisms rings~$\A_i:=\End_\D(\V_i)$ and~$\B_i=\End_{\E_i\otimes\D}(\V_i)$;
\item natural embeddings~$\widetilde{\H}^1(\varphi(\beta_i),\Lambda_i)\hookrightarrow \widetilde{\H}^1(\varphi(\beta),\Lambda)$, and restriction maps
\begin{align*}
\Cc(\Lambda,\varphi(\beta))&\rightarrow \Cc(\Lambda_i,\varphi(\beta_i)),\quad \theta\mapsto \theta_i:=\theta\mid_{\widetilde{\H}^1(\varphi(\beta_i),\Lambda_i)}.
\end{align*}
\end{enumerate}
We call the~$\theta_i$ the \emph{simple block restrictions} of~$\theta$.  

Given~$(\V,\varphi,\Lambda),(\V',\varphi',\Lambda')\in \mathscr{Q}(\beta)$, there are natural bijections
\[\tau_{\Lambda',\Lambda,\varphi',\varphi}:\Cc(\Lambda,\varphi(\beta))\rightarrow \Cc(\Lambda',\varphi'(\beta)),\]
called \emph{transfer maps} (\cite[Lemma 9.3]{KSS} and~\cite[\S6.2]{SkodInnerFormI}), and we collect semisimple characters into families following \cite{BH96,KSS}:

\begin{definition}
A \emph{pss-character supported on}~$\beta$ is a function~$\Theta:\mathscr{Q}(\beta)\rightarrow \Cc(\beta)$ whose values are related by transfer:
\[\Theta(\V',\varphi',\Lambda')=\tau_{\Lambda',\Lambda,\varphi',\varphi}\Theta(\V,\varphi,\Lambda).\]
We call a value of~$\Theta$ a \emph{realization}; thus, by definition,~$\Theta$ is determined by any one of its realizations.
\end{definition}

Let~$\beta'$ be another full semisimple element and set~$\E'=\F[\b']$.  Let~$\Theta,\Theta'$ be pss-characters supported on~$\beta,\beta'$ respectively. 
 
\begin{definition}\label{def:endo-e}
We say that~$\Theta$ and~$\Theta'$ are \emph{endo-equivalent} if there exist realizations of~$\Theta$ and~$\Theta'$ on a common~$\F$-vector space~$\V$ which intertwine in~$\tG$.
\end{definition} 

By~\cite[Theorem~9.9]{KSS} and~\cite[Theorem~6.18]{SkodInnerFormI}, endo-equivalence is an equivalence relation and we call the equivalence classes \emph{semisimple endo-classes}. If~$\theta$ is a realization of a pss-character~$\Theta$ then we define the endo-class of~$\theta$ to be the endo-class of~$\Theta$. The \emph{degree} of a semisimple endo-class is~$[\F[\b]:\F]$ where~$\b$ is any full semisimple element which supports a pss-character of this endo-class; this is well-defined by~\cite[Proposition~6.2]{KSS} because we are dealing only with full semisimple characters -- it is for this reason that we don't need to assume the degrees of the pss-characters are equal in Definition~\ref{def:endo-e}.

We call a semisimple character or endo-class \emph{simple} if it is defined by a semisimple element~$\beta$ which generates a field extension~$\F[\b]/\F$ (i.e.~$\I$ is a singleton in the notation above), by~\cite[Theorem 9.9(i)]{KSS} this is well defined. Write~$\cE(\F)$ for the set of all simple endo-classes of simple characters for inner forms of general linear groups over~$\F$. 

\begin{definition}\label{defEndoparGL}
An \emph{endo-parameter~$\ft$ for~$\tG$} is a formal sum~$\ft=\sum_{c\in\cE(\F)}m_{c}c$, with~$m_{c}\in\mathbb{Z}^{\geqslant 0}$, satisfying~$\sum_{c\in\cE(\F)}m_{c}\deg(c)=\deg_\F(\A)$ and~$d\mid m_c\deg(c)$, for~$c\in\cE(\F)$.
\end{definition}

By \cite[Theorem 12.9]{KSS} and~\cite[Theorem 7.2]{SkodInnerFormI}, the set of intertwining classes of semisimple characters for~$\tG$ is in canonical bijection with the set of  endo-parameters for~$\tG$.  The bijection is given as follows: let~$\theta$ be a semisimple character for~$\tG$, let~$\theta_i$ be the simple block restrictions of~$\theta$, and let~$c_i$ be the endo-classes of~$\theta_i$.  Then we map the intertwining class of~$\theta$ to the endo-parameter~$\ft_{\theta}$ defined by~$\ft_{\theta}=\sum_{i\in\I}\deg_{\E_i}(\B_i)c_i$.  

%%%%%%%%%%%%%%%%%%%%%%%%%%%%%
\subsection{Semisimple characters and intertwining for~$\G$, following~\cite[\S8]{KSS} and~\cite[\S6.1]{SkodInnerFormII}
}\label{semisimplecharG}

We now expand the scope of the last section to include inner forms of classical groups and define endo-parameters for~$h$.  In contrast to previous works, we include semisimple characters for inner forms of general linear groups in this framework as the special case where~$h=0$ (i.e., Case~\caseGL), to allow us to state our results in future sections uniformly. 

Recall that we have fixed~$\F/\F_\so$ and~$\e\in\{\pm,0\}$. {We also assume that the additive character~$\psi_{\F}$ is~$\sigma$-invariant.}
\begin{definition}
We say that a full semisimple element~$\beta$ is \emph{$\e$-self-dual}, or just~\emph{self-dual}, if the generator~$\overline{\phantom{ll}}$ of~$\Gal(\F/\F_\so)$ extends to an involution~$\overline{\phantom{ll}}$ on~$\F[\beta]$ such that~$\overline{\beta}=(-1)^\e\beta$.  Note that, when~$\e=0$ this imposes no extra condition on a full semisimple element.
\end{definition}
Let~$\beta$ be an~$\e$-self-dual full semisimple element.  Then the involution~$\overline{\phantom{ll}}$ induces an action of~$\sigma$ on the indexing set~$\I$ of~$\beta=\sum_{i\in\I}\beta_i$, which decomposes as
\[\I=\I_+\cup\I^\sigma\cup\I_-\]
with~$\I^\sigma$ the~$\sigma$-fixed orbits,~$\I_+$ a set of representatives for the the orbits of size~$2$, and~$\I_-=\sigma(\I_+)$.  We write~$\E=\F[\beta]$ and~$\E_\so$ for the set of~$\overline{\phantom{ll}}$-fixed points on~$\E$.

In this case, as in~\cite{KSS} or similarly in~\cite[\S7.1]{SkodInnerFormII}, we let~$\mathscr{Q}_{\F/\F_\so,\e}(\beta)$ denote the class of all triples~$((\V,h),\varphi,\Lambda)$ where
\begin{enumerate}[$\bullet$]
\item $(\V,h)$ is an~$\e$-hermitian space for~$(\F/\F_\so,\e)$,
\item $(\V,\varphi,\Lambda)\in \mathscr{Q}(\beta)$, 
\item and~$\varphi,\Lambda$ are $h$-\emph{self-dual}, see~\cite[\S 7.2]{KSS} and~\cite[Definition 4.1]{SkodInnerFormII}. 
\end{enumerate}
Note that this does not give any extra condition on~$\varphi,\Lambda$ if~$h$ is zero so that~$\mathscr{Q}_{\F/\F_\so,0}(\beta)=\mathscr{Q}(\beta)$. 

If~$((\V,h),\varphi,\Lambda)\in\mathscr{Q}_{\F/\F_\so,\e}(\beta)$, then~$\widetilde{\H}^1(\varphi(\beta),\Lambda)$ is~$\Sigma$-stable and, with~$\G=\G_h$ the associated classical group, we write
\[
\H^1(\varphi(\beta),\Lambda)=\widetilde{\H}^1(\varphi(\beta),\Lambda)\cap \G.
 \]
 When~$h=0$, we have~$\H^1(\varphi(\beta),\Lambda)=\widetilde{\H}^1(\varphi(\beta),\Lambda)$.
 
The group~$\Sigma$ acts on~$\mathscr{C}(\Lambda,\varphi(\beta))$ with fixed points~$\mathscr{C}(\Lambda,\varphi(\beta))^\Sigma$ -- the set of \emph{self-dual semisimple characters} -- and we define the set of characters~$\mathscr{C}_h(\Lambda,\varphi(\beta))$ of~$\H^1(\varphi(\beta),\Lambda)$ by restriction from~$\mathscr{C}(\Lambda,\varphi(\beta))^\Sigma$.  
A self-dual semisimple character is called an~\emph{elementary character}
if the set~$\I^\sigma\cup \I_+$ has cardinality one. 

For fixed~$(\V,h)$ and associated classical group~$\G$, we write:
\begin{itemize}
 \item $\Cc(h)$ for the union over all $h$-self dual full semisimple elements~$\beta$, of all the sets~$\Cc(\Lambda,\varphi(\b))^\Sigma$ such that 
 $((\V,h),\varphi,\Lambda)$ is an element of~$\mathscr{Q}_{\F/\F_\so,\e}(\beta)$; 
 \item $\Cc_-(h)$ for the union over all $h$-self dual full semisimple elements~$\beta$, of all the sets~$\Cc_h(\Lambda,\varphi(\b))$ such that 
 $((\V,h),\varphi,\Lambda)$ is an element of~$\mathscr{Q}_{\F/\F_\so,\e}(\beta)$. 
\end{itemize}
We call the elements of~$\Cc_-(h)$ \emph{semisimple characters for~$\G$} (this depends only on the group~$\G$, not on~$h$). It is useful to have the notion of a \emph{parametrization} of a semisimple character: 

\begin{definition}\label{defParametrization}
 For a character~${\t}\in\Cc_h(\Lambda,\varphi(\b))$ we call the data~$((\V,h),\varphi,\Lambda,\b)$ a~\emph{parametrization}  of ${\t}$. 
 Having chosen  a parametrization~$((\V,h),\varphi,\Lambda,\b)$ for~${\t}$ we define the restrictions of~${\t}$ by
 \[({\t})_i:=\left\{\begin{array}{ll}
                     {\t}|_{\H^1(\varphi(\b),\Lambda)\cap\Aut_\D(\V_i)} &,i\in\I^\sigma \\
                     {\t}|_{\H^1(\varphi(\b),\Lambda)\cap(\Aut_\D(\V_i)\times\Aut_\D(\V_{\sigma(i)}))} &,i\in\I_+\cup\I_-\\
                    \end{array}\right.
\]
To~$\varphi(\b)$ is attached the following Levi subgroup of~$\G$. 
\[
\M(\varphi(\b),\G):=\left\{\begin{array}{ll}
                               \G\cap(\Aut_\D(\sum_{i\in\I^\sigma}\V_i)\times\prod_{i\in \I_+\cup\I_-}\Aut_\D(\V_i))&,\text{ if~$h$ is non-zero}\\
                               \prod_{i\in\I}\Aut_\D(\V_i),\text{ if~$h$ is zero}
                              \end{array}
                               \right.
\]
For a Levi subgroup~$\M$ of~$\G$ we write~${\t_{\M}}$ for the restriction of~${\t}$ to $\H^1(\varphi(\b),\Lambda)\cap\M$.
\end{definition}

We also need a finer Levi attached to a parametrization~$((\V,h),\varphi,\Lambda,\b)$. This Levi is related to the cuspidal endo-support of an irreducible representation, which we will define in Section~\ref{secEndoSupport}. 
We define~$\M_c(\varphi(\b),\G)$ to be~$\M(\varphi(\b),\G)$ except for the case where
\begin{itemize}
\item $\G$ is a special orthogonal group, i.e.~$\varepsilon=1,\ \D=\F=\F_\so$ and 
\item there exists an index $i_0\in \I^\sigma $ such that~$\beta_{i_0}=0$ and~$\dim_\F\V_{i_0}=2$ and~$h|_{\V_{i_0}}$ has Witt index~$1$.   
\end{itemize} 
In the latter case we put
\[\M_c(\varphi(\b),\G):=\G\cap(\SAut_\F(\sum_{i\in\I^\sigma,\ i\neq i_0}\V_i)\times(\prod_{i\in \I_+\cup\I_-}\Aut_\F(\V_i))\times\SAut_\F(V_{i_0})),\]
i.e., we remove the~$\V_{i_0}$-part from the classical factor of~$\M(\varphi(\b),\G)$. (Note that when such an index~$i_0$ exists then it is unique.)
{
\begin{remark} 
We will not need this, but in fact~$\M_c(\varphi(\b),\G)$ is the minimal (rational) Levi subgroup containing~$\G_{\varphi(\beta)}$.
\end{remark}
}
Recall also that, if two semisimple characters~${\t},{\t'}\in\Cc_-(h)$ with parametrizations~$((\V,h),\varphi,\Lambda,\b)$ and~$((\V,h),\varphi',\Lambda',\b')$ intertwine in~$\G$ then there is a canonical bijection~$\zeta:\I\rightarrow\I'$, called a \emph{matching} (see~\cite[Theorem 10.1]{SkSt}, \cite[Theorem 8.8]{KSS}). Moreover, from~\cite[Theorem~6.5,~Corollary~6.15]{SkodInnerFormII} we have:

\begin{lemma}\label{lemMphib} 
Let~${\t},{\t'}\in\Cc_-(h)$ be semisimple characters with parametrizations~$((\V,h),\varphi,\Lambda,\b)$ and~$((\V,h),\varphi',\Lambda',\b')$. If~${\t},{\t'}$ intertwine in~$\G$ then~$\M(\varphi(\b),\G)$ is conjugate in~$\G$ to~$\M(\varphi'(\b'),\G)$, and 
~$\M_c(\varphi(\b),\G)$ is conjugate in~$\G$ to~$\M_c(\varphi'(\b'),\G)$.
\end{lemma}

%%%%%%%%%%%%%%%%%%%%%%%%%%%%%
\subsection{Endo-parameters for~$h$}
We let the group~$\Sigma$ act on~$\cE(\F)$ by the action defined in \cite[Definition 12.13]{KSS}, and we write~$(\cE(\F)/\Sigma)$ for the set of orbits. 

More precisely, we define the~$\Sigma$-action on the endo-class of a ps-character~$\Th$ by using a realization of~$\Th$ in a split general linear group and then applying~\cite[Definition 12.10]{KSS}. An orbit~$t$ in~$(\cE(\F)/\Sigma)$ corresponds to the endo-class of an elementary character~$\theta$. The orbit~$t$ is the set of simple endo-classes of the simple block restrictions of~$\theta$.  We say~$t$ is attached to~$\theta$. 

To define endo-parameters for~$\G$ we need to introduce an extra datum.  For~$t\in(\cE(\F)/\Sigma)$ we associate a set~$\mathcal{W}_{(\overline{\phantom{ll}})_\D,\e}(t)$ of \emph{Witt types} which, if~$t$ is an orbit of size $2$ or if~$\e=0$, identifies with a singleton.  When~$\e\neq 0$ and $t$ is an orbit of size~$1$, say consisting of the endo-class of~$\Theta$, then we can choose~$\Theta$ to be supported  on a full self-dual simple element~$\beta$ and then~$\mathcal{W}_{(\overline{\phantom{ll}})_\D,\e}(t)$ is in canonical bijection to the Witt group~$\W_{\e}((\overline{\phantom{ll}})_{\E}\otimes_\F(\overline{\phantom{ll}})_\D)$, see~\cite[\S7.2, Proposition 7.4, Definition 7.6]{SkodInnerFormII}. This Witt-group is (non-canonically) isomorphic to the Witt group~$\W_\e(\E/\E_\so)$ if both~$\D$ and~$\E$ are different from~$\F$. Otherwise, this Witt group is canonically isomorphic to~$\W_\e(\E/\E_\so)$ (if~$\D=\F$) or to~$\W_{\e}((\overline{\phantom{ll}})_\D)$ (if~$\E=\F$).

This definition was introduced for~$\D=\F$ in~\cite[\S13 after Remark 12.20]{KSS}, see~$\W\T(\mathcal{O})$ there. For~$\D\neq\F$, see~\cite[\S7.4]{SkodInnerFormII} and the definition of~$\mathcal{W}_{c_-}$ there. But there is a subtlety to take into account: In~\cite[\S7.4]{SkodInnerFormII} the third author attaches to a \emph{set of Witt-types} to an endo-class of an \emph{elementary} character. Thus we need to prove that every orbit in~$\cE(\F)/\Sigma$ is attached  to an elementary character for an~$\e$-Hermitian form with respect to the fixed datum~$(\D,(\overline{\phantom{ll}})_\D)$. We give the proof and the precise statement in Appendix~\ref{propAppDRealization}. 

 There is a~\emph{transfer map}~$\lambda_t^*$, from~$\mathcal{W}_{(\overline{\phantom{ll}})_\D,\e}(t)$ to~$\mathcal{W}_{\e}(\F/\F_\so)$ (see~\cite[\S~3.5]{KSS} and~\cite[\S~7.2]{SkodInnerFormII})

\begin{definition}\label{defEndoparG}
\label{def11}  
An \emph{endo-parameter~$\ft$ for~$h$} is a formal sum~$\ft=\sum_{t\in(\cE(\F)/\Sigma)}m_{t}(t,w_t)$, where $m_{t}\in\mathbb{Z}^{\geqslant 0}$ and~$w_t\in \mathcal{W}_{(\overline{\phantom{ll}})_\D,\e}(t)$, satisfying~
\begin{align*}
\label{endo1}
\tag{1}
&  
d \mid m_t\deg(c),\ \text{ for all }c\in t\in (\cE(\F)/\Sigma),\\
\tag{2}\label{endo2}
&\sum_{\substack{c\in t\in (\cE(\F)/\Sigma)}}n_t\deg(c)
=\deg_\F(\A)\ with\ n_t:=\frac{2}{|t|}m_t+\deg_{an}(w_t),\\
&\ \hskip2cm\text{ where }\deg_{an}(w_t)\text{ is the  anisotropic degree of }w_t,\\
\tag{3}
&\sum_{t\in(\cE(\F)/\Sigma)}\lambda_{t}^*(w_t)= [h]\text{ in }\mathcal{W}_{\e}(\F/\F_\so).
\end{align*}
\end{definition}

\begin{remarks}\label{remEndoparG}
\begin{enumerate}
\item\label{remEndoparGi}
We have the obvious forgetful map, from endo-parameters for~$h$ to  endo-parameters for~$\tilde{\G}$, given by forgetting Witt types: 
\[
\sum_{t\in(\cE(\F)/\Sigma)}m_{t}(t,w_t)\mapsto\sum_{c\in t\in(\cE(\F)/\Sigma)}n_{t}c,
\]
with~$n_t$ given in Definition~\ref{defEndoparG}\eqref{endo2}.
Endo-parameters for~$\tilde{\G}$ obtained in this way are called~\emph{unrefined endo-parameters for~$h$.} 

\item Suppose that~$\G$ is not special orthogonal or~$\D\neq\F$. Then an endo-parameter for~$(\G,h)$ is defined to be an endo-parameter for~$h$.  
\item Suppose that~$\G$ is special orthogonal and~$\D=\F$. An endo-parameter for~$(\G,h)$ is a pair, consisting of an endo-parameter defined as in \ref{def11} (an endo-parameter for ~$h$) and possibly an additional sign - see \cite[Definition 12.33]{KSS}.
\item In the case~$\e=0$, an endo-parameter for~$\G$ is the endo-parameter given in Definition~\ref{defEndoparGL}, because the Witt type~$w_t$ is just a dummy variable with no further information. 
\end{enumerate}
\end{remarks}

For~$\e\neq 0$, by \cite[Theorem 12.29 and Corollary 12.34]{KSS} and~\cite[Theorem 7.17]{SkodInnerFormII}, the set of intertwining classes of semisimple characters for~$\G$ is in canonical bijection with the set of endo-parameters for~$(\G,h)$ (see ibid.~ for the description of this map; this bijection depends on the hermitian form~$h$, not only on the group~$\G$). 
On the other hand, when~$\e=0$, by the references after Definition \ref{defEndoparGL}, the set of intertwining classes of semisimple characters for~$\G$ is in canonical bijection with the set of endo-parameters for~$\G$.

It now follows from Lemma~\ref{lemMphib} that an endo-parameter~$\ft$ determines a~$\G$-conjugacy class of Levi subgroups, namely the class containing~$\M(\vphi(\b),\G)$ for any parametrization~$((\V,h),\varphi,\Lambda,\b)$ of any semisimple character~${\t}$ with endo-parameter~$\ft$. We write~$\M(\ft)$ for this conjugacy class. Similarly we denote by~$\M_c(\ft)$ the~$\G$-conjugacy class of~$\M_c(\vphi(\b),\G)$.

%%%%%%%%%%%%%%%%%%%%%%%%
\subsection{Parabolic induction and restriction maps for endo-parameters}\label{SeccuspsuppEP}

Let~$\P$ be a Levi subgroup of~$\G$ and~$\P=\M\ltimes \N$ a Levi decomposition of~$\P$.  We have non-normalized parabolic induction and restriction functors, we denote by~$i_\P^\G:\Rep_{\R}(\M)\rightarrow \Rep_{\R}(\G)$ and~$r^\G_\P:\Rep_\R(\G)\rightarrow \Rep_\R(\M)$ respectively.  Here we define maps of endo-parameters, which we will later show are compatible with parabolic induction and restriction.

We write~$\M=\prod_{j=0}^s \M_j$, with~$\M_0$ a classical group (or trivial), and~$\M_j$ general linear groups. 
\begin{definition}
An \emph{endo-parameter}~$\ft$ for~$\M$ is a tuple~$(\ft_j)_{j=0}^s$ of endo-parameters~$\ft_j$ for~$\M_j$.
\end{definition}

Let~$\ft=(\ft_j)_{j=0}^s$ be an endo-parameter for~$\M$ and~choose realizations~${\theta_{j}}$ of~$\ft_{j}$. Then~${\theta_{\M}}=\bigotimes_{j=1}^s {\theta_{j}}$ is a semisimple character for~$\M$, and by~\cite[Proposition 5.1]{MiSt}, \cite[Theorems 6.6 and 6.10]{SkodInnerFormII}, 
we can choose a semisimple character~${\theta}$ for~$\G$ with~${\theta}\vert_{\M}={\theta_{\M}}$.  We let~$i_{\M}^\G(\ft)$ denote the endo-parameter for~$\G$ with realization~${\theta}$.  This is independent of the choice of~${\theta}$, as any two choices for~${\theta_{\M}}$ intertwine in~$\M$ (by definition) so, since the corresponding semisimple characters~${\theta}$ are decomposed with respect to~$(\M,\P)$ for any parabolic~$\P$ with Levi~$\M$, these realizations~${\theta}$ also intertwine so give the same endo-parameter~$i_{\M}^\G(\ft)$. Moreover, parabolic induction of endo-parameters is clearly transitive.

Conversely, let~$\ft$ now be an endo-parameter for~$\G$. We define~$r_{\M}^{\G}(\ft)$ by
\[
r_{\M}^{\G}(\ft)=\{\text{endo-parameters }\ft_\M\text{ for }\M:\ft=i_{\M}^\G(\ft_\M)\}.
\]
Set~$\W_\M=\N_\G(\M)/\M$. In general,~$r_{\M}^{\G}(\ft)$ will \emph{not} consist of a single~$\W_\M$-conjugacy class of endo-parameters for~$\M$, but is a (possibly empty) finite set of~$\W_\M$-conjugacy classes. 

There is one case of particular interest, namely when~$\M\in \M(\ft)$. If~${\theta}$ is any semisimple character with endo-parameter~$\ft$, and parametrization $((\V,h),\varphi,\Lambda,\b)$ such that~$\M=\M(\varphi(\b),\G)$, then the restriction~${\theta_{\M}}:={\theta}\vert_{\M}$ is a semisimple character for~$\M$, and we write~$\ft_\M$ for its corresponding endo-parameter for~$\M$.  The~$\G$-conjugacy class of this is independent of the choice of~${\theta}$ by~\cite[Theorem~6.5,~Corollary~6.15]{SkodInnerFormII}, and clearly~$\ft_\M\in r_{\M}^{\G}(\ft)$. We call the~$\G$-conjugacy class of the pair~$(\M,\ft_{\M})$ the \emph{support} of~$\ft$. Similarly we construct for~$\M_c:=\M_c(\varphi(\b),\G)$ the endo-parameter~$\ft_{\M_c}$ by restriction and we call the~$\G$-conjugacy class of~$(\M_c,\ft_{\M_c})$ the \emph{cuspidal support} of~$\ft$. 

%%%%%%%%%%%%%%%%%%%%%%%%%%%%%
\section{$m$-realizations of endo-parameters}\label{Section:m-realizations}
%%%%%%%%%%%%%%%%%%%%%%%%%%%%%

We fix~$(\G,h)$ as introduced in~\S\ref{subsecClassicalGroups}, and an endo-parameter $\ft$ for~$(\G,h)$.

\begin{definition}
A \emph{realization} of~$\ft$ for~$(\G,h)$ is a semisimple character for~$\G$ of endo-parameter~$\ft$. 
\end{definition}

Given a parametrization~$((\V,h),\varphi,\Lambda,\b)$ of a semisimple character~${\theta}$ of~$\G$, we write~$\G_{\varphi(\beta)}$ for the centralizer of~$\varphi(\beta)$ in~$\G$, and similarly~$\B_{\varphi(\beta)}$ for the centralizer of~$\varphi(\beta)$ in~$\A$, so that~$\B=\bigoplus_{i\in\I}\B_i$. There is then an embedding of Bruhat--Tits buildings
\begin{equation}\label{eqjb}
 j_{\varphi(\beta)}:\ \mathfrak{B}(\G_{\varphi(\b)})\hookrightarrow\mathfrak{B}(\G)
\end{equation}
which is affine,~$\G_{\varphi(\b)}$-equivariant, and respects the Moy--Prasad filtrations, see~\cite[Theorem 7.2]{SkodFourier} based on~\cite[Theorem II.1.1, Lemma II.3.1]{broussousLemaire:02} and~\cite[Theorem 6.3, \S 9]{broussousStevens:09}. Note that by Bruhat--Tits building we always mean the enlarged building. The image of~$j_{\varphi(\beta)}$ corresponds to the set of~$h$-self-dual~${\o_{\varphi(\E)}}$-$\o_\D$-lattice functions. 

{As in~\S\ref{parahoricssec},} to~$\Lambda$ corresponds a parahoric subgroup of~$\G_{\varphi(\beta)}$ which we denote by~$\P(\Lambda_{\varphi(\b)})^\circ$, with pro-$p$ unipotent radical~$\P_1(\Lambda_{\varphi(\b)})$, and we write~$\P(\Lambda_{\varphi(\b)})$ for the full fixator of the point~$\Lambda$.  

We compose~$j_{\varphi(\b)}$ with the canonical map~$\mathfrak{B}(\G)\rightarrow\mathfrak{B}_{red}(\G)$ (which sends a lattice function to its translation class) to obtain a map
\begin{equation}\label{eqjbbar}
 \bar{j}_{\varphi(\beta)}:\ \mathfrak{B}(\G_{\varphi(\b)})\rightarrow\mathfrak{B}_{red}(\G)
\end{equation}
 which is affine,~$\G_{\varphi(\b)}$-equivariant, and respects Moy--Prasad filtrations.  In the following cases, this map is uniquely determined by these three 
 properties:
 \begin{itemize}
  \item if $\I$ is a singleton and~$\e=0$, by \cite[Theorem II.1.1]{broussousLemaire:02}, and
  \item if~$\I=\I^\sigma$ and~$\e\neq 0$ and the center of~$\G_{\varphi(\b)}$ is compact, by~\cite[Theorem 11.3]{SkodFourier}.
 \end{itemize}

Of most importance are the realizations of~$\ft$ on a maximal parahoric {subgroup} of the centralizer of an embedded full semisimple element. 
We denote by~$\C(\H)$ the centre of a group~$\H$.

\begin{definition}\label{defmRealization}
An \emph{$m$-realization} of~$\ft$ for~$(\G,h)$ is a realization of~$\ft$ which is in a set of semisimple characters~$\Cc_h(\Lambda,\varphi(\beta))$ for~$\G$ such that~$\P(\Lambda_{\varphi(\b)})^\circ$ is a maximal parahoric subgroup of~$\G_{\varphi(\beta)}$.  
If further~$\C(\G_{\varphi(\beta)})/\C(\M(\varphi(\b),\G))$ is compact then the~$m$-realization is called a~\emph{cuspidal~$m$-realization} of~$\ft$
\end{definition}

The property of being {a (cuspidal)}~$m$-realization of~$\ft$ can be verified by any parametrization:
\begin{proposition}[cf.~{\cite[Propositions~11.3]{KSS}}]\label{propmRealization}
 Let~${\theta}$ be an~$m$-realization of~$\ft$ for~$(\G,h)$ and let $((\V,h),\varphi',\Lambda',\b')$ be a parametrization of~${\theta}$.
 Then~$(\G_{\varphi'(\b')})_{\Lambda'}^\circ$  is a maximal parahoric subgroup of~$\G_{\varphi'(\b')}$, and~$\C(\G_{\varphi'(\b')})/\C(\M(\varphi'(\b'),\G))$ is compact if~${\theta}$ is a cuspidal~$m$-realization of~$\ft$. 
\end{proposition}

The technical idea for the proposition lies in the following lemma. For~$\theta$ a semisimple character, we write~$\I_\G(\theta)$ for the set of elements of~$\G$ which intertwine~$\theta$ with itself.
\begin{lemma}\label{lemTransfer}
 Suppose~$h=0$ and let~$\theta\in\Cc(\Lambda,\varphi(\beta))$ and~$\theta'\in\Cc(\Lambda,\varphi'(\beta))$ be transfers and suppose that~$\F[\beta]$ is a field. Suppose further that~$1$ intertwines~$\theta$ with~$\theta'$. Then~$\theta$ and~$\theta'$ coincide. 
\end{lemma}

\begin{proof}
 By Skolem--Noether there is an element~$g\in\G=\tG$ such that~${^g\varphi}=\varphi'$. Then~$g$ intertwines~$\theta$ with~$\theta'$, because~$\theta$ and~$\theta'$ are transfers. Thus we have
 \[
 1=s'gbs,\ {\qquad\text{for some }}s'\in\I_\G(\theta')\cap\P_1(\Lambda),\ s\in\I_\G(\theta)\cap\P_1(\Lambda),\ b\in\G_{\varphi(\beta)},
 \]
 by \cite[Theorem 6.7(ii)]{SkodInnerFormII}. 
 On the residue fields, conjugation by the elements~$1$ and~$g$ give the same map
 \[
k_{\F[\varphi(\beta)]}\rightarrow k_{\F[\varphi'(\beta)]},
 \]
{by~\cite[Lemma~4.46]{SkodInnerFormI}.} Thus, by \cite[Proposition~4.39]{SkodInnerFormI}, we can assume without loss of generality that~$g$ is an element of~$\P(\Lambda)$, so that~$b\in\P(\Lambda)\cap\G_{\varphi(\beta)}$. In particular,~$g$ normalizes the lattice sequence~$\Lambda$ and conjugates~$\vphi(\b)$ to~$\vphi'(\b)$, so also conjugates~$\H^1(\vphi(\b),\Lambda)$ to~$\H^1(\vphi'(\b),\Lambda)$; it follows that~$\theta'= {^g\theta}$. Since~$b$ also normalizes~$\theta$, we get~$\theta= {^1\theta}= {^{s'gbs}\theta}= {^{s'}\theta'}=\theta'$.
\end{proof}

\begin{proof}[Proof of Proposition~\ref{propmRealization}]Let~$((\V,h),\varphi,\Lambda,\b)$ be a parametrization for~${\theta}$ given as in Definition~\ref{defmRealization}.
By~\cite[Corollary 5.17]{SkodInnerFormI} together with the proof of~\cite[Proposition 6.9]{SkodInnerFormII} there is an element in the normalizer of~${\theta}$ in~$\G$ which maps the associated splitting of~$\varphi(\beta)$ to the associated splitting of~$\varphi'(\beta')$. We therefore assume without loss of generality that~$\varphi(\beta)$ and~$\varphi'(\beta')$ have the same associated 
splitting and the matching is the identity. Thus we can reduce to the cases ($h=0$ and~$|\I|=1$) and ($h\neq 0$ and~$\I^\sigma=\I$). 

We start with the second case ($h\neq 0$ and~$\I^\sigma=\I$): If we are in the case that~$h$ is orthogonal giving~$\SO(1,1)(\F)$ (which is isomorphic to~$\GL_1(\F)=\F^\times$) then every semisimple character for~$\G$ is an~$m$-realization for its endo-parameter. 
So let us suppose that we are not in this particular case. Then, if the center of~$\G_{\varphi(\b)}$ is compact,  by~\cite[\S7]{St08} and~\cite[\S10.3]{SkodlerackCuspQuart} the center of~$\G_{\varphi'(\beta')}$ is compact and~$\P(\Lambda_{\varphi(\b)})^\circ$ is a maximal parahoric, because~${\theta}$ is contained in a cuspidal irreducible representation of~$\G$. If the center of~$\G_{\varphi(\b)}$ is not compact (that can only happen when~$\D=\F$) then~$\P(\Lambda_{\varphi(\b)})^\circ$ is a maximal parahoric by the proof of~\cite[Lemma 11.5(iv)]{KSS}.

We now prove the first case ($h=0$ and~$|\I|=1$): At first $\C(\G_{\varphi'(\b')})/\C(\M(\varphi'(\b'),\G))$ is compact because~$\F[\beta']$ is a field.
By intertwining, see~\cite[Proposition~5.31]{SkodInnerFormI}, we have that~$\F[\beta]/\F$ and~$\F[\beta']/\F$ share the same degrees and inertia degrees. Thus we only have to show 
\begin{equation}\label{eqJ}
\J(\varphi(\beta),\Lambda)/\J^1(\varphi(\beta),\Lambda)\simeq\J(\varphi(\beta'),\Lambda')/\J^1(\varphi(\beta'),\Lambda')
\end{equation}
as groups, because the first one is isomorphic to a general linear group over a finite field. 

By~\cite[Proposition 5.42]{SkodInnerFormI} there are an element~$\gamma\in\A$ (the Lie algebra of~$\tG=\G$) and an element~$g\in\G$ such that~$\gamma$ generates a field and
\[
\Cc(\Lambda,\gamma)=\Cc(\Lambda,\varphi(\beta)),\ \Cc(\Lambda', ^g\gamma)=\Cc(\Lambda',\varphi'(\beta'))
\]
and~$\theta={{\theta}}$ is its {own} transfer between those sets. 
So we can reduce to the following two cases: $\Lambda=\Lambda'$; and~$\b=\b'$ (with possibly~$\Lambda\neq\Lambda'$).

If~$\Lambda=\Lambda'$ then~\eqref{eqJ} follows, because~$\J(-,\Lambda)=\I_\G(\theta)\cap\P(\Lambda)$
with pro-$p$-radical~$\J^1(-,\Lambda)$.
So suppose~$\b=\b',\ \Lambda\neq\Lambda'$ and~$\theta$ is its {own} transfer from~$\Lambda$ to~$\Lambda'$. 
The normalizer of~$\theta$ is contained in the normalizer of~$\Lambda$ because~$\P(\Lambda_{\varphi(\b)})$ is a maximal compact subgroup of~$\G_{\varphi(\beta)}$. Thus~$\F[\varphi'(\beta)]^\times$ normalizes~$\Lambda$. By Lemma~\ref{lemTransfer} the character~$\theta$ and its transfer to~$\Cc(\Lambda,\varphi'(\beta))$ coincide. Thus by the case of equal lattice sequences above we can 
reduce to the case~$\varphi=\varphi'$, i.e., it is sufficient to prove the result in the case
\[
\theta\in\Cc(\Lambda,\varphi(\beta))\cap\Cc(\Lambda',\varphi(\beta)).
\]
The domain of~$\theta$ is equal to~$\H^1(\varphi(\beta),\Lambda)$ and~$\H^1(\varphi(\beta),\Lambda')$. Intersecting with~$\G_{\varphi(\beta)}$ implies
\[
\P_1(\Lambda_{\varphi(\beta)})=\P_1(\Lambda'_{\varphi(\beta)})
\]
and therefore the corresponding hereditary orders~${\mathfrak{b}_0}(\Lambda)$ and~${\mathfrak{b}_0}(\Lambda')$ in~$\B_{\varphi(\beta)}$ have the same radical and therefore coincide. This finishes the proof.
\end{proof}

We are going to study the~$\G$-conjugacy classes of~$m$-realizations of~$\ft$ for~$(\G,h)$. 
Unfortunately, in the case of non-simple endo-parameters there are infinitely many conjugacy classes. So we consider a coarser partition than~$\G$-conjugacy. We will see that this partition is finite.

\begin{definition}
We call two realizations~${\t},{\t'}\in\Cc_-(h)$ of~$\ft$~\emph{essentially~$\G$-conjugate} (resp.~\emph{cuspidal~$\G$-conjugate}) if they intertwine in~$\G$ and there are parametrizations~$((\V,h),\varphi,\Lambda,\b)$ and~$((\V,h),\varphi',\Lambda',\b')$ for~${\t}$ and~${\t'}$, respectively, and an element~$g\in\G$ such that
 \begin{enumerate}
  \item $g\V_i=\V_{\zeta(i)}$, for all~$i\in\I$, and 
  \item ${}^g(\t_{-,\M(\varphi(\b),\G)})=\t'_{-,\M(\varphi'(\b'),\G)}$, see Definition~\ref{defParametrization}, (resp. ${}^g(\t_{-,\M_c(\varphi(\b),\G)})=\t'_{-,\M_c(\varphi'(\b'),\G)}$)
 \end{enumerate}
where~$\zeta:\I\rightarrow\I'$ is the matching given by the intertwining of~${\t},{\t'}$. 
\end{definition}

\begin{proposition}\label{propEssGconj}
Essential~$\G$-conjugacy and cuspidal~$\G$-conjugacy are equivalence relations on~$\Cc_-(h)$. 
\end{proposition}

We call the equivalence classes of Proposition~\ref{propEssGconj} \emph{essential~$\G$-conjugacy classes} (resp. \emph{cuspidal~$\G$-conjugacy classes}). 

\begin{proof}
 It is enough to prove the following claim: 
 \begin{quote}
 {\it Let~$((\V,h),\varphi,\Lambda,\b)$ and~$((\V,h),\varphi',\Lambda',\b')$ be parametrizations of~${\t}\in\Cc_h(\G)$.
 Then there is~$g\in\G$ such that~$g\V_i=\V_{\zeta(i)}$, for all~$i\in\I$,  
 and~$ ^g(\t_{-,\M(\varphi(\b),\G)})=\t_{-,\M(\varphi'(\b'),\G)}$, where~$\zeta:\I\rightarrow\I'$ is the matching given by the intertwining of~${\t}$ with itself.}
 \end{quote}
 Note that such an element~$g$ conjugates~$\M_c(\varphi(\b),\G)$ to $\M_c(\varphi'(\b'),\G)$, so that we also have the equality $ ^g(\t_{-,\M_c(\varphi(\b),\G)})=\t_{-,\M_c(\varphi'(\b'),\G)}$.
 Without loss of generality we can assume that there is an element~$x\in\tG$ such that 
 $x\varphi(\b)x^{-1}=\varphi'(\b')$ and such that~${\t}$ is the transfer of~${\t}$ from~$\Cc_h(\Lambda,\varphi(\b))$ to~$\Cc_h(\Lambda',\varphi'(\b'))$ -- see \cite[Theorem~6.6,~6.10]{SkodInnerFormII}. 
 By~\cite[Theorem~6.7(ii)~(6.8)]{SkodInnerFormII} when~$\e=0$ and~\cite[Theorem~6.12(ii)~(6.14)]{SkodInnerFormII} when~$\e\ne 0$, we have
 \[
 \I_\G({\t})=( S(\varphi'(\b'),\Lambda')\cap\G)(x\tG_{\varphi(\b)}\cap\G)( S(\varphi(\b),\Lambda)\cap\G),
 \]
where~$S(\varphi(\b),\Lambda)$ is a certain compact open subgroup of~$\G$ whose definition can be found in~\cite[before Proposition~5.15]{SkodInnerFormI}. In particular, since~$1$ intertwines~${\t}$, we can write
 \[
 1=uxbv,\ u\in S(\varphi'(\b'),\Lambda')\cap\G,\ v\in S(\varphi(\b),\Lambda)\cap\G,\ xb\in\G\cap x\tG_{\varphi(\b)}.
 \]
Then~$xb$ normalizes~${\t}$ because~$1,u,v$ do. Therefore~$g=xb$ fulfils the desired property. 
\end{proof}  

\begin{theorem}\label{thmEssentiallyConjugateClasses}
 There are only finitely many essential~$\G$-conjugacy classes of m-reali\-zations of~$\ft$ for~$(\G,h)$.
\end{theorem}

As a cuspidal conjugacy class is a union of essential conjugacy classes we obtain the following corollary:

\begin{corollary}\label{corCuspConjugateClasses}
 There are only finitely many cuspidal~$\G$-conjugacy classes of m-reali\-zations of~$\ft$ for~$(\G,h)$.
\end{corollary}

To prove the theorem we need the following lemma:

\begin{lemma}\label{lemTranslatesOfSelfdualLatticeSeq}
Let~$((\V,h),\varphi,\Lambda,\b)$ with non-zero~$h$ be a parametrization of a semisimple character of~$\G$ with~$\I=\I^\sigma$. Let~$\Lambda'$ be a self-dual 
lattice sequence split by the associated splitting of~$\varphi(\b)$ and of the same~$\F$-period as~$\Lambda$.  Suppose that, for all~$i\in\I$, the lattice sequence~$\Lambda'_i$ is a translate of~$\Lambda_i$. 
Then~$\Lambda$ is a translate of~$\Lambda'$. 
\end{lemma}

\begin{proof}
As~$\Lambda$ and~$\Lambda'$ are self-dual, and by the hypothesis, there exist~$l,k,k_i\in\mathbb{Z},$ for~$i\in\I$, such that 
 \begin{enumerate}
  \item $\Lambda(j)^{\#_h}=\Lambda(-j+l)$,
  \item $\Lambda'(j)^{\#_h}=\Lambda'(-j+k)$,
  \item $\Lambda_i(j)=\Lambda'_i(j+k_i)$,
 \end{enumerate}
 for all integers~$j$, where~${\#_h}$ denotes the duality on lattices induced by~$h$. Therefore~$k_i=\frac{k-l}{2}$ is independent of~$i$. 
\end{proof}

\begin{proof}[Proof of Theorem \ref{thmEssentiallyConjugateClasses}]
Let~${\t}\in\Cc_h(\Lambda,\varphi(\b))$ be an~$m$-realization of~$\ft$ for~$(\G,h)$. All other~$m$-realizations of~$\ft$ for~$(\G,h)$ intertwine with~${\t}$ in~$\G$ and we can therefore conjugate the associated splittings to the splitting of~$\varphi(\b)$ by an element of~$\G$, by~\cite[Theorem~6.7(i) and~6.12(i)]{SkodInnerFormII}, which reduces us to the case where~$\I=\I^\sigma$ and~$\M(\varphi(\b),\G)=\G$. Thus if~$h=0$ then we have~$|\I|=1$ and if~$h\neq 0$  we have~$\I=\I^\sigma$. We have to show that there are finitely many~$\G$-conjugacy classes of~$m$-realizations of~$\ft$ for~$(\G,h)$.  
 Note that there are only finitely many~$\G$-orbits of points~$x'=[\Lambda']$ in the reduced building~$\mathfrak{B}_{red}(\G)$ for which there exists an~$m$-realization~${\t'}\in\Cc_h (\Lambda',\varphi'(\b'))$ of~$\ft$ for~$(\G,h)$ such that, under the embedding~\eqref{eqjbbar}, the point~$x'$ is the image of a point attached to a maximal parahoric subgroup of~$\G_{\varphi'(\b')}$; indeed, such a point~$x'=(x'_i)_{i\in\I},\ x'_i\in\mathfrak{B}_{red}(\Aut_\D(\V_i)),$ fulfils extra conditions on the barycentric coordinates: for example, the barycentric coordinates of~$x'_i$  have to be of the form $\frac{q}{2e(\E_i:\F)}$ with~$0\leq q\leq 2e(\E_i/\F)$, by the formula in~\cite[Lemma II.3.1]{broussousLemaire:02}. Thus we have only a finite number of~$\G\cap\Aut_\D(\V_i)$-orbits for points~$x'_i$. Lemma~\ref{lemTranslatesOfSelfdualLatticeSeq} then provides the finiteness of the number of~$\G$-orbits for points~$x'$. 

We consider the point in~$\mathfrak{B}_{red}(\G)$ corresponding to~$\Lambda$ given by the chosen parametrization of~${\t}$. Every~$m$-realization~${\t'}\in\Cc_h(\Lambda,\varphi'(\b'))$ intertwines with~${\t}$ in~$\G$. So we obtain a map
\[
\zeta_{{\t'}}:k_{\E}\hookrightarrow\mathfrak{a}_0/\mathfrak{a}_1,
\]
by~\cite[Lemma~5.51]{SkodInnerFormI}, where~$\mathfrak{a}_0=\mathfrak{a}_0(\Lambda)$ is the hereditary order corresponding to~$\Lambda$ and~$\mathfrak{a}_1$ its radical. Two~$m$-realizations~${\t'}\in\Cc_h(\Lambda,\varphi'(\b'))$ and~${\t''}\in\Cc_h(\Lambda,\varphi''(\b''))$ of~$\ft$ for~$(\G,h)$ are~$\G$-conjugate if~$\zeta_{{\t'}}=\zeta_{{\t''}}$, by~\cite[Theorem 1.5]{SkodInnerFormI} when~$\e=0$,~and~by~\cite[Theorem 1.1]{SkodInnerFormII} (quaternionic case),~\cite[Theorem 10.3]{SkSt}  (symplectic and unitary case)~and~\cite[Theorem~10.4, see Remark 10.5]{KSS} (special orthogonal case) when~$\e\ne 0$. The number of field embeddings of~{$k_\E$} into~$\mathfrak{a}_0/\mathfrak{a}_1$ is finite, because both sets are finite. This finishes the proof. 
\end{proof}

\begin{lemma}\label{EPsupportmap}
Suppose $\ft$ is an endo-parameter for $(\G,h)$ and put $\M=\M(\ft)$ and~$\M_c=\M_c(\ft)$, the Levi subgroups associated to $\ft$ (well-defined up to $\G$-conjugacy), and set~$\ft_{\M}$ and~$\ft_{\M_c}$ as in \ref{SeccuspsuppEP}.  Then:
\begin{enumerate}
\item The essential $\G$-conjugacy classes of~$m$-realizations of $\ft$ are in bijection with the $\M$-conjugacy classes of~$m$-realizations of $\ft_\M$.
\item The cuspidal $\G$-conjugacy classes of~$m$-realizations of $\ft$ are in bijection with the $\M_c$-conjugacy classes of~$m$-realizations of $\ft_{\M_c}$.
\end{enumerate}
\end{lemma}

\begin{proof}
Let ${\theta}$ be an~$m$-realization of $\ft$ from $((\V,h),\varphi,\Lambda,\beta)$, so that (by conjugating in $\G$) we can assume $\M(\varphi(\beta),\G)=\M$ and $\M_c(\varphi(\beta),\G)=\M_c$. Then ${\theta_{\M}}={\theta}\mid_{\M}$ is a semisimple character for $\M$ and determines an endo-parameter $\ft_\M$, and ${\theta_{\M}}$ is an~$m$-realization of $\ft_\M$.  Similar we obtain~$\ft_{\M_c}$.
 
Now suppose ${\theta'}$ is another~$m$-realization of $\ft$, from $((\V,h),\varphi',\Lambda',\beta')$; again, by conjugating in $\G$, we can assume $\M(\varphi'(\beta'),\G)=\M$ and $\M_c(\varphi'(\beta'),\G)=\M_c$. Then we similarly get an endo-parameter $\ft'_\M$ for $\M$, which is $\N_\G(\M)$-conjugate to $\ft_\M$. In fact, by changing our conjugation in $\G$, we can assume $\ft_\M=\ft'_{\M}$ and~$\ft_{\M_c}=\ft'_{\M_c}$ -- this is the same as conjugating so that the matching between ${\theta}$ and ${\theta'}$ is the identity map.
 
By definition, ${\theta'}$ is essentially $\G$-conjugate to ${\theta}$ if and only if there is $g\in\G$ such that ${}^g\M=\M$ and ${}^g{\theta_{\M}}={\theta'_{\M}}$, that is, if and only if there is $g\in \N_\G(\M)$ such that ${}^g{\theta_{\M}}={\theta'_{\M}}$. But any such $g$ must come from the matching between ${\theta},{\theta'}$ so in fact is in $\M$. Thus~${\theta'}$ is essentially $\G$-conjugate to ${\theta}$ if and only if ${\theta_{\M}}$ is $\M$-conjugate to ${\theta'_{\M}}$.

In the case~$\M\neq\M_c$ the second assertion has the identical proof, noting the fact that the restriction of~${\theta}$ and~${\theta'}$ onto the factor where~$\beta$ and~$\beta'$ are zero are just trivial characters of~$1+\mathfrak{p}_\F$.
\end{proof}

For further applications it is important to be able to parametrize all~$m$-realizations of an endo-parameter~$\ft$ with a fixed embedding~$\varphi$. This is already a property for realizations of~$\ft$. 
\begin{proposition}\label{propFixedEmbedding}
Let~$\ft$ be an endo-parameter of~$(\G,h)$ and~$\theta\in\Cc_h(\Lambda,\varphi(\b))$ a realization of~$\ft$. 
\begin{enumerate}
\item Let~$\theta'\in\Cc_h(\Lambda',\varphi'(\b'))$ be a second realization of~$\ft$. 
Then there is an~$h$-self-dual embedding~$\psi:\ \F[\beta]\rightarrow \A$ such that~$\theta'\in\Cc_h(\Lambda',\psi(\b))$.
\item Let~$C$ be the closure of a chamber of the building of~$\G_{\varphi(\b)}$ (strong simplicial structure). Then for every realization~$\theta'$ of~$\ft$ there {are} an~${\o_{\varphi(\E)}}$-$\o_\D$-lattice sequence~$\Lambda''$ in~$j_{\varphi(\b)}(C)$
and a semisimple character~$\theta''\in\Cc_h(\Lambda'',\varphi(\b))$ which is~$\G$-conjugate to~$\theta'$. If~$\theta'$ is an~$m$-realization of~$\ft$ then~$\theta''$ is an~$m$-realization of~$\ft$. 
\end{enumerate}
\end{proposition}

\begin{proof}
The second assertion follows from the first one, the transitivity of the~$\G_{\varphi(\b)}$-action on the set of chambers of~$\mathfrak{B}(\G_{\varphi(\b)})$ and Proposition~\ref{propmRealization}. 
As~$\theta$ and~$\theta'$ are realizations of the same endo-parameter, they intertwine over~$\G$ and in particular are endo-equivalent. Thus by~\cite[Theorem 6.6, Theorem 6.10]{SkodInnerFormII}  there exist an~$h$-self-dual semisimple element~$\gamma$ and 
a semisimple character~$\tilde{\theta}\in\Cc_{h\oplus h}(\Lambda\oplus\Lambda',\varphi(\b)\oplus\gamma)$ such that~$\tilde{\theta}|_{\H^1(\varphi(\b),\Lambda)}=\theta$ and~$\tilde{\theta}|_{\H^1(\gamma,\Lambda')}=\theta'$. As~$\theta$ and~$\theta'$ are endo-equivalent
~$\varphi(\b)$ and~$\gamma$ must have the same minimal polynomial over~$\F$; here we use that~$[\Lambda\oplus\Lambda',-,0,\varphi(\b)\oplus\gamma]$ is a semisimple stratum and that different block-restrictions of~$\tilde{\theta}$ are not endo-equivalent. Further, as~$\theta,\theta'$ intertwine they have the same reduced characteristic polynomial by the matching theorem~\cite[Theorem 6.5]{SkodInnerFormII}. Thus, by~\cite[(6.14)]{SkodInnerFormII} there is an element 
of~$\G$ which conjugates~$\varphi(\b)$ to~$\gamma$. We take~$\psi$ to be~$\varphi$ followed by this conjugation. 
\end{proof}

One could suspect that two~$m$-realizations of~$\ft$ parametrized by the same pair~$(\Lambda,\varphi(\b))$ are essentially~$\G$-conjugate, but the following remark provides a counter example. 
The example is related to~\cite[Remark 8.12]{KSS}.

\begin{remark}\label{remNoIntertwiningImpliesEssConjugacy}
Suppose~$\F$ has positive odd characteristic~$p$, $\G=\Sp_{8p}(\F)=\U(h)$ {and fix a uniformizer~$\varpi_\F$}. We write~$h$ as an orthogonal direct sum of Hermitian forms~$h_0$, such that~$\U(h_0)=\Sp_8(\F)$. 
Let~$\alpha\in o_\F^\times$ be a non-square in~$\F$ and~$\b_0$ be a~$\sigma_{h_0}$ skew-element of~$\M_8(\F)$ whose square is~$\varpi_\F^{-4}\alpha$ and generates a quadratic unramified field extension of~$\F$. 
We put~$\lambda:=\varpi_\F\b_0$. Then for a lattice sequence~$\Lambda_0$ of~$\F^8$ which under~$j_{\b_0}$ is the image of a special vertex of~$\mathfrak{B}(\G_{\b_0})$ we have 
\[{\Cc_h}(\Lambda_0,\b_0)={\Cc_h}(\Lambda_0,\b_0+\lambda).\]
(Note that~$[\Lambda_0,-,0,\b_0]$ and~$[\Lambda_0,-,0,\b_0+\lambda]$ are simple strata as the elements are minimal over~$\F$.)
We put~$\Lambda=\Lambda_0\oplus\ldots\oplus\Lambda_0$ ($p$-times) and
\[\b=\b_0\oplus (\b_0+\lambda)  \oplus (\b_0+2\lambda)\oplus \ldots\oplus (\b_0+(p-1)\lambda) \]
We choose a simple character~$\theta_0\in{\Cc_h}(\Lambda_0,\b_0)$ and its transfer~$\theta\in{\Cc_h}(\Lambda,\b_0\oplus\ldots\oplus\b_0)$ and we put
$\theta^{(1)}:=\theta\psi_{\oplus_{i=0}^{p-1}i\lambda}$ and~$\theta^{(2)}:=\theta\psi_{\oplus_{i=1}^{p}i\lambda}$, both are elements of~${\Cc_h}(\Lambda,\b)$. 
(Note that~$[\Lambda,-,0,\b]$ is a semisimple stratum by~\cite[Theorem 6.15]{SkSt}). We now choose a non-special vertex~$\Lambda'_{0,\b_0}$ in~$\mathfrak{B}(\G_{\b_0})$ and we put~$\Upsilon:=\Lambda'_0\oplus\Lambda_0\oplus\ldots\oplus\Lambda_0$
and we obtain the transfers~$\zeta^{(1)},\zeta^{(2)}\in{\Cc_h}(\Upsilon,\b)$ of~$\theta^{(1)},\theta^{(2)}$. Then~$\zeta^{(1)}$ and~$\zeta^{(2)}$ intertwine in~$\G$ such that the matching is a cyclic permutation.  Now if~$\zeta^{(1)}$ and~$\zeta^{(2)}$ are~$\G$-conjugate then 
by~\cite[Theorem 6.12]{SkodInnerFormII} the second block restriction~$\zeta^{(1)}_2$ is~$\G$-conjugate to the first block-restriction~$\zeta^{(2)}_1$. Then the finite quotients~$\P(\Lambda_{0,\b_0})/\P_1(\Lambda_{0,\b_0})$ and~$\P(\Lambda'_{0,\b_0})/\P_1(\Lambda'_{0,\b_0})$
are isomorphic. A contradiction as the first one is isomorphic to the unitary group~$\U(2,2)(k_\E|k_\F)$ and the second one to~$\U(1,1)(k_\E|k_\F)\times \U(1,1)(k_\E|k_\F)$ where~$\E:=\F[\b_0]$.
\end{remark}

{Nevertheless,} for cuspidal conjugacy classes we can still find a nice set of representatives. 
\begin{proposition}\label{propNiceSystemOfRepresentatives}
Let~$\ft$ be an endo-parameter of~$(\G,h)$ and~$\theta\in\Cc_h(\Lambda,\varphi(\b))$ {be} a realization of~$\ft$. Suppose that~$\Lambda_1,\ldots,\Lambda_k$ are lattice sequences such that~$\Lambda_{1,\varphi(\b)},\ldots,\Lambda_{k,\varphi(\b)}$ form a system of representatives of 
$\G_{\varphi(\b)}$-conjugacy classes of vertices of a chamber of~$\mathfrak{B}(\G_{\varphi(\b)})$ in the strong simplicial structure. 
Let~$\theta_i$ be the transfer of~$\theta$ to~$\Lambda_i$ for~$i=1,\ldots,k$. 
Then~$\theta_1,\ldots,\theta_k$ form a system of representatives of cuspidal $\G$-conjugacy classes of~$m$-realizations of~$\ft$. 
\end{proposition}

\begin{proof}
At first we observe that different~$\theta_i$ and~$\theta_j$ cannot be cuspidal~$\G$-conjugate, because, as they are transfers this cuspidal conjugation can be obtained by an element of the centralizer~$\G_{\varphi(\b)}$ which conjugates the maximal parahoric~$\P(\Lambda_{i,\varphi(\b)})$
to~$\P(\Lambda_{j,\varphi(\b)})$, so that the corresponding vertices are~$\G_{\varphi(\b)}$-conjugate, a contradiction. 

Let~$\theta'$ be an~$m$-realization of~$\ft$. By Proposition~\ref{propFixedEmbedding} and its proof we can parametrize~$\theta'$ by a pair~$(\Lambda',\psi(\b))$ such that~$\theta'$ is a transfer of~$\theta$. As $\theta'$ intertwines with~{$\theta$} in~$\G$ we can find an element
$g\in\G$ such that~$g\varphi(\b) g^{-1}=\psi(\b)$, see Skolem-Noether in \cite[Corollary 4.14]{SkodInnerFormII}.  Therefore we can without loss of generality assume that~$\varphi=\psi$. The vertex associated to~$\Lambda'_{\varphi(\b)}$ is~$\G_{\varphi(\b)}$-conjugate to the vertex associated
 to some~$\Lambda_{i,\varphi(\b)}$.
So without loss of generality we can assume that~$\Lambda'_{\varphi(\b)}$ and~$\Lambda_{i,\varphi(\b)}$ represent the same vertex. As~$\theta'$ and~$\theta_i$ are transfers they have to intertwine by~$1$. On the factors of~$\M_c(\varphi(\b),\G)$ the restriction of~$\Lambda'$ 
is a translate of the restriction of~$\Lambda_{i}$, because those restrictions represent the same vertex in the corresponding factor of the building of the centralizer~$\G_{\varphi(\b)}$. 
Thus~$\theta'$ and~$\theta_i$ coincide on~$\M_c(\varphi(\b),\G)$. 
\end{proof}

\begin{remark}
The above proposition would be wrong if we {had} chosen essential conjugacy instead of cuspidal conjugacy, as if two~$m$-realizations~$\theta\in{\Cc_h}(\Lambda,\varphi(\b))$ and~$\theta'\in{\Cc_h}(\Lambda',\varphi(\b))$ intertwine by $1$ and if~$\Lambda_{\varphi(\b)},\Lambda'_{\varphi(\b)}$ represent the same vertex in the reduced building of~$\G_{\varphi(\b)}$ then it may happen that~$\H^1(\varphi(\b),\Lambda)\cap\M(\varphi(\b),\G)$ and~$\H^1(\varphi(\b),\Lambda)\cap\M(\varphi(\b),\G)$ are different and even not~$\G_{\varphi(\b)}$-conjugate. But we always have that their intersections with~$\M_c(\varphi(\b),\G)$ coincide.  Here we provide an example. 

Let~$\D=\F$ and~$h$ be an orthogonal bilinear form on a~$6$-dimensional~$\F$-vector space~$\V$ with Witt index~$3$, i.e., we consider a basis of~$\V$ such that the Gram matrix of~$h$ is anti-diagonal with all entries being~$1$ on the anti-diagonal (a Witt basis).  
We consider the {strict self-dual lattice sequences~$\Lambda,\Lambda'$ with the following associated} hereditary orders, and element~$\b$, with respect to the fixed Witt basis: 
\[{\mathfrak{a}_0(\Lambda)=\left(\begin{array}{lll} \o_\F & \p_\F & \p_\F \\ \o_\F & \o_\F & \p_\F  \\ \p_\F^{-1} & \o_\F & \o_\F \end{array}\right)^{(2,2,2)}\!\!\!\!,\ \mathfrak{a}_0(\Lambda')=\left(\begin{array}{ll} \o_\F & \p_\F \\ \p_\F^{-1} & \o_\F \end{array}\right)^{(3,3)}}\!\!\!, \ 
\b=\varpi^{-1}\left(\begin{array}{lllll}  & \zeta & & & \\ 1 & & & & \\ & & 0 & &  \\ & & & & -\zeta \\ & & & 1 & \end{array}\right)^{(1,1,2,1,1)} \]
with a non-square~$\zeta\in\o^\times$, {where the superscripts indicate the block sizes}. 
Then~$[\Lambda,2,0,\b]$, $[\Lambda',1,0,\b]$ are semisimple strata and~$\Lambda_\beta, \Lambda'_\beta$ represent vertexes in the reduced building of~$\G_\b$. We note that the critical exponents are~$-2$ and $-1$ respectively. 
So we obtain the groups~$\H^1(\b,\Lambda)=\P_1(\Lambda_\b)\P_2(\Lambda)$, {where~$\P_2(\Lambda)=(1+\p_\F\mathfrak{a}_0(\Lambda))\cap\G$,} and~$\H^1(\b,\Lambda')=\P_1(\Lambda')$. Note that~$\G_\b$ (and therefore $\P_1(\Lambda_\b)$) is contained in
\[\M_c=\left(\begin{array}{lll}* & & * \\ & * & \\ * & & * \end{array}\right)^{(2,2,2)}.\]
So we find an element in~$\H^1(\b,\Lambda')$ which is not in~$\H^1(\b,\Lambda)$. 
Similar we see that any~$\G_\b$-conjugate of~$\H^1(\b,\Lambda)$ is different {from}~$\H^1(\b,\Lambda')$. 
\end{remark}

%%%%%%%%%%%%%%%%%%%%%%%%%%%%%
%\section{Heisenberg representations}\label{Heisenberg}
%%%%%%%%%%%%%%%%%%%%%%%%%%%%%

%%%%%%%%%%%%%%%%%%%%%%%%%%%%%
\section{Exhaustion of~$m$-realizations in an Endo-factor}\label{secESing}
%%%%%%%%%%%%%%%%%%%%%%%%%%%%%

%%%%%%%%%%%%%%%%%%%%%%%%%%%%%
%\subsection{Heisenberg representations and semisimple characters over~$\mathbb{Z}[1/p,\mu_{p^{\infty}}]$}\label{HeisenbergZ}
Let~$\beta$ be a self-dual full semisimple element, and~$((\V,h),\varphi,\Lambda)\in\mathscr{Q}_{\F/\F_\so,\varepsilon}(\beta)$.  Then also attached to this datum, see \cite[3.2]{St08}, \cite[Definition~5.4]{SkodInnerFormI}, and \cite[\S4]{SkodlerackCuspQuart}, are~$\Sigma$-stable compact open subgroups~$\widetilde{\J}^1(\varphi(\beta),\Lambda)\leqslant \widetilde{\J}(\varphi(\beta),\Lambda)$ of~$\tG$ containing~$\H^1(\varphi(\beta),\Lambda)$, and we write
\[\J^1(\varphi(\beta),\Lambda)=\widetilde{\J}^1(\varphi(\beta),\Lambda)\cap\G,\quad \J(\varphi(\beta),\Lambda)=\widetilde{\J}(\varphi(\beta),\Lambda)\cap\G,\]
for the associated compact open subgroups of~$\G$.  The group~$\J^1(\varphi(\beta),\Lambda)$ is pro-$p$, and normal in~$\J(\varphi(\beta),\Lambda)$ with
\[\J(\varphi(\beta),\Lambda)/\J^1(\varphi(\beta),\Lambda)\simeq \P(\Lambda_{\varphi(\beta)}) /\P_1(\Lambda_{\varphi(\b)}),\]
a finite reductive group.  

Let~$\t\in\Cc_h(\Lambda,\varphi(\b))$ be a self-dual semisimple character.  Then there exists a unique irreducible $\mathbb{C}$-representation~$\eta$ of~$\J^1(\varphi(\beta),\Lambda)$ which contains~$\t$, by \cite[Proposition 3.5]{St08}, Lemma~\ref{lemkthetanondeg} and \cite[Proposition 4.3]{SkodlerackCuspQuart}.  These representations are called \emph{Heisenberg~$\mathbb{C}$-representations}, and this definition is extended to algebraically closed fields in \cite{RKSS}.

Let~$\Rep_\R(\G)$ denote the abelian category of smooth $\R$-representations of~$\G$.  
%We fix our base ring~$\R_0=\mathbb{Z}[1/p,\mu_{p^{\infty}}]$, which contains all values of all semisimple characters over~$\mathbb{C}$, and base rings~$\R_{0,r}=\mathbb{Z}[1/p,\mu_{p^{r}}]$ which contain all values of all semisimple characters over~$\mathbb{C}$ of depth $\leq d(r)$, so~$\R_0=\bigcup \R_{0,r}$.

\begin{definition}
Let~$\Sigma$ be a collection of finitely generated projective smooth~$\R$-representations of compact open subgroups of~$\G$, and~$\mathcal{H}$ a full abelian subcategory of~$\Rep_{\R}(\G)$.  We say that~$\Sigma$ \emph{exhausts}~$\mathcal{H}$ if, for any smooth~$\R$-representation~$\pi$ of~$\G$ contained in~$\mathcal{H}$, there exists~$(\K,\rho)\in\Sigma$ such that~$\Hom_{\R[\K]}(\rho,\pi)\neq 0$.
\end{definition}

By a result of Dat, {the} semisimple characters for~$\G$ exhaust the category of smooth representations:

\begin{proposition}[{\cite[Propositions 7.5 \& 8.5]{Dat09}, Theorem~\ref{thmSemiCharInRep}, \cite[Theorem~3.1]{SkodlerackCuspQuart}}]\label{DatProp}~
 The collection of semisimple~$\R$-characters for~$\G$ exhausts~$\Rep_{\R}(\G)$. 
\end{proposition}
%For classical groups with~$\R=\R_0$ this is proved by Dat in ibid., in the other references this is proved for smooth~$\mathbb{C}$-representations and for smooth~$\Fl$-representations of~$\G$ for primes~$\ell\neq p$.  

\begin{corollary}\label{Corollarythetaisotypic}
Let $(\pi,\cV)$ be a smooth $\R$-representation of $\G$.  Then $\pi$ is generated by the sum of all its $\theta$-isotypic components, where the sum is over all semisimple~$\R$-characters $\theta$ for $\G$. 
\end{corollary}

\begin{proof}
Let $\cW$ be the $\R$-subrepresentation of $\cV$ generated by the sum of the $\theta$-isotypic components of $\cV$ over all semisimple~$\R$-characters~$\theta$. For each self-dual semisimple character~$\theta$, since the~$\theta$-isotypic functor is exact,~$\cV^\theta=\cW^\theta$ and~$(\cV/\cW)^\theta=0$.  Thus $\cV/\cW$ contains no self-dual semisimple characters, and by Proposition \ref{DatProp} it is zero.
\end{proof}

\begin{definition}
Let~$\ft$ be an endo-parameter for~$\G$.   A smooth $\R$-representation $\pi$ of $\G$ is of \emph{class} $\ft$ if every semisimple character for~$\G$ contained in $\pi$ has endo-parameter $\ft$.
\end{definition}

Note that a representation of~$\pi$ is of class~$\ft$ if it is generated by semisimple characters of class~$\ft$, as a semisimple character~$\theta$ in~$\pi$ in that case is contained, {by projectivity}, in a direct sum of compactly induced representations of semisimple characters of class~$\ft$, which implies by Mackey that~$\theta$ is of class~$\ft$. We immediately get the following.

\begin{corollary}\label{corollaryendoclassifandonlyifgeneratedbyisotypic}
Let $\pi$ be a smooth $R$-representation of $\G$.  Then $\pi$ is of class $\ft$ if and only if it is generated by the sum of its $\theta$-isotypic components, where the sum is over all semisimple characters $\theta$ for~$\G$ of endo-parameter $\ft$.
\end{corollary}

Note that, a smooth $\R$-representation~$\pi$ of $\G$ is of class $\ft$ if and only if every irreducible subquotient of $\pi$ is of class $\ft$, since semisimple characters are projective.  Hence we make the following definition:

\begin{definition}
We let~$\Rep_\R(\ft)$ denote the full abelian subcategory of $\Rep_\R(\G)$ consisting of all representations of class~$\ft$.  We call~$\Rep_\R(\ft)$ an \emph{endo-factor} of~$\Rep_\R(\G)$.
\end{definition}

In fact, we only need to look at~$m$-realizations in order to generate a representation of class~$\ft$.

\begin{theorem}\label{thm43}
Let~$\ft$ be an endo-parameter for~$\G$. The collection of all~$m$-realizations of~$\ft$ exhausts~$\Rep_\R(\ft)$. 
\end{theorem}

\begin{proof}
\
Suppose~$\pi$ is of class~$\ft$.   By Corollary~\ref{corollaryendoclassifandonlyifgeneratedbyisotypic}, it is generated by the sum of its~$\theta$-isotypic components, for~$\theta$ of endo-parameter~$\ft$.  Writing~$\cV_{\text{m}}$ for the subspace generated by the sum of its~$\theta$-isotypic components where~$\theta$ passes over the set of all~$m$-realizations of~$\ft$, we see that~$\cV/\cV_{\text{m}}$ is of class~$\ft$ but contains no~$m$-realization of~$\ft$. If~$\cV/\cV_{\text{m}}$ is non-zero then it has an irreducible subquotient with the same property. Thus it suffices to show that if~$\pi$ is irreducible of class~$\ft$ then it contains an~$m$-realization of~$\ft$.

Thus let~$\pi$ be irreducible of class~$\ft$ and let~$\theta\in\Cc_h(\Lambda,\varphi(\beta))$ be any semisimple character of endo-parameter~$\ft$ such that~$\cV^\theta\neq 0$.  Then~$\cV^\eta=\cV^\theta \neq 0$, where~$\eta$ is the unique Heisenberg representation of~$\J^1(\varphi(\beta),\Lambda)$ containing~$\theta$.

Choose~$\Lambda_M\in\mathcal{B}(\G_{\varphi(\beta)})$ a vertex of the facet containing~$\Lambda$ (strong simplicial structure) such that $\P(\Lambda_{M,\varphi(\b)})^{\circ}$ is a maximal parahoric subgroup in~$\G_{\varphi(\beta)}$.  We will show that~$\pi$ contains the transfer~$\theta_M$ of~$\theta$ to~$\La_M$.
By~\cite[Lemma 2.8]{St08} (see also the paragraph following Lemma 3.11 in~\cite{SkodlerackCuspQuart}), 
there exists~$\Lambda'\in\mathcal{B}(\G_{\varphi(\beta)})$ such that~$\P(\Lambda'_{\varphi(\b)})^{\circ}$ and~$\P(\Lambda_{\varphi(\b)})^{\circ}$ coincide and~$\P(\Lambda_M)^\circ\supseteq \P(\Lambda')^\circ$.  Moreover, by~\cite[Lemma 2.10]{St08} there exists a sequence in~$\mathcal{B}(\G_{\varphi(\beta)})$
\[
\Lambda=\Lambda_0,\ldots,\Lambda_t=\Lambda'
\]
such that~$\P(\Lambda_{i,\varphi(\beta)})^\circ=\P(\Lambda_{\varphi(\b)})^{\circ}=\P(\Lambda'_{\varphi(\b)})^{\circ}$ and either~$\P(\Lambda_i)^\circ\subseteq \P(\Lambda_{i-1})^\circ$ or~$\P(\Lambda_i)^\circ\supseteq \P(\Lambda_{i-1})^\circ$. In fact we choose~$\Lambda_i$ close enough to~$\Lambda_{i-1}$ such that the above inclusions are satisfied for the corresponding hereditary orders.  Let~$\theta',\theta_i$ be the transfer of~$\theta$ to~$\La'$ and~$\La_i$, respectively, for~$0\leqslant i \leqslant t$.  Similarly, we write~$\eta_M,\eta',\eta_i$ for the Heisenberg representations associated to~$\theta_M,\theta',\theta_i$ respectively, and abbreviate~$\J^1_M=\J^1(\varphi(\beta),\Lambda_M)$ and similarly~$\J'^1,\J^1_i$.

For~$0< i\leqslant t$, suppose first that~$\P(\Lambda_i)^\circ\subseteq \P(\Lambda_{i-1})^\circ$ so that~$\P_1(\Lambda_i)\supseteq \P_1(\Lambda_{i-1})$. By~\cite[Proposition 3.7]{St08},~\cite[Proposition~4.5]{SkodlerackCuspQuart} and Proposition~\ref{propHeisenbergPairs}, we find that~$\ind_{\J^1_{i-1}}^{\P_1(\Lambda_i)}(\eta_{i-1})\simeq \ind_{ \J^1_i}^{\P_1(\Lambda_i)}(\eta_i)$ is irreducible, hence
\[
{\cV^{\eta_i}}=\cV^{ \ind_{ \J^1_i}^{\P_1(\Lambda_i)}(\eta_i)}=\cV^{\ind_{\J^1_{i-1}}^{\P_1(\Lambda_i)}(\eta_{i-1})}={\cV^\eta_{i-1}}.
\]
Therefore,~$\cV^{\eta_i}$ is non-zero if and only if~$\cV^{\eta_{i-1}}$ is non-zero.
The same result holds when~$\P(\Lambda_i)^\circ\supseteq \P(\Lambda_{i-1})^\circ$, by inducing instead to~$\P_1(\Lambda_{i-1})$. 
Thus, iterating this procedure along the path~$\Lambda=\Lambda_0,\ldots,\Lambda_t=\Lambda'$, we find that~$\cV^{\eta'}$ is non-zero.

By~\cite[Proposition 3.7]{St08} etc.~again, there is a unique irreducible representation~$\hat\eta_M$ of~$\hat\J^1_M=\P_1(\Lambda'_{\varphi(\b)})\J^1_M$ which extends~$\eta_M$ and such that~$\ind_{\J'^1}^{\P_1(\Lambda')}(\eta')\simeq \ind_{ \hat\J^1_M}^{\P_1(\Lambda')}(\hat\eta_M)$ is irreducible. Hence, we repeat the above argument to get that~$\cV^{\hat\eta_M} $ is non-zero. 
%\[
%\cV^{\eta'} \subseteq 
%\cV^{ \ind_{\J'^1}^{\P_1(\Lambda')}(\eta')}=\cV^{\ind_{ \hat\J^1_M}^{\P_1(\Lambda')}(\hat\eta_M)} \text{ and } 
%\cV^{\hat\eta_M} \subseteq \cV^{\eta_M}.
%\]
Thus~$\cV^{\eta_M}=\cV^{\theta_M}$ is non-zero, as required.
\end{proof}

%%%%%%%%%%%%%%%%%%%%%%%%%%%%%
\section{Beta extensions and types for Bernstein blocks}\label{secbetaextsandtypes}
%%%%%%%%%%%%%%%%%%%%%%%%%%%%%

Let~$\mathfrak{t}$ be an endo-parameter for~$(\G,h)$ and~$(0,\beta)$ a full semisimple pair for~$\mathfrak{t}$ and set~$\E=\F[\beta]$.  
In this section, we collect results which in a future work will allow us to decompose~$\Rep_\R(\ft)$.  It follows from a simple cohomology calculation (cf.~\cite[5.2.4]{BK93}, \cite[Theorem 4.1]{St08}) in characteristic zero, and either by an analogous argument or a simple reduction modulo~$\ell$ argument for positive characteristic, that the Heisenberg representation~{$\eta$ of~$\J^1=\J^1(\varphi(\beta),\Lambda)$} extends to an irreducible~$\R$-representation of~$\J=\J(\varphi(\beta),\Lambda)$.  Now, 
 \begin{enumerate}
\item for the construction of types for Bernstein blocks it is useful to choose extensions with strong intertwining properties, this leads to the notion of ``beta extensions'';
\item for questions related to {comparing} types for different blocks, we need to choose beta extensions ``compatibly''. 
\end{enumerate}

In the literature in the case of depth zero for the construction of types the strong simplicial structure of the Bruhat-Tits building of~$\G$ has been used. 
So the first main purpose is to introduce beta extensions where we use the strong simplicial structure of the building of the centralizer of~$\beta$. 
The second is to recall the constructions of types for Bernstein blocks, {and show that the collection of beta extensions using the strong simplicial structure is sufficient for the construction of types}.  {As well as making it more resemble the depth zero situation, this very mildly reduces a choice made in the construction of types (from choosing a weak vertex to choosing a strong vertex).}

%%%%%%%%%%%%%%%%%%%%%%%%%%%%%
\subsection{Beta extensions}\label{subsecBeta}
Let~$\theta$ be an $m$-realization of~$\mathfrak{t}$ for~$(\G,h)$ and choose for~$\theta$ a para\-metri\-za\-tion~$((\V,h),\varphi,\Lambda,\beta)$ so that~$\theta\in\mathcal{C}_h(\Lambda,\varphi(\beta))$,
 in particular~$\P(\Lambda_{{\varphi(\beta)}})^\circ$ is a maximal parahoric subgroup of~$\G_{{\varphi(\beta)}}$. For any self-dual~${\o_{\varphi(\E)}}$-lattice sequence~$\Upsilon$, via the transfer map we have a semisimple character $\theta_{\Upsilon}=\tau_{\Lambda,\Upsilon,\varphi(\beta)}(\theta)$ and a Heisenberg extension~$\eta_{\Upsilon}$ of~$\J^1_\Upsilon=\J^1(\varphi(\beta),\Upsilon)$ the pro-$p$-radical of~$\J_\Upsilon=\J(\varphi(\beta),\Upsilon)$. We  denote by~$\P^{\st}(\Upsilon_{{\varphi(\beta)}})$ the pointwise fixator of the facet~$\overline{\Upsilon_{{\varphi(\beta)}}}$ (under the strong simplicial structure, as in \ref{parahoricssec}) containing~$\Upsilon_{{\varphi(\beta)}}$ in the building of~$\G_\beta$. 
 
%%%%%%%%%%%%%%%%%%%%%%%%%%%%%
\subsubsection{Extensions of Heisenberg extensions}
 Suppose~$\Upsilon$ and~$\Upsilon'$ are~${\o_{\varphi(\E)}}$-lattice sequences satisfying \[\P(\Upsilon_{{\varphi(\beta)}})^\circ \subseteq\P(\Upsilon'_{{\varphi(\beta)}})^\circ.\] 
 As~$\P_1(\Upsilon_{{\varphi(\beta)}})\subseteq\P^{\st}(\Upsilon_{{\varphi(\beta)}})\subseteq\P(\Upsilon'_{{\varphi(\beta)}})$, they both normalize~$\J^1(\varphi(\beta),\Upsilon')$ and we can form the subgroups 
 \[\J^{\st}_{\Upsilon,\Upsilon'}:=\P^{\st}(\Upsilon_{{\varphi(\beta)}})\J^1(\varphi(\beta),\Upsilon'),\text{ and }\J^{1}_{\Upsilon,\Upsilon'}:=\P_1(\Upsilon_{{\varphi(\beta)}})\J^1(\varphi(\beta),\Upsilon'),\]
 of~$\J(\varphi(\beta),\Upsilon)$.  We let~$\J^{\st}_{\Upsilon}=\J^{\st}_{\Upsilon,\Upsilon}$ and also use the notation~$\J^{\st}(\varphi(\beta),\Upsilon)$ where we want to emphasize the dependence on~$\varphi(\beta)$.
 Thus we have the following chain of subgroups of~$\J_{\Upsilon'}$:
 \[\J^1_{\Upsilon'}\subseteq\J^1_{\Upsilon,\Upsilon'}\subseteq\J^{\st}_{\Upsilon,\Upsilon'}\subseteq\J^{\st}_{\Upsilon'}\subseteq\J_{\Upsilon'}.\]
 
Suppose~$\P(\Upsilon)\subseteq \P(\Upsilon')$, then by \cite[Proposition 3.7]{St08}, Proposition~\ref{propHeisenbergPairs},~\cite[Proposition 4.5]{SkodlerackCuspQuart}, there exists a unique irreducible representation~$\eta_{\Upsilon,\Upsilon'}$ extending~$\eta_{\Upsilon}$ to~$\J^1_{\Upsilon,\Upsilon'}$ such that~$\eta_{\Upsilon,\Upsilon'}$ and~$\eta_{\Upsilon}$ induce equivalent irreducible representations on~$\P_1(\Upsilon)$. 
We denote 
\[\ext(\Upsilon,\Upsilon'):=\{\text{extensions of~$\eta_{\Upsilon,\Upsilon'}$ to~$\J^{\st}_{\Upsilon,\Upsilon'}$}\},\]
and~$\ext(\Upsilon)=\{\text{extensions of~$\eta_{\Upsilon}$ to~$\J^{\st}_{\Upsilon}$}\}=\ext(\Upsilon,\Upsilon)$.

%%%%%%%%%%%%%%%%%%%%%%%%%%%%%
\subsubsection{The maximal case}
 Let~$\Gamma$ be an~${\o_{\varphi(\E)}}$-lattice sequence corresponding to a minimal parahoric subgroup~$\P(\Gamma_{{\varphi(\beta)}})^\circ$ of~$\G_{{\varphi(\beta)}}$ contained in the maximal parahoric subgroup~$\P(\Lambda_{{\varphi(\beta)}})^\circ$.  
 Then~$\J^1_{\Gamma,\Lambda}$ is a pro-$p$ Sylow subgroup of~$\J_{\Lambda}$. 
By \cite[Theorem 4.1]{St08} (the same proof works in the other cases, cf.~\cite[Proposition 6.1(i)]{SkodlerackCuspQuart}),~$\eta_{\Gamma,\Lambda}$ extends to~$\J_{\Lambda}$, and we call any extension of~$\eta_{\Gamma,\Lambda}$ to~$\J_{\Lambda}$ a \emph{beta extension} of~$\eta$ (this definition is independent of the choice of~$\Gamma$). We write
\[\bext(\Lambda):=\{\text{beta extensions of~$\eta$ to~$\J_{\Lambda}$}\}.\]
Any two beta extensions of~$\eta$ differ by a character of~$\P(\Lambda_{{\varphi(\beta)}})/\P_1(\Lambda_{{\varphi(\beta)}})$ which is trivial on the subgroup generated by all its unipotent subgroups. 
Note that~$\bext(\Lambda)$ is the set of those representations of~$\J^{\st}_\Lambda$ with restriction to~$\J^{\st}_{\Gamma,\Lambda}$ contained in~$\ext(\Gamma,\Lambda)$.
 
\begin{remark}
Our definition of the ``maximal case'' for beta extensions is more restrictive than that of \cite{St08}, \cite{SkodlerackCuspQuart} who allow non-maximal parahoric subgroups for~$\P^{\circ}(\Lambda_{{\varphi(\beta)}})$ (and non-$m$-semisimple realizations), but where the full stabilizer is a maximal compact in their definition of the ``maximal case''.  We restrict to just maximal parahoric subgroups as we wish to have stronger compatibility properties, prefer less choice between the extensions we eventually define, and for our application there is no need to extend as far as in \cite{St08},  \cite{SkodlerackCuspQuart}.
\end{remark}
 
%%%%%%%%%%%%%%%%%%%%%%%%%%%%%
\subsubsection{The non-maximal case:~compatibility} 
For~$\Lambda$ an~${\o_{\varphi(\E)}}$-lattice sequence in~$\V$, we write~${\mathfrak{a}}_0(\Lambda)$, and~${\mathfrak{b}}_0(\Lambda)$ for the associated hereditary~{$\o_\F$ and~$\o_{\varphi(\E)}$}-orders in~{$\Aut_{\D}(\V)$ and~$\Aut_{\varphi(\E)\otimes \D}(\V)$} respectively.

Suppose now that~$\Upsilon$ is a self-dual~${\o_{\varphi(\E)}}$-lattice sequence with associated parahoric subgroup~$\P^{\circ}(\Upsilon_{{\varphi(\beta)}})$ of~$\G$ contained in the parahoric~$\P^{\circ}(\Lambda_{{\varphi(\beta)}})$. 

\begin{lemma}\label{betaextsbij}
There exists a natural bijection
\[\mathbf{b}_{\Upsilon,\Lambda}:\ \ext(\Upsilon)\rightarrow\ext(\Upsilon,\Lambda)\]
which can be characterized as follows:  there exists a sequence of (self-dual)~${\o_{\varphi(\E)}}$-lattice sequences
\begin{equation}
\label{eqUpsSeq}\Upsilon=\Upsilon_0,\Upsilon_1,\ldots,\Upsilon_l=\Lambda
\end{equation}
satisfying
\begin{enumerate}
\item for all~$0<i<l$,~$\P(\Upsilon_{{\varphi(\beta)}})^\circ=\P(\Upsilon_{i,{\varphi(\beta)}})^\circ$ (which implies~$\P^{\st}(\Upsilon_{i,{\varphi(\beta)}})=\P^{\st}(\Upsilon_{{\varphi(\beta)}})$), and
\item for all~$0\leqslant i<l$,~${\mathfrak{a}}_0(\Upsilon_i)\subseteq {\mathfrak{a}}_0(\Lambda_{i+1})\text{ or }{\mathfrak{a}}_0(\Upsilon_i)\supseteq {\mathfrak{a}}_0(\Lambda_{i+1})$ (which implies~$\P_1(\Upsilon_i)\geqslant \P_1(\Upsilon_{i+1})$ or~$\P_1(\Upsilon_i)\leqslant \P_1(\Upsilon_{i+1})$ respectively),
\end{enumerate} 
and, choosing such a sequence, {for any~$\kappa_0\in\ext(\Upsilon)$} there exist unique representations~$\kappa_i\in\ext(\Upsilon_i)$, for~$i<l$, and a unique representation~$\kappa_l\in\ext(\Upsilon,\Upsilon_l)$, such that:
\begin{enumerate}
\item for~$i<l-1$,  if~${\mathfrak{a}_0}(\Upsilon_i)\subseteq{\mathfrak{a}_0}(\Upsilon_{i+1})$, then 
\[\ind_{\J^{\st}_{\Upsilon_i}}^{\P^{\st}(\Upsilon_{{\varphi(\beta)}})\P_1(\Upsilon_{i})}(\kappa_i)\simeq \ind_{\J^{\st}_{\Upsilon_{i+1}}}^{\P^{\st}(\Upsilon_{{\varphi(\beta)}})\P_1(\Upsilon_{i})}(\kappa_{i+1})\]
else if~${\mathfrak{a}_0}(\Upsilon_{i+1})\subseteq{\mathfrak{a}_0}(\Upsilon_i)$, then 
\[\ind_{\J^{\st}_{\Upsilon_i}}^{\P^{\st}(\Upsilon_{{\varphi(\beta)}})\P_1(\Upsilon_{i+1})}(\kappa_i)\simeq \ind_{\J^{\st}_{\Upsilon_{i+1}}}^{\P^{\st}(\Upsilon_{{\varphi(\beta)}})\P_1(\Upsilon_{i+1})}(\kappa_{i+1}).\]
\item the representation~$\kappa_l$ satisfies
\[\ind_{\J^{\st}_{\Upsilon_{l-1}}}^{\P^{\st}(\Upsilon_{{\varphi(\beta)}})\P_1(\Upsilon_{l-1})}(\kappa_{l-1})\simeq \ind_{\J^{\st}_{\Upsilon,\Upsilon_{l}}}^{\P^{\st}(\Upsilon_{{\varphi(\beta)}})\P_1(\Upsilon_{l-1})}(\kappa_{l})\]
\end{enumerate}
and we set~$\mathbf{b}_{\Upsilon,\Lambda}(\kappa_0)=\kappa_l$.
\end{lemma}

\begin{proof}
The proof follows the arguments given in~\cite[\S6]{SkodlerackCuspQuart} and~\cite[Lemma 4.3]{St08}, replacing their maximal condition with ours and their use of~$\J$-groups with the~$\J^{\st}$-groups we introduce in this paper.
\end{proof}

\begin{definition}
Suppose~$\P(\Lambda_{{\varphi(\beta)}})^\circ$ is a maximal parahoric subgroup of~$\G_{\varphi(\b)}$. An extension~$\kappa_{\Upsilon}\in\ext(\Upsilon)$ is called a \emph{beta extension relative to~$\Lambda$} if there exists a (maximal) beta extension~$\kappa\in\bext(\Lambda)$ such that~\[\mathbf{b}_{\Upsilon,\Lambda}(\kappa_{\Upsilon})=\kappa\mid_{\J^{\st}_{\Upsilon,\Lambda}};\]
and if this is the case we say that~$\kappa_{\Upsilon}$ and~$\kappa$ are \emph{compatible}.\end{definition}

\begin{lemma}\label{positivedepthparaindinfinitegroup}
Let~$\theta$ be an~$m$-realization of~$\mathfrak{t}$ for~$(\G,h)$ with parameterization $((\V,h),\varphi,\Lambda,\beta)$, and let~$\kappa\in\bext(\Lambda)$.  Let~$\overline{\P}$ be {the subgroup} of~$\overline{\G}=\J_{\Lambda}^{\st}/\J^1_{\Lambda}$ corresponding to {the subgroup}~$\P^{\st}(\Upsilon_{{\varphi(\beta)}})$ of~$\P^{\st}(\Lambda_{{\varphi(\beta)}})$ ({i.e.,~$\overline{\P}$ is the image of~$\J^{\st}_{\Upsilon,\Lambda}$ under the reduction map~$\J_{\Lambda}^{\st}\rightarrow \overline{\G}$}), and self-dual~${\o_{\varphi(\E)}}$-lattice sequence~$\Upsilon_{{\varphi(\beta)}}$, and set~$\overline{\M}=\J_{\Upsilon}^{\st}/\J^1_{\Upsilon}\simeq \P^{\st}(\Upsilon_{{\varphi(\beta)}})/\P_1(\Upsilon_{{\varphi(\beta)}})$, {considered as a quotient of~$\overline{\P}$}.  {Let~$\kappa_{\Upsilon}\in\ext(\Upsilon)$ be compatible with~$\kappa$.} Then, for any~$\R$-representation~$\rho$ of~$\overline{\M}$,
\[\ind_{\J^{\st}_{\Upsilon}}^{\G}(\kappa_{\Upsilon}\otimes \rho)\simeq \ind_{\J^{\st}_{\Lambda}}^{\G}(\kappa\otimes \ind_{\overline{\P}}^{\overline{\G}}(\rho)). \]
\end{lemma}

\begin{proof}
We choose a sequence of self-dual~${\o_{\varphi(\E)}}$-lattices as in \eqref{eqUpsSeq}, and~$\kappa_i$ as in Lemma \ref{betaextsbij} so that~$\kappa_l=\kappa\mid_{\J^{\st}_{\Upsilon,\Lambda}}$.  If~$i<l-1$, then as~$\kappa_{i}$ and~$\kappa_{i+1}$ induce equivalent representations of a subgroup of~$\G$, transitivity of induction along the path allows one to deduce an isomorphism between~$\ind_{\J^{\st}_{\Upsilon}}^{\G}(\kappa_{\Upsilon}\otimes \rho)$ and~$\ind_{\J^{\st}_{\Upsilon_{l-1}}}^{\G}(\kappa_{l-1}\otimes \rho)$.  Indeed, if~${\mathfrak{a}}_0(\Upsilon_i)\subseteq {\mathfrak{a}}_0(\Upsilon_{i+1})$ then we have
\begin{align*}
\ind_{\J^{\st}_{\Upsilon_i}}^{\G}(\kappa_{i}\otimes \rho)&\simeq \ind_{\P^{\st}(\Upsilon_{{\varphi(\beta)}})\P_1(\Upsilon_{i})}^{\G}(\ind_{\J^{\st}_{\Upsilon_i}}^{\P^{\st}(\Upsilon_{{\varphi(\beta)}})\P_1(\Upsilon_{i})}(\kappa_i\otimes \rho))\\
&\simeq \ind_{\P^{\st}(\Upsilon_{{\varphi(\beta)}})\P_1(\Upsilon_{i})}^{\G}(\ind_{\J^{\st}_{\Upsilon_i}}^{\P^{\st}(\Upsilon_{{\varphi(\beta)}})\P_1(\Upsilon_{i})}(\kappa_i)\otimes \widetilde{\rho})\\
&\simeq\ind_{\P^{\st}(\Upsilon_{{\varphi(\beta)}})\P_1(\Upsilon_{i})}^{\G}(\ind_{\J^{\st}_{\Upsilon_{i+1}}}^{\P^{\st}(\Upsilon_{{\varphi(\beta)}})\P_1(\Upsilon_{i})}(\kappa_{i+1})\otimes \widetilde{\rho})\\&\simeq \ind_{\J^{\st}_{\Upsilon_{i+1}}}^{\G}(\kappa_{i+1}\otimes \rho),
\end{align*}
where~$\widetilde{\rho}$ extends~$\rho$ to~$\P^{\st}(\Upsilon_{{\varphi(\beta)}})\P_1(\Upsilon_{i})$ by trivial extension to~$\P_1(\Upsilon_{i})$.  The analogous argument reversing the roles of~$i$ and~$i+1$ gives the other required isomorphism when~${\mathfrak{a}}_0(\Upsilon_i)\supseteq {\mathfrak{a}}_0({\Upsilon}_{i+1})$.

At the final step in the path, similarly (noting $\Upsilon_l=\Lambda$) we have
\begin{align*}
\ind_{\J^{\st}_{\Upsilon_{l-1}}}^{\G}(\kappa_{l-1}\otimes \rho)
&\simeq \ind_{\P^{\st}(\Upsilon_{{\varphi(\beta)}})\P_1(\Upsilon_{l-1})}^{\G}(\ind_{\J^{\st}_{\Upsilon_{l-1}}}^{\P^{\st}(\Upsilon_{{\varphi(\beta)}})\P_1(\Upsilon_{l-1})}(\kappa_{l-1})\otimes \widetilde{\rho})\\
&\simeq \ind_{\P^{\st}(\Upsilon_{{\varphi(\beta)}})\P_1(\Upsilon_{l-1})}^{\G}( \ind_{\J^{\st}_{\Upsilon,\Upsilon_l}}^{\P^{\st}(\Upsilon_{{\varphi(\beta)}})\P_1(\Upsilon_{l-1})}(\kappa)\otimes\widetilde{\rho}) \\
&\simeq \ind_{\P^{\st}(\Upsilon_{{\varphi(\beta)}})\P_1(\Upsilon_{l-1})}^{\G}( \ind_{\J^{\st}_{\Upsilon,\Upsilon_l}}^{\P^{\st}(\Upsilon_{{\varphi(\beta)}})\P_1(\Upsilon_{l-1})}(\kappa\otimes\rho)) \\
&\simeq \ind_{\J^{\st}_{\Lambda}}^{\G}( \ind_{\J^{\st}_{\Upsilon,\Upsilon_l}}^{\J^{\st}_{\Lambda}}(\kappa\otimes\rho))\simeq  \ind_{\J^{\st}_{\Lambda}}^{\G}(\kappa\otimes \ind_{\overline{\P}}^{\overline{\G}}(\rho)),
\end{align*}
where~$\widetilde{\rho}$ extends~$\rho$ to~$\P^{\st}(\Upsilon_{{\varphi(\beta)}})\P_1(\Upsilon_{l-1})$.  The composition of all these isomorphisms gives the required statement.
\end{proof}

%%%%%%%%%%%%%%%%%%%%%%%%%%%%%
\subsubsection{Compatible families of beta extensions}
In order to make reductions to the depth zero situation, we will need compatibility relations between {a whole family of} beta extensions.

\begin{definition}\label{def:compatfamilyofbetas}
Let~$\theta$ be an~$m$-realization of~$\mathfrak{t}$ for~$(\G,h)$ with parametrization $((\V,h),\varphi,\Lambda,\beta)$.  Let~$\mathcal{C}$ be a chamber in the building of~$\G_{\varphi(\beta)}$ and let~$\{\Lambda_j\}$ denote a set of~${\o_{\varphi(\E)}}$-lattice sequences which form a complete set of representatives for the~$\G_{\varphi(\beta)}$-conjugacy classes of the vertices in the closure of~$\mathcal{C}$. % (without repetition). 
Let~$\{\kappa_j\}$ be a collection of beta extensions of~$\{\theta_j=\tau_{\Lambda,\Lambda_j,\beta}(\theta)\}$.
\begin{enumerate}
\item\label{def:compatfamilyofbetas.i} 
We call~$\{\kappa_j\}$ a \emph{compatible family} of beta extensions if, whenever $\Upsilon,\Upsilon'$ are an~${\o_{\varphi(\E)}}$-lattice sequences in~$\mathcal{C}$ for which there is~$g\in\G_{\varphi(\beta)}$ with~$g\cdot \Upsilon=\Upsilon'$, then for any pair~$\Lambda_i,\Lambda_j$ in our chosen set of representatives such that~$\P({\Upsilon_{{\varphi(\beta)}}})^{\circ}\subseteq \P({\Lambda_{i,{\varphi(\beta)}}})^{\circ}$ and~$\P({\Upsilon'_{{\varphi(\beta)}}})^{\circ}\subseteq \P({\Lambda_{j,{\varphi(\beta)}}})^{\circ}$, we have
\[
\mathbf{b}_{\Upsilon',\Lambda_j}^{-1}(\kappa_j\mid_{\J^{\st}_{\Upsilon',\Lambda_j}})
\simeq
\left(\mathbf{b}_{\Upsilon,\Lambda_i}^{-1}(\kappa_i\mid_{\J^{\st}_{\Upsilon,\Lambda_i}})\right)^g.
\]
Then, by conjugation, a compatible family of beta extensions defines a beta extension at every point in the building of~$\G_{\varphi(\beta)}$.
\item We say that~$\{\kappa_j\}$ has \emph{full intertwining} if $\I_{\G}(\kappa_i,\kappa_j)\supseteq\G_{\varphi(\beta)}$, for all~$i,j$.
\end{enumerate}
\end{definition}

The notions of compatible family and full intertwining are in fact related.

\begin{proposition}\label{Fullimpcomp}
If~$\{\kappa_j\}$  has full intertwining then it is a compatible family.
\end{proposition}

We will prove this below. We also make the following

\begin{conjecture}\label{Conjcompatfam}
There exists a family~$\{\kappa_j\}$ of beta extensions with full intertwining.
\end{conjecture}

\begin{remark}
\begin{enumerate}
\item For inner forms of general linear groups the existence of a compatible family of beta extensions is straightforward as there is a unique class of maximal parahoric subgroup, and full intertwining can be shown in this case following the original method of \cite{BK93} for beta extensions of simple characters of~$\GL_n(\F)$.  The details will appear in the forthcoming work of \cite{SkodlerackYe}.  It is currently not known if compatible families exist for inner forms of classical groups with~$p\neq 2$ or if there are always beta extensions with full intertwining (though we expect many cases will be covered in forthcoming work of \cite{SkodlerackYe}).
  \item As further evidence towards Conjecture \ref{Conjcompatfam}, in the tame setting, rephrasing into the Bushnell--Kutzko language of this paper, Fintzen--Kaletha--Spice in \cite{FinKalSpi} construct canonical (maximal) beta extensions of the Heisenberg representations they consider and prove that they have full intertwining.
  \end{enumerate}
\end{remark}

It remains to prove Proposition~\ref{Fullimpcomp}. The main idea is to see that intertwining by elements of~$\G_\beta$ is preserved under compatibility of beta extensions, which is an asymmetric version of~\cite[Lemma~4.3]{St08}. We thus begin with the following lemma.

\begin{lemma}\label{LemCompatibleintertwining} Suppose~$\theta$ is a realization of~$\mathfrak{t}$ for~$(\G,h)$ with parametrization $((\V,h),\varphi,\Lambda,\beta)$. Let~$\Upsilon,\Lambda',\Upsilon'$ be~${\o_{\varphi(\E)}}$-lattice sequences with~$\P({\Upsilon_{{\varphi(\beta)}}})^{\circ}\subseteq \P({\Lambda_{{\varphi(\beta)}}})^{\circ}$ and~$\P({\Upsilon'_{{\varphi(\beta)}}})^{\circ}\subseteq \P({\Lambda'_{{\varphi(\beta)}}})^{\circ}$, and let~$\theta'=\tau_{\Lambda,\Lambda',\beta}(\theta)$. Let~$\kappa,\kappa'$ be beta extensions of~$\theta,\theta'$ respectively, {and} set~$\kappa_\Upsilon=\mathbf{b}_{\Upsilon,\Lambda}^{-1}(\kappa\mid_{\J^{\st}_{\Upsilon,\Lambda}})$ and~$\kappa'_{\Upsilon'}=\mathbf{b}_{\Upsilon',\Lambda'}^{-1}(\kappa'\mid_{\J^{\st}_{\Upsilon',\Lambda'}})$. Let~$g\in\G_{\varphi(\beta)}$.
\begin{enumerate}
\item\label{LemCompatibleintertwining.i} If~$\P(\Upsilon)\subseteq \P(\Lambda),\ \P(\Upsilon')\subseteq \P(\Lambda')$  and~$g\in\I_{\G}(\kappa,\kappa')$ then also~$g\in\I_{\G}(\kappa_\Upsilon,\kappa'_{\Upsilon'})$.
\item If~$\P({\Upsilon_{{\varphi(\beta)}}})^{\circ}= \P({\Lambda_{{\varphi(\beta)}}})^{\circ}$ and~$\P({\Upsilon'_{{\varphi(\beta)}}})^{\circ}= \P({\Lambda'_{{\varphi(\beta)}}})^{\circ}$ then~$g\in\I_{\G}(\kappa,\kappa')$ if and only if~$g\in\I_{\G}(\kappa_\Upsilon,\kappa'_{\Upsilon'})$.
\item\label{LemCompatibleintertwining.iii} If~$g\in\I_{\G}(\kappa,\kappa')$ then also~$g\in\I_{\G}(\kappa_\Upsilon,\kappa'_{\Upsilon'})$.
\end{enumerate}
\end{lemma}

Of course, case~\ref{LemCompatibleintertwining.iii} supsersedes case~\ref{LemCompatibleintertwining.i} but the proof is achieved in this order.

\begin{proof}
(i) Suppose~$g\in\I_{\G}(\kappa,\kappa')$ so certainly~$g$ intertwines their restrictions and, by~\cite[Lemma~2.2]{RKSS}, we have that
\[
\text{$g$ intertwines~$\Ind_{\J^{\st}_{\Upsilon,\Lambda}}^{\P^{\st}(\Upsilon_{{\varphi(\beta)}})\P_1(\Lambda)} (\kappa\mid_{\J^{\st}_{\Upsilon,\Lambda}})$ with~$\Ind_{\J^{\st}_{\Upsilon',\Lambda'}}^{\P^{\st}(\Upsilon'_{{\varphi(\beta)}})\P_1(\Lambda')} (\kappa\mid_{\J^{\st}_{\Upsilon',\Lambda'}})$.}
\]
By definition of compatibility of beta extensions, these representations are precisely~$\Ind_{\J^{\st}_\Upsilon}^{\P^{\st}(\Upsilon_{{\varphi(\beta)}})\P_1(\Lambda)}\kappa_\Upsilon$ and~$\Ind_{\J^{\st}_{\Upsilon'}}^{\P^{\st}(\Upsilon'_{{\varphi(\beta)}})\P_1(\Lambda')}\kappa'_{\Upsilon'}$, and are irreducible. Thus, by the support condition in \emph{loc.\ cit.}, there are then~$u\in\P^{\st}(\Upsilon_{{\varphi(\beta)}})$, $k\in\P_1(\Lambda)$ and~$u'\in\P^{\st}(\Upsilon'_{{\varphi(\beta)}})$, $k'\in\P_1(\Lambda')$ such that
\[
\text{$kugu'k'$ intertwines~$\kappa_\Upsilon$ with~$\kappa'_{\Upsilon'}$.}
\]
Now~$\I_{\G}(\kappa_\Upsilon,\kappa'_{\Upsilon'})\subseteq\I_{\G}(\theta_\Upsilon,\theta_{\Upsilon'}) = \J^1_{\Upsilon}\G_{\varphi(\beta)}\J^1_{\Upsilon'}$ so there exist~$h\in\J^1_{\Upsilon}$ and~$h'\in\J^1_{\Upsilon'}$ such that
\[
\text{$hkugu'k'h'\in\G_{\varphi(\beta)}$ intertwines~$\kappa_\Upsilon$ with~$\kappa'_{\Upsilon'}$.}
\]
But then~$hkugu'k'h'\in\P_1(\Upsilon)ugu'\P_1(\Upsilon')\cap\G_{\varphi(\beta)}$, which is~$\P_1(\Upsilon_{{\varphi(\beta)}})ugu'\P_1(\Upsilon'_{{\varphi(\beta)}})$, by~\cite[Lemma~4.6, Corollary~4.7]{RKSS} (see also~\cite[Proposition~4.8]{SkodInnerFormII}). Thus there are~$s\in\P_1(\Upsilon_{{\varphi(\beta)}})$ and~$s'\in\P_1(\Upsilon'_{{\varphi(\beta)}})$ such that~$sugu's'$ intertwines~$\kappa_\Upsilon$ with~$\kappa'_{\Upsilon'}$. Thus~$g\in\I_{\G}(\kappa_\Upsilon,\kappa'_{\Upsilon'})$, as required.

For~(ii), we note first that the argument in~(i) is reversible when~$\P({\Upsilon_{{\varphi(\beta)}}})^{\circ}= \P({\Lambda_{{\varphi(\beta)}}})^{\circ}$ and~$\P({\Upsilon'_{{\varphi(\beta)}}})^{\circ}= \P({\Lambda'_{{\varphi(\beta)}}})^{\circ}$, since then~$\J^{\st}_{\Upsilon,\Lambda}=\J^{\st}_{\Lambda}$ and~$\J^{\st}_{\Upsilon',\Lambda'}=\J^{\st}_{\Lambda'}$. Now, as in the proof of {Theorem~\ref{thm43}}, we cut up the line segments in the building from~$\Upsilon$ to~$\Lambda$ and from~$\Upsilon'$ to~$\Lambda'$ into sufficiently small pieces so that the hypotheses in~(i) are satisfied, and the result follows from~(i).

Finally,~(iii) follows from the previous parts by choosing~$\Upsilon_1,\Upsilon_1'$ auxiliary points in the chamber~$\mathcal{C}$ such that the hypotheses of~(i) are satisfied when~$\Upsilon,\Upsilon'$ are replaced by~$\Upsilon_1,\Upsilon_1'$, while the hypothesis of~(ii) are satisfied when~$\Lambda,\Lambda'$ are replaced by~$\Upsilon_1,\Upsilon_1'$.
\end{proof}

\begin{proof}[Proof of Proposition~\ref{Fullimpcomp}]
Suppose we are in the situation of Definition~\ref{def:compatfamilyofbetas}\ref{def:compatfamilyofbetas.i}, with~$g\in\G_{\varphi(\beta)}$ such that~$g\cdot\Upsilon=\Upsilon'$. Since the family~$\{\kappa_j\}$ has full intertwining, we certainly have~$g\in\I_{\G}(\kappa_i,\kappa_j)$ for each~$i,j$. But then Lemma~\ref{LemCompatibleintertwining}\ref{LemCompatibleintertwining.iii} implies that~$g$ also intertwines~$\mathbf{b}_{\Upsilon,\Lambda_i}^{-1}(\kappa_i\mid_{\J^{\st}_{\Upsilon,\Lambda_i}})$ with~$\mathbf{b}_{\Upsilon',\Lambda_j}^{-1}(\kappa_j\mid_{\J^{\st}_{\Upsilon',\Lambda_j}})$, and hence~$1$ intertwines the latter with~$\left(\mathbf{b}_{\Upsilon,\Lambda_i}^{-1}(\kappa_i\mid_{\J^{\st}_{\Upsilon,\Lambda_i}})\right)^g$. But these are representations of the same group, so being intertwined by~$1$ is the same as being isomorphic.
\end{proof}

%%%%%%%%%%%%%%%%%%%%%%%%%%%%%
\subsection{Types for Bernstein blocks}\label{typesoverC}
We write~$\mathfrak{B}_{\R}(\G)$ for the set of inertial classes of (super)cuspidal supports for~$\G$, if~$\R$ has characteristic zero.    
\begin{definition}
Suppose~$\R$ is an algebraically closed field of characteristic zero.   Let~$\mathfrak{s}\in\mathfrak{B}_{\R}(\G)$.  A pair~$(\U,\Sigma)$, with $\U$ a compact open subgroup of~$\G$, and~$\Sigma$ an irreducible representation of~$\U$, is called an~\emph{$\mathfrak{s}$-type} if~$\ind_{\U}^{\G}(\Sigma)$ is a (finitely generated projective) generator of~$\Rep_{\R}(\mathfrak{s})$.
\end{definition}

For classical~$p$-adic groups,~$\GL_m(\D)$, and quarternionic forms of classical groups, we have a construction of types for Bernstein blocks of Miyauchi--Stevens~\cite{MiSt}, S\'echerre--Stevens~\cite{SecherreStevensVI}, Skodlerack--Ye~\cite{SkodlerackYe}.%, and in depth zero for an arbitrary connected reductive group, Morris in~\cite{Morris} has constructed types for Bernstein blocks.  

Let~$\mathfrak{s}\in\mathfrak{B}_{\R}(\G)$ with representative~$(\M,\rho)$.  This determines a supercuspidal inertial class~$\mathfrak{s}_{\M}\in \mathfrak{B}_{\R}(\M)$ with representative~$(\M,\rho)$. 

A Levi subgroup~$\M$ in~$\G$ decomposes as a product~$\M=\prod \M_i$ of (inner forms) of general linear groups, or of (inner forms) of general linear groups and (an inner form) of a classical group.  We define an (m-semisimple) semisimple %stratum in
{character for~$\M$, to be %a direct sum 
the tensor product of (m-semisimple) semisimple %strata 
characters in the corresponding~$\M_i$, and define the associated groups, %characters, 
Heisenberg representations, and beta extensions %,  associated to a stratum 
in~$\M$ by taking the appropriate product or tensor product over~$i$.}

 The construction of cuspidal representations has been extended to all algebraically closed fields of characteristic~$\ell\neq p$ and we have:

\begin{theorem}[{%Depth zero \cite{Morris, Vigbarcelona}, 
Classical groups \cite{St08, RKSS}, $\GL_m(\D)$ \cite{SecherreStevensIV,MinSec}, inner forms of classical groups~\cite{SkodlerackCuspQuart}}]\label{thm:covers}
Let~$\rho$ be an irreducible cuspidal~$\R$-representation of~$\M$.  There exist
\begin{enumerate}
\item an $m$-semisimple %stratum~$[\Lambda,n,0,\varphi(\beta)]$ 
{character~$\theta_\M\in\Cc(\Lambda,\varphi(\beta))$} for~$\M$, and a beta extension~{$\kappa_\M$} to~$\J_\M(\varphi(\beta),\Lambda)$; % of a semisimple character for~$[\Lambda,n,0,\varphi(\beta)]$; 
%\item in depth zero we set~$\E=\F$, and let~$\kappa$ be the trivial character of a maximal parahoric subgroup~$\M_{\Lambda_{{\varphi(\beta)}}}$ which we denote by~$\J_\M(0,\Lambda)$ with pro-$p$ unipotent radical~$\J_\M^1(0,\Lambda)$ for uniformity; 
\item an irreducible representation~${\tau}_{\M}$ of~$\P(\Lambda_{{\varphi(\beta)}})/\P_1(\Lambda_{{\varphi(\beta)}})$ with cuspidal restriction {to}~$\P(\Lambda_{{\varphi(\beta)}})^\circ/\P_1(\Lambda_{{\varphi(\beta)}})$; and
\item an extension~$\widetilde{\Sigma}_\M$ of~$\Sigma_{\M}=\kappa_{\M}\otimes{\tau}_{\M}$ to~$\N_{\M}(\Sigma_{\M})$, the normalizer of $\Sigma_{\M}$ in~$\M$, 
\end{enumerate}
such that, $\rho \simeq \ind_{\N_{\M}(\Sigma_{\M})}^{\M}(\widetilde{\Sigma}_\M)$.  Moreover, if~$\ell=0$, then~$(\J_\M(\varphi(\beta),\Lambda),\Sigma_{\M})$ is an~$\mathfrak{s}_{\M}$-type for~$\mathfrak{s}_{\M}=[\M,\rho]_{\M}$.
\end{theorem}

The construction of covers allows one to construct an $\mathfrak{s}$-type, as we now explain again working in the broader setting of algebraically closed fields of characteristic~$\ell\neq p$:

Let~$\P=\M\N$ be a parabolic subgroup of~$\G$.  Let~$\rho$ be an irreducible cuspidal~$\R$-representation of~$\M$.  Then we can choose a semisimple %stratum~$\Delta=[\Lambda,n,0,\varphi(\beta)]$ for~$\G$ and a semisimple 
character~$\theta\in\Cc({\varphi(\beta)},\Lambda)$ such that the decomposition of~$\V$ associated to~$\M$ is {\emph{properly subordinate to the underlying stratum} (see~\cite[\S8]{SkodlerackCuspQuart} for the definitions)} %~$\Delta$, 
and~$\theta\mid_{\H^1(\varphi(\beta),\Lambda)\cap \M}$ is an $m$-semisimple character contained in~$\rho$.  {(Note that this is a stronger requirement than the statement of Theorem~\ref{thm:covers} but is in fact what is proved in the references.)} We set
\[\J_\P=\J_{\P}(\varphi(\beta),\Lambda)=\H^1(\varphi(\beta),\Lambda)(\J^{\st}(\varphi(\beta),\Lambda)\cap \P).\]
From Appendix \ref{AppendixC} Propositions \ref{propJPsteqJP} and \ref{propJPst} (see also Remark~\ref{remJPst}), or \cite{Morris} in the depth zero case,
\[\J_\P=\H^1(\varphi(\beta),\Lambda)(\J(\varphi(\beta),\Lambda)\cap \P),\]
and hence agrees with the~$\J_{\P}$ group considered in \cite{MiSt}, \cite{SkodlerackYe}. 

Let~$\kappa$ be a beta extension to~$ \J^{\st}(\varphi(\beta),\Lambda)$.  We form the natural representation~$\kappa_{\P}$ of~$\J_{\P}$ on the space of~$(\J^{\st}(\varphi(\beta),\Lambda)\cap \U)$-fixed vectors in~$\kappa$.  Then~$\kappa_{\P}$ extends~$\eta_\P$ and~$\ind_{\J_{\P}}^{\J^{\st}(\varphi(\beta),\Lambda)}(\kappa_{\P})\simeq \kappa$.

\begin{theorem}[{%Depth zero \cite{Morris}, 
Classical groups \cite{MiSt}, $\GL_m(\D)$ \cite{SecherreStevensVI}, inner forms of classical groups~\cite{SkodlerackYe}}]\label{Gcoverstheorem}
Under the above notation, writing~$\Sigma_{\P}=\kappa_{\P}\otimes{\tau}_{\P}$, we have 
\begin{enumerate}
\item $(\J_{\P},\Sigma_{\P})$ is a~$\G$-cover of~$(\J_{\M},\Sigma_{\M})$ relative to~$\P$.
\item $(\J_{\P},\Sigma_{\P})$ is a~$\mathfrak{s}$-type, if~$l=0$.
\end{enumerate}\end{theorem}

\section{Cuspidal endo-support of an irreducible representation} \label{secEndoSupport}
{In~\S\ref{SeccuspsuppEP} we introduced the notion of \emph{cuspidal support} for endo-parameters of~$\G$. In this section, we explain how the cuspidal support of the endo-parameter of an irreducible representation~$\pi$ of~$\G$ relates to the endo-parameter of the (inertial) cuspidal support of~$\pi$. To do so, we first introduce the notion of a \emph{sub-endo-parameter for~$\G$}.}
%In this section we introduce the analogue of cuspidal support for the arithmetic core of an irreducible representation, the cuspidal endo-support. For that we need to modify the support of an endo-parameter~$(\G,\ft)$ in the special orthogonal case, because in the special orthogonal case we need to make sure that the split rank of the center of~$\M_{\varphi(\beta)}$ is not greater than the split rank of the center of~$\M$. This modified support is the cuspidal support of~$(\G,\ft)$, see \S\ref{SeccuspsuppEP}. Test~$\cSE,\cE$.

\begin{definition}\label{defSubEndopar}
{
Let~$\M$ be a Levi subgroup of~$\G$. A \emph{sub-endo-parameter for~$\M$} is the~$\M$-conjugacy class~$(\M',\fs')_\M$ of a pair consisting of a Levi subgroup~$\M'$ of~$\M$ and an endo-parameter~$\fs'$ for~$\M'$. We write~$\cSE(\M)$ for the set of all sub-endo-parameters for~$\M$.
}

{
If~$(\M,\fs)_\G$ and~$(\M',\fs')_\G$ are two sub-endo-parameters for~$\G$, we say that~$(\M',\fs')_\G$ is a \emph{sub-endo-parameter of~$(\M,\fs)_\G$} if there exists~$g\in\G$ such that~$\M$ contains~${}^g\M'$ and~$\fs=\ind_{\M'}^\M{}^g\fs'$ (which is clearly independent of the choice of representatives for the classes). This defines a partial order on~$\cSE(\G)$ and we write~$(\M',\fs')_\G\leq (\M,\fs)_\G$. 
}

{
For~$\ft$ an endo-parameter for~$\G$, we write~$\cSE(\G,\ft)$ for the set of sub-endo-parameters~$(\M,\fs)_\G$ such that~$(\M,\fs)_\G\leq (\G,\ft)_\G$. 
}
\end{definition}

{Recall that, for~$\ft$ an endo-parameter for~$\G$, we defined the \emph{cuspidal support}~$\cusp(\ft):=(\M_c,\ft_{\M_c})_\G$ of~$\ft$ in~\S\ref{SeccuspsuppEP}; thus~$\cusp(\ft)$ is a sub-endo-parameter for~$\G$ in~$\cSE(\G,\ft)$. If~$\pi$ is an irreducible~$\R$-representation of~$\G$ with endo-parameter~$\endo(\pi):=\ft$ then we also call~$\cusp(\ft)$ the~\emph{cuspidal endo-support} of~$\pi$.}

%\begin{definition}\label{defEndoSupport}
%Let~$\pi$ be an irreducible~$\R$-representation of~$\G$ with endo-parameter~$(\G,\ft)$ and we fix Levi subgroups~$\M\in \M(\ft)$ and~$\M_c\in \M_c(\ft)$. The~$\G$-conjugacy class of~$(\M,\ft_\M)$ is called the~\emph{endo-support} of~$\pi$,  and the~$\G$-conjugacy class of~$(\M_c,\ft_{\M_c})$ is called the~\emph{cuspidal endo-support} of~$\pi$.  
%\end{definition}

%We will observe in the remaining part that cuspidal support of an irreducible representation can be used to obtain its cuspidal endo-support.   
%Except of Proposition~\ref{propBelowConstrofEndoSuppInCCcase} there is no restriction on the characteristic of~$\R$. 

\begin{definition}\label{defCuspEndopar}
{Let~$\M$ be a Levi subgroup of~$\G$ and let~$\fs$ be an endo-parameter for~$\M$. We say that~$\fs$ is \emph{cuspidal} if there is an irreducible cuspidal complex representation of~$\M$ with endo-parameter~$\fs$. We also then say that the sub-endo-parameter~$(\M,\fs)_\G$ is cuspidal.}

{We write~$\cSE^\circ(\G)$ for the set of cuspidal sub-endo-parameters for~$\G$, and similarly~$\cSE^\circ(\G,\ft)$ for the cuspidal sub-endo-parameters of~$(\G,\ft)_\G$.}
%$(\M,\fs)$ be a Levi subgroup together with an endo-parameter on~$\M$. The endo-parameter~$\fs$ is called \emph{cuspidal} if there is an irreducible cuspidal complex representation of~$\M$ of type~$\fs$. 
\end{definition}

\begin{remark}\label{remCuspEndopar}
{If~$\ft$ is an endo-parameter for~$\G$, then the cuspidal support~$(\M_c,\ft_{\M_c})_\G$ of~$\ft$ is a cuspidal sub-endo-parameter for~$\G$. Indeed, if~$\theta_{\M_c}\in\Cc_h(\Lambda,\varphi(\beta))$ is an~$m$-semisimple character for~$\M_c$ with endo-parameter~$\ft_{\M_c}$ and~$\kappa_{\M_c}$ is a beta extension to~$\J_{\M_c}:=\J_{\M_c}(\varphi(\beta),\Lambda)$ then, for~$\tau_{\M_c}$ any irreducible representation of~$\P(\Lambda_{\varphi(\beta)})/\P_1(\Lambda_{\varphi(\beta)})$ with cuspidal restriction to~$\P(\Lambda_{\varphi(\beta)})^\circ/\P_1(\Lambda_{\varphi(\beta)})$, any irreducible subquotient of~$\ind_{\J_{\M_c}}^{\M_c}(\kappa_{\M_c}\otimes\tau_{\M_c})$ is cuspidal with endo-parameter~$\ft_{\M_c}$.}
\end{remark}

{In fact the cuspidal support of an endo-parameter~$\ft$ is not merely any cuspidal sub-endo-parameter, but is maximal with respect to the partial order~$\leq$:}

\begin{proposition}\label{propEndosuppbyMaxcuspEP}
{Let~$\ft$ be an endo-parameter for~$\G$. Then~$\cusp(\ft)$ is maximal in~$\cSE^\circ(\G)$ and is the unique maximal element of~$\cSE^\circ(\G,\ft)$.}
%Let~$\pi$ be an irreducible~$\R$-representation of~$\G$ with endo-parameter~$(\G,\ft)$ and~$(\M,\fs)$ be a~$\leq$-maximal cuspidal sub-endo-parameter of~$(\G,\ft)$, i.e., maximal among all cuspidal sub-endo-parameters of~$(\G,\ft)$. 
%Then~$(\M,\fs)$ lies in the cuspidal endo-support of~$\pi$.  
\end{proposition}

\begin{proof}
{
Let~$(\M,\fs)_\G\in\cSE^\circ(\G,\ft)$ so we need to show that~$(\M,\fs)_\G \le (\M_c,\ft_{\M_c})_\G$. Let~$\t\in\Cc_h(\varphi(\b),\La)$ be a semisimple character such that~$\t|_{\M}$ is a realization of~$\fs$; by the construction of such a character (see~\cite[Proposition 5.1]{MiSt}, \cite[Theorems 6.6 and 6.10]{SkodInnerFormII}) and the fact that~$\fs$ is cuspidal, we may assume that~$\M\subseteq\M(\varphi(\b),\G)$. Since the Levi subgroups~$\M(\varphi(\b),\G)$ and~$\M_c$ differ at most in the classical factor of~$\M(\varphi(\b),\G)$, we immediately reduce to the case that~$\G_{\varphi(\b)}$ is classical, i.e.~$\I=\I_0$, which implies~$\M(\varphi(\b),\G)$ is equal to~$\G$.
}

{
If the quotient of centres~$\C(\G_{\varphi(\b)})/\C(\G)$ is compact, then~$\M(\varphi(\b),\G)=\M_c=\G$ so~$\ft$ is itself cuspidal and we get~$\fs=\ft$ by maximality. So suppose~$\C(\G_{\varphi(\b)})/\C(\G)$ is not compact, in which case~$\b$ is non-zero and there is a unique index~$i_0\in\I$ such that~$\b_{i_0}=0$ and the form $h_{i_0}=h|_{\V_{i_0}}$ is two-dimensional orthogonal with Witt index~$1$. Moreover, writing~$1_{i_0}$ for the idempotent in~$\F[\varphi(\beta)]$ projecting onto~$\V_{i_0}$, there is a unique (up to labelling) decomposition~$1_{i_0}=1_++1_-$ into idempotents which are swapped by the involution~$\s_h$. Then~$\M_c$ is the subgroup of~$\G$ stabilizing the decomposition~$1=1_++(1-1_{i_0})+1_-$.
}

{
Now let~$\tM$ be the minimal Levi subgroup of~$\tG$ containing~$\M$ and let~$e$ be the unique simple central idempotent in the Lie algebra of~$\tM$ such that~$e1_{i_0}\ne 0$. Note that~$e$ commutes with~$\varphi(\b)$ so it also commutes with~$1_{i_0}$. Suppose first that~$e1_{i_0}= 1_{i_0}$. Then also~$\sigma_h(e)1_{i_0}=1_{i_0}$ so~$e,\sigma_h(e)$ are not orthogonal and we must have~$\sigma_h(e)=e$. If~$e\ne 1_{i_0}$ then the centre of the classical factor of~$\M$ is compact, while the classical factor of~$\M_{\varphi(\beta)}$ has~$\SO(1,1)(\F)$ as a component, so~$\C(\M_{\varphi(\b)})/\C(\M)$ is non-compact, which contradicts the cuspidality of~$\fs$. Thus~$e=1_{i_0}$, but then~$\M$ has classical factor~$\SO(1,1)(\F)$ and~$e$ is not simple, a contradiction.
}

{
Thus we have~$e1_{i_0}\ne 1_{i_0}$ and there is a unique further simple central idempotent~$f$ such that~$1_{i_0}=e1_{i_0}+f1_{i_0}$. But then we also have~$1_{i_0}=\sigma_h(1_{i_0})=\sigma_h(e)1_{i_0}+\sigma_h(f)1_{i_0}$ so by uniqueness, we have~$\{e,f\}=\{\sigma_h(e),\sigma_h(f)\}$. Since there is at most one simple central idempotent fixed by~$\sigma_h$ (the one corresponding to the classical factor of~$\M$) we deduce that~$\sigma_h(e)=f\ne e$. In particular,~$e$ corresponds to a general linear block of the Levi subgroup~$\M$ so that~$e\varphi(\beta)$ is a \emph{simple} element and there is a unique~$i\in\I$ such that~$e1_i=e$. Since~$e1_{i_0}\ne 0$, we deduce that~$e=e1_{i_0}$, and similarly~$\sigma_h(e)=\sigma_h(e)1_{i_0}$. Thus we have the decomposition~$1_{i_0}=e+\sigma_h(e)$ as a sum of idempotents switched by~$\sigma_h$ so, by uniqueness, they are~$1_+,1_-$. Thus~$\M$ is indeed contained in~$\M_c$, as required.
}
\end{proof}

{We can now generalize Remark~\ref{remCuspEndopar} to deduce that every maximal cuspidal sub-endo-parameter of~$\G$ arises as the cuspidal endo-support of some irreducible representation.}

\begin{corollary}\label{corAllMaxcuspSEPareendosupp}
{Suppose~$(\M,\fs)_\G\in\cSE^\circ(\G)$ is maximal. Then there is an irreducible~$\R$-representation~$\pi$ of~$\G$ such that~$\cusp(\endo(\pi))=(\M,\fs)_\G$.} 
%Let~$(\M,\fs)$ be a maximal cuspidal sub-endo-parameter for~$\G$. Then there is an irreducible~$\R$-representation~$\pi$ of~$\G$ such that~$(\M,\fs)$ lies in the cuspidal endo-support of~$\pi$. 
\end{corollary}

\begin{proof}
{Let~$\ft=\ind_\M^\G\fs$ so that~$(\M,\fs)_\G=\cusp(\ft)$ and is the unique maximal element in~$\cSE^\circ(\G,\ft)$, by Proposition~\ref{propEndosuppbyMaxcuspEP}. Now let~$\theta\in\Cc_h(\varphi(\beta),\Lambda)$ be a realization of~$\ft$ and take any irreducible quotient~$\pi$ of the representation of~$\G$ compactly induced from~$\theta$. Then~$\pi$ contains~$\theta$ so that~$\endo(\pi)=\ft$ and~$\cusp(\endo(\pi))=(\M,\fs)_\G$.}
%Let~$(\G,\ft )$ be the endo-parameter of~$\G$ induced from~$(\M,\fs )$ and let~${\theta}\in\Cc_-(\varphi(\beta),\Lambda)$ be a realization of~$\ft$.  We can take an irreducible quotient~$\pi$ of the representation of~$\G$ compactly induced from~${\theta}$. Then~$\pi$ contains~${\theta}$. { Now apply Proposition~\ref{propEndosuppbyMaxcuspEP} to obtain the assertion.}
\end{proof}

{Finally, we arrive at the main result of this section. For~$\pi$ an irreducible~$\R$-representation of~$\G$, we write~$\cusp(\pi)$ for its inertial cuspidal support. Note that the endo-parameter of the inertial cuspidal support is a well-defined sub-endo-parameter of~$\G$.} 

\begin{theorem}\label{thmBelowConstrofEndoSuppInCCcase}
{Let~$\pi$ be an irreducible $\R$-representation of~$\G$.
\begin{enumerate}
\item The endo-parameter of the inertial cuspidal support of~$\G$ is a sub-endo-parameter of the cuspidal endo-support of~$\pi$:
\[
\endo(\cusp(\pi)) \le \cusp(\endo(\pi)).
\]
\item Moreover,~$\cusp(\endo(\pi))$ is the unique maximal element of~$\cSE^\circ(\G)$ which is greater than or equal to~$\endo(\cusp(\pi))$.
\end{enumerate}
}
%Let~$\pi$ be an irreducible $\R$-representation of~$\G$ with cuspidal support given by the~$\G$-conjugacy class of~$(\M,\rho)$. Let~$(\M,\fs)$ be the endo-parameter of~$(\M,\rho)$ and~$(\M',\fs')$ be in the cuspidal endo-support of~$\pi$. {Suppose that $\ft=\ind_\M^G\fs$ is the endo-parameter for~$\pi$.} Then~$(\M,\fs)$ is a sub-endo-parameter of a~$\G$-conjugate of~$(\M',\fs')$. 
\end{theorem}

\begin{proof}
{Write~$(\M,\fs)_\G=\endo(\cusp(\pi))$ and~$\ft=\ind_\M^\G\fs$. Then~$\endo(\pi)=\ft$ by the compatibility of endo-parameter with parabolic induction (see~\cite[Theorem~5.9(ii)]{HKSSBlocks}). Then~$(\M,\fs)_\G\in\cSE^\circ(\G,\ft)$, while~$\cusp(\endo(\ft))$ is its maximal element, by Proposition~\ref{propEndosuppbyMaxcuspEP}.
}
%{Take a maximal cuspidal sub-endo-parameter~$(\M'',\fs'')$ of~$(\G,\ft)$, such that~$(\M,\fs)\leq (\M'',\fs'')$. Then, by Proposition~\ref{propEndosuppbyMaxcuspEP}, ~$(\M'',\fs'')$ is~$\G$-conjugate to~$(\M',\fs')$.}
%Refer to \cite{HKSSBlocks}
\end{proof}

\begin{remark}
{In fact, given~$\ft$ an endo-parameter for~$\G$, the set~$\cSE^\circ(\G,\ft)$ also has a unique minimal element~$\min(\ft)$ so that we have
\[
\min(\endo(\pi)) \le \endo(\cusp(\pi)) \le \cusp(\endo(\pi)),
\]
for any irreducible~$\R$-representation~$\pi$ of~$\G$.
This minimal element~$\min(\ft)$ can be described as follows. Write~$\ft=\sum_{t\in(\cE(\F)/\Sigma)}m_{t}(t,w_t)$, where $m_{t}\in\mathbb{Z}^{\geqslant 0}$ and~$w_t\in \mathcal{W}_{(\overline{\phantom{ll}})_\D,\e}(t)$, as in Definition~\ref{defEndoparG}. For each~$t\in(\cE(\F)/\Sigma)$, set
\[
d_t:= \frac{\deg(t)}{\gcd(d,\deg(t))}\qquad\text{and}\qquad r_t:= \frac{m_t}{d\,/\gcd(d,\deg(t))},
\]
so that~$r_td_td=m_t\deg(t)$. Then there is (up to conjugacy) a unique Levi subgroup~$\M_p$ of~$\G$ such that
\[
\M_p \simeq \prod_{t\in(\cE(\F)/\Sigma)} \GL_{d_t}(\D)^{r_t} \times \G_0,
\]
for some classical group~$\G_0$. In fact,~$\G_0$ is the isometry group of the unique form~$h_0$ (up to isometry) on a space of dimension~$\dim_\D\V-2\sum_{t\in(\cE(\F)/\Sigma)}d_tr_t$ such that~$\sum_{t\in(\cE(\F)/\Sigma)}\lambda_{t}^*(w_t)= [h_0]$. (The Levi subgroup~$\M_p$ is in fact the minimal Levi subgroup of~$\G$ which contains the centre of~$\G_{\varphi(\b)}$, but we will not use this.)
}

{Now~$\ft$ determines a unique sub-endo-parameter~$\min(\ft):=(\M_p,\ft_{\M_p})_\G$ as follows. Let~$\t$ be a realization of~$\ft$ with parametrization~$((\V,h),\varphi,\Lambda,\b)$ such that~$\Lambda$ is in the interior of a chamber of the building of~$\G_\varphi(\beta)$; that is,~$\P(\Lambda_{\varphi(\b)})^\circ$ is an Iwahori subgroup of~$\G_\varphi(\beta)$. By conjugating, we may assume that~$\P(\Lambda_{\varphi(\b)})^\circ$ has an Iwahori decomposition with respect to~$(\M_p,\P_p)$, for any parabolic subgroup~$\P_p$ with Levi component~$\M_p$. Then~$\t|_{\M_p}$ is a semisimple character for~$\M_p$ and we let~$\ft_{\M_p}$ be its endo-parameter.
}
%{
%In Proposition~\ref{propBelowConstrofEndoSuppInCCcase} the condition that~$\ind_\M^G\fs$ is the endo-parameter of~$\pi$ is superfluous as will be seen in the following up work on block decompositions. 
%}
\end{remark}

%%%%%%%%%%%%%%%%%%%%%%%%%%%%%
\appendix
%%%%%%%%%%%%%%%%%%%%%%%%%%%%%

%%%%%%%%%%%%%%%%%%%%%%%%%%%%%
\section{Elementary characters}
%%%%%%%%%%%%%%%%%%%%%%%%%%%%%

In this appendix, we prove a technical result which we use to unify the definitions of endo-parameters given in the quaternionic case \cite{SkodInnerFormII} and in the non--quaternionic case \cite{KSS}.   

We fix~$(\F/\F_\so,\varepsilon)$ as in \S\ref{secPreliminaries}.   We need to prove that to every orbit of~$\Ee(\F)/\Sigma$ there is attached an elementary character for a given pair~$(\D,(\overline{\phantom{ll}})_\D)\in\Div(\F/\F_\so,\varepsilon)$ (Note that~$\Ee(\F)/\Sigma$ does not depend on~$(\D,(\overline{\phantom{ll}})_\D)$.)

At first we need the following notation: Let~$\theta\in\Cc(\Lambda,\varphi(\beta))$ be a semisimple character, with block restrictions~$\theta_i,\ i\in\I$. 
We write~$\mathcal{F}(\theta)$ for the set
\[\{[\Theta_i] \mid\ i\in\I\}\]
where~$[\Theta_i]$,~$i\in\I$, is the endo-class of~$\theta_i$. 

\begin{proposition}\label{propAppDRealization}
 Given a pair~$(\D,(\overline{\phantom{ll}})_\D)\in\Div(\F/\F_\so,\varepsilon)$ and an orbit~$\mathcal{O}\in\Ee(\F)/\Sigma$ then there {are} a Hermitian space~$(\V,h)\in\Herm(\F/\F_\so,\varepsilon)$ for~$(\D,(\overline{\phantom{ll}})_\D)$ and~$\theta\in\Cc(h)$ such that~$\mathcal{F}(\theta)=\mathcal{O}$. 
\end{proposition}

\begin{proof}
 We have to consider several cases:
 \begin{enumerate}
  \item $\varepsilon\neq 0$ and $\F=\D$.
  \item $\varepsilon\neq 0$ and $\F\neq\D$ and the orbit~$\mathcal{F}$ has cardinality one. 
  \item $\varepsilon\neq 0$ and $\F\neq\D$ and the orbit~$\mathcal{F}$ has cardinality two. 
  \item $\varepsilon=0$. 
 \end{enumerate}
Case (i) is done in~\cite[Theorem 12.16]{KSS}. Case (iv) follows from transfer using the fact that if~$\E/\F$ is a finite field extension then~$\E\otimes_\F\D$
is a finite~$\D$-vector space with a bi-$\E$-$\D$-module structure. 
We prove the case (ii) first, i.e., we have~$\mathcal{O}=\{[\Theta]\}$ for a ps-character~$\Theta$. By~\cite[Theorem~12.16]{KSS} the ps-character~$\Theta$ can be chosen to be supported on a full self-dual simple element~$\beta$. Consider any~$\varepsilon$-Hermitian form 
\[h':\ \V\times\V\rightarrow (\E\otimes_\F\D, (\overline{\phantom{ll}})_\E\otimes_\F(\overline{\phantom{ll}})_\D).\]
For constructing the desired hermitian form~$h$ we use the~$\F$-linear map~$\lambda_\beta:\E=\F[\beta]\rightarrow\F$ given by
\[\lambda_\beta(\beta^k)=\left\{\begin{array}{ll}
0 & k=1,\ldots,[\E:\F]-1\\
1 & k=0.
\end{array}\right.
\]
 We define~$h$ via~$h:=(\lambda_\beta\otimes\id_\D)\circ h'$. The embedding~$\varphi:\E\rightarrow\End_\D(\V)$ is self-dual with respect to the anti-involution~$\sigma_h$ of~$h$. We write~$\G$ for~$\U(\V,h)$ and~$\G_{\varphi(\beta)}$ for {the} centralizer {of~$\varphi(\beta)$ in~$\G$}. The image of the embedding of buildings
 \[j_{{\varphi(\beta)}}:\mathfrak{B}(\G_{\varphi(\beta)})\hookrightarrow\mathfrak{B}(\G)\]
 consists of the set of self-dual~${\o_{\varphi(\E)}}$-$\o_\D$-lattice functions, see~\cite[Theorem 7.2]{SkodFourier}. Take a rational point~$x$ of~$\mathfrak{B}(\G_{\varphi(\beta)})$, i.e.~a point with rational barycentric coordinates with respect to the vertexes of a chamber. Then~$j_{{\varphi(\beta)}}(x)$ corresponds to a self-dual~${\o_{\varphi(\E)}}$-$\o_\D$-lattice sequence~$\Lambda$. Then {the simple character~$\theta %\in\Cc(\Lambda,\varphi(\beta))^\Sigma\cap\im(\Theta)$ 
 =\Theta((\V,h),\varphi,\Lambda)$} satisfies the assertion. 
 
 Case (iii) has a slightly different proof. By~\cite[Theorem~12.16]{KSS} there exist~$(\V',h')\in\Herm(\F/\F_\so,\varepsilon)$ for~$(\F,\id_\F)$ and an elementary character~$\theta'\in\Cc(h')$, say in~$\Cc(\Lambda',\varphi'(\beta))^{\Sigma'}$, such that~$\mathcal{F}(\theta')=\mathcal{O}$.  
We define on~$\V=\V'\otimes_\F\D$ an~$\varepsilon$-Hermitian form:
\[h:\V\times\V\rightarrow\D,\ h(v'\otimes x,w'\otimes y):=\overline{x}h(v',w')y.\]
We have 
\[\End_\D\V=\End_\D(\V'\otimes_\F\D)=(\End_\F\V')\otimes_\F\D\]
and an embedding
\[\varphi:\E\hookrightarrow\End_D\V\]
by composing~$\varphi'$ with the inclusion. We set~$\G=\U(\V,h)$ and choose a~${\o_{\varphi(\E)}}$-$\o_\D$-lattice sequence~{$\Lambda$} in~$\V$ corresponding to some point in the image of the embedding of~$\mathfrak{B}(\G_{\varphi(\beta)})$ into~$\mathfrak{B}(\G)$ (\cite[Theorem 7.2]{SkodFourier}).
The transfer~$\theta\in\Cc(\Lambda,\varphi(\beta))$ of~$\theta'$ satisfies~$\mathcal{F}(\theta)=\mathcal{O}$.
\end{proof}

%%%%%%%%%%%%%%%%%%%%%%%%%%%%%
\section{Semisimple characters for inner forms of general linear groups}
%%%%%%%%%%%%%%%%%%%%%%%%%%%%%

For inner forms of general linear groups, for~$\R$ an algebraically closed field of characteristic $\ell\neq p$, every cuspidal $\R$-representation contains a simple~$\R$-character.  Thus in some previous works, the authors consider only simple characters and use a combination of simple character theory and parabolic induction to approach~$\Rep_{\R}(\G)$.   In this article, we have preferred a uniform approach utilizing semisimple characters for (inner forms of) classical groups and general linear groups.  The point of this appendix is to prove some results on semisimple characters of inner forms of general linear groups, currently missing from the literature.

%%%%%%%%%%%%%%%%%%%%%%%%%%%%%
\subsection{Semisimple characters in irreducible representations}
In this section, we suppose that~$h=0$, i.e., that~$\G=\GL_m(\D)$.  Note that in this case the sets~$\Cc(\V)$  and~$\Cc_-(h)$ coincide, see~\S\ref{semisimplechartG}and~\S\ref{semisimplecharG}.

\begin{theorem}\label{thmSemiCharInRep}
 Let~~$\R$ be an algebraically closed field of characteristic~$\ell\neq p$, and~$\pi$ be a smooth~$R$-representation of~$\G$. Then there exists a semisimple 
 character in~$\Cc(\V)$ contained in~$\pi$. 
\end{theorem}

As semisimple characters are projective and every smooth~$\R$-representation contains an irreducible subquotient, we reduce to proving the theorem for irreducible~$\R$-representations. For this section, from now on we {fix}~$\pi$ an irreducible~$\R$-representation of~$\G$. 

We consider semisimple characters for strata~$\Delta=[\Lambda,n,s,\beta]$ allowing~$s$ to be zero or positive, see~\cite[Definition 5.4]{SkodInnerFormI}. We define the following sets:
\begin{align*}
\mathfrak{M}_\pi&:=\{(\Delta,\theta)\mid \Delta=[\Lambda,n,s,\beta]\ \text{ a semisimple stratum},\ \theta\in\Cc(\Delta),\ \theta\subseteq\pi\},\\% \text{ and}\\
\mathfrak{N}_\pi&:=\left\{q\in\mathbb{Q}\mid \exists\ ([\Lambda,n,s,\beta],\theta)\in\mathfrak{M}_\pi:\ q=\tfrac{s}{e(\Lambda/\F)}\right\},
\end{align*}
{where~$e(\Lambda/\F)$ denotes the~$\F$-period of~$\Lambda$.}
Then Theorem~\ref{thmSemiCharInRep} states that the set~$\mathfrak{N}_\pi$ has a minimum  equal to zero.  For the proof we show the following assertions. 

\begin{lemma}\label{lemStratumContainedInRep}
 The set~$\mathfrak{M}_\pi$ is not empty. 
\end{lemma}

\begin{proof}
 The representation~$\pi$ being smooth contains a null-stratum~$[\Lambda,s,s,0]$ for~$s$ large enough. 
\end{proof}

\begin{proposition}[{cf.~\cite[Lemma~5.4]{St05},~\cite[Proposition~3.15]{SecherreStevensIV}}]\label{propCentralizerMove}
 Suppose~$(\Delta=[\Lambda,n,s,\beta],\theta)$ is an element of~$\mathfrak{M}_\pi$ with~$s$ positive and~$\tilde{\theta}\in\Cc(\Lambda,s-1,\beta)$
 is an extension of~$\theta\in\Cc(\Lambda,s,\beta)$ and~$c\in\mathfrak{a}_{-s}$. 
 
 Suppose further that~$[\Lambda',n',s',\beta]$ is a semisimple stratum,~$\theta'\in\Cc(\Lambda',s'-1,\beta)$ and~$b'\in\mathfrak{b}_{-s'}\cap\mathfrak{b}_{-s}$ such that 
 \begin{itemize}
  \item $\frac{s'}{e(\Lambda'/\F)}\leq\frac{s}{e(\Lambda/\F)} $
  \item there is a tame corestriction~$s_\beta$ with respect to~$\beta$ (see~\cite[Definition 4.13]{SkodInnerFormI}) such that 
  \[s_\beta(c)+\mathfrak{b}_{1-s}\subseteq b'+\mathfrak{b}_{1-s'}.\]
 \end{itemize}
Then there exist~$c'\in\mathfrak{a}'_{-s'}$ and~$\theta'\in\Cc(\Lambda',s'-1,\beta)$ such that~$s_\beta(c')=b'$ and
$\theta'\psi_{c'}$ is contained in~$\pi$. The element~$c'$ can be chosen to be zero if~$b'=0$.
\end{proposition}

\begin{proof}[Proof of Theorem~\ref{thmSemiCharInRep}]
The proof is mutatis mutandis {that} of~\cite[Theorem 3.1]{SkodlerackCuspQuart} using Proposition~\ref{propCentralizerMove} instead of~\cite[Proposition 3.11]{SkodlerackCuspQuart} and Lemma~\ref{lemStratumContainedInRep} instead of~\cite[Proposition 3.6]{SkodlerackCuspQuart}.
Precisely the proof consists of three parts 

Part 1: $\mathfrak{M}_\pi\neq\emptyset$ (See Lemma~\ref{lemStratumContainedInRep}, cf.~\cite[Proposition 3.6]{SkodlerackCuspQuart}.)

Part 2: Let~$z$ {be} the smallest element of~$\frac{1}{((4N)!)^2}\mathbb{Z}$ such that there is an element~$q$ of~$\mathfrak{N}_\pi$ such that~$q\leq z$. Then~$z$ is the infimum of~$\mathfrak{N}_\pi$.  (See the first paragraph of the proof of~\cite[Theorem 3.1]{SkodlerackCuspQuart}.)

Part 3: The minimum of~$\mathfrak{N}_\pi$ is $0$. (See the part after the first paragraph of the proof of~\cite[Theorem 3.1]{SkodlerackCuspQuart}.)
\end{proof} 

We are left with the proof of Proposition~\ref{propCentralizerMove}. We need to adapt the proofs of~\cite[Lemma 3.14/3.15]{SkodlerackCuspQuart} to the case of~$\GL_m(\D)$.
We recall the groups~$\K^t_1(\La),\ \K^r_2(\La),\ t\in\mathbb{N},$ for a stratum~$[\La,n,s,\beta]$ from~\cite[\S3]{SkodlerackCuspQuart} and~\cite[\S5.2]{St05}: Let~$t$ be a positive integer, and set
\begin{align*}
\K^t_1(\La):=1+\mathfrak{a}_{\lfloor \frac{t}{2}\rfloor+1}\cap \left((\prod_{i\neq j}\A^{i,j})\oplus (\prod_i (\mathfrak{a}_t\cap\A^{i,i}))\right)\\
\K^t_2(\La):=1+\mathfrak{a}_{\lfloor \frac{t+1}{2}\rfloor}\cap \left((\prod_{i\neq j}\A^{i,j})\oplus (\prod_i (\mathfrak{a}_t\cap\A^{i,i}))\right).
\end{align*}
We  denote the intersection of~$\H(\beta,\Lambda)$ and~$\J(\beta,\Lambda)$ with~$\K^t_i(\La)$ by~$\H^t_i(\beta,\La)$ and~$\J^t_i(\beta,\La)$ ($i\in\{1,2\}$). 

\begin{lemma}[{cf.~\cite[Theorem~(3.4.1)]{BK93},~\cite[Proposition~3.24]{St05}}]\label{lemkthetanondeg}
 Given a semisimple character~$\theta\in\Cc(\Lambda,s,\beta)$, the bilinear form
 \[\k_\theta: (\J^{s+1}(\beta,\Lambda)/\H^{s+1}(\beta,\Lambda))\times(\J^{s+1}(\beta,\Lambda)/\H^{s+1}(\beta,\Lambda))\rightarrow \R^\times \]
 defined via~$\k_\theta([g_1]_{\H^{s+1}},[g_2]_{\H^{s+1}})=\theta([g_1,g_2])$
 is non-degenerate, {where~$[g_1,g_2]=g_1g_2g_1^{-1}g_2^{-1}$ denotes the commutator.} % of~$g_1$ with~$g_2$,~$[g_1,g_2]=g_1g_2g_1^{-1}g_2^{-1}$. 
\end{lemma}

Note first that the groups~$\J^{s+1}(\beta,\Lambda)$ and~$\H^{s+1}(\beta,\Lambda)$ are of the form
\[\J^{s+1}(\beta,\Lambda)=1+\mathcal{J}^{s+1}(\beta,\Lambda),\text{ and }\H^{s+1}(\beta,\Lambda)=1+\mathcal{H}^{s+1}(\beta,\Lambda)\]
for~$\o_\F$-orders~$\mathcal{H}^{s+1}(\beta,\ \Lambda),$~$\mathcal{J}^{s+1}(\beta,\Lambda)$.
Let~$\L/\F$ be a maximal unramified field extension in~$\D$. Recall the map \[a_\beta:\ \A\rightarrow\A,\ a_\beta(x):=\beta x-x\beta.\] 

\begin{proof}
 Suppose~$x\in\mathcal{J}^{s+1}(\beta,\Lambda)$ satisfies~$\theta([1+x,1+y])=1$ for all~$y\in\mathcal{J}^{s+1}(\beta,\Lambda)$.
 We claim that~$x$ is an element of~$\mathcal{H}^{s+1}(\beta,\ \Lambda)$. 
 Let~$\theta_\L\in\Cc(\Lambda_\L,s,\beta)$ be an extension of~$\theta$, where~$\Lambda_\L$ is~$\Lambda$ seen as an~$\o_{\L[\beta]}$-lattice sequence in~$\V$ and~$\theta_\L$ is a semisimple character for~$\A_\L:=\Aut_\L(\V)$. 
 From~\cite[Lemma~3.23]{St05} we obtain
 \[1=\theta([1+x,1+y])=\theta_\L([1,+x,1+y])=\psi_{(1+x)^{-1}\beta (1+x)-\beta}(1+y).\]
 Thus~$a_\beta(x)$ is an element of~$(\mathcal{J}^{s+1}(\beta,\Lambda))^*$ (see~\cite[before Lemma~3.4]{SkodlerackCuspQuart} for the~$*$-operation). 
 The set~$(\mathcal{J}^{s+1}(\beta,\Lambda))^*$ is a subset of~$(\mathcal{J}^{s+1}(\beta,\Lambda_\L))^*$ because
 \[\mathcal{J}^{s+1}(\beta,\Lambda_\L)=\mathcal{J}^{s+1}(\beta,\Lambda)\otimes_{o_\F}o_\L.\]
 Thus~$a_\beta(x)$ is an element of~$(\mathcal{J}^{s+1}(\beta,\Lambda_\L))^*$ and therefore
 \[x\in\mathcal{H}^{s+1}(\beta,\La_\L)\cap\End_\D(\V)=\mathcal{H}^{s+1}(\beta,\La)\]
 by the non-degeneracy of~$\k_{\theta_\L}$, see~\cite[Proposition~3.24]{St05}. This finishes the proof. 
\end{proof}

Let~$\Delta=[\La,n,s,\beta],\theta$ and~$c$ be given as in Proposition~\ref{propCentralizerMove}. We {choose a character in~$\Cc(\Lambda,\lfloor\frac{s}{2}\rfloor,\beta)$ which extends~$\theta$, which we shall also call~$\theta$, and} % to~$\H^{\lfloor\frac{s}{2}\rfloor+1}(\beta,\Lambda)$ to an element of~$\Cc(\Lambda,\lfloor\frac{s}{2}\rfloor,\beta)$ and still call this extension~$\theta$ and we 
 consider the character~$(\xi,\H^{s}_1(\beta,\Lambda))$ given by~$\xi:=\theta\psi_c$. 

\begin{lemma}\label{lemNondegkxi}
%Let~$s$ be a positive integer.  T
{With notation as above,} the form
 \begin{align*}
 \k_\xi: \J^s_1(\beta,\La)/\H^s_1(\beta,\La)\times\J^s_1(\beta,\La)/\H^s_1(\beta,\La)&\rightarrow \R^\times \\
([x]_{\H^s_1},[y]_{\H^s_1})&\mapsto \xi([x,y])\end{align*}
 is non-degenerate, and there is unique representation~{$\mu$ of~$\J^s_1(\beta,\La)$ containing~$\xi|_{\H^s_1(\beta,\Lambda)}$.}
\end{lemma}

\begin{proof}
The idea of the proof is given in~\cite[Lemma~5.9]{St05}, but we want to give more details. We consider a linear order on the index set~$\I$ of~$\Delta$ which gives an Iwahori decomposition~$(\U_-\cap\J^s_1)(\M\cap\J^s_1)(\U_+\cap\J^s_1)$ of~$\J_1^s$, see~\cite[Lemma~8.5]{SkodlerackCuspQuart}. 
Let~$j$ be an element of~$\J^s_1$ such that~$\xi([j,j'])=1$,
for all~$j'\in\J^s_1(\beta,\Lambda)$. It is enough to show that~$j$ is an element of~$\H^{1}(\beta,\Lambda)$.
The commutator subgroup of~$\J^{\lfloor\frac{s}{2}\rfloor+1}(\beta,\Lambda)$ is contained in~$\H^{s+1}(\beta,\Lambda)$, {by~\cite[Lemma~3.23]{St05} for~$\J^{\lfloor\frac{s}{2}\rfloor+1}(\beta,\Lambda_\L)$ and restriction to~$\A$. Moreover,} the restriction of~$\xi$ to~$\H^{s+1}(\beta,\Lambda)$ coincides with~$\theta$. 
By Lemma~\ref{lemkthetanondeg} it is enough to prove
\begin{equation}\label{eqthetajjp}\theta([j,j'])=1\end{equation}
for all~$j'\in\J^{\lfloor\frac{s}{2}\rfloor+1}(\beta,\Lambda)$.  
Assertion~\eqref{eqthetajjp} is true for~$j'\in (\J^{\lfloor\frac{s}{2}\rfloor+1}\cap\U_{\pm})$ by assumption, because those intersections are subsets of~$\J^s_1(\beta,\Lambda)$. By the Iwahori decomposition for~$\J^{\lfloor\frac{s}{2}\rfloor+1}(\beta,\Lambda)$ we are left to prove~\eqref{eqthetajjp} for the case~$j'\in\J^{\lfloor\frac{s}{2}\rfloor+1}\cap\M$. We use the factorization
\[j=j_-j_\M j_+,\ j_-\in(\U_-\cap\J^s_1),\ j_\M\in(\M\cap{\J^s}),\ j_+\in(\U_+\cap\J^s_1).\]
{Since~$j_\M$ and~$j_+$ normalize~$\theta$,} we have
\[\theta([j,j'])=\theta([j_-,j'])\theta([j_\M,j'])\theta([j_+,j']),\ j'\in\J^{\lfloor\frac{s}{2}\rfloor+1}\cap\M,\]
and the first and the last factor are~$1$, because~$\theta$ is trivial on~$\H^{s+1}\cap\U_{\pm}$. Further~$j_\M$ is an element of~$\H^s\cap\M$ by Lemma~\ref{lemkthetanondeg}, because~\eqref{eqthetajjp} holds for all~$j'\in\J^s\cap\M$. Therefore~$j_\M$ is an element of~$\H^{\lfloor\frac{s}{2}\rfloor+1}\cap\M$ and thus 
$\theta([j_\M,j'])=1$ for all~$j'\in\J^{\lfloor\frac{s}{2}\rfloor+1}\cap\M$. 
This finishes the proof. 
\end{proof}

\begin{lemma}[{cf.~\cite[Proposition~(8.1.7)]{BK93},~\cite[Lemma~5.8]{St05}}]\label{lemxidecisionforSUbrep}
{Let~$\U$ be a subgroup of~$\K_2^s(\La)$ and let~$\rho$ be an irreducible representation and~$\U$.} 
%Granted that~$s$ is a positive integer. Let~$(\rho,\U)$ be an irreducible representation and~$\U$ be a subgroup of~$\K_2^s(\La)$. 
 Suppose that~$\rho|_{\U\cap\H_1^s(\beta,\La)}$ contains~$\xi|_{\U\cap\H_1^s(\beta,\La)}$. Then~$\rho$ is a subrepresentation of~$\pi|_{\U}$.
\end{lemma}

\begin{proof}
 By the second part of the proof of~\cite[Proposition (8.1.7)]{BK93} we only need to prove that~$\ind_{\H^s_1}^{\K^s_2}\xi$ is a multiple of a unique irreducible representation. By Lemma~\ref{lemNondegkxi} the representation {of}~$\J^s_1(\beta,\Lambda)$ induced {from~$\xi$} is a multiple of a unique irreducible representation~$\mu$. %,\J^s_1(\beta,\Lambda))$. 
 We only need to show that~$\mu$ induces irreducibly to~$\K^s_2(\beta,\Lambda)$, i.e., that the intertwining of~$\mu$ in~$\K^s_2(\beta,\Lambda)$ is contained in~$\J^1(\beta,\Lambda)$, noting that~$\J^s_1(\beta,\Lambda)$ and~$\J^s_2(\beta,\Lambda)$ coincide because~$s$ is smaller than~$-k_0(\beta,\Lambda)$, {where}~$k_0(\beta,\Lambda)$ {is} the critical exponent of~$\beta$ with respect to~$\Lambda$ {see~\cite[\S4.1]{SkodInnerFormI}}. 
  An element~$k$ of~$\K^s_2(\beta,\Lambda)$ which intertwines~$\mu$ also intertwines~$\theta|_{\H^{s+1}(\beta,\Lambda)}$ and therefore~$\theta_\L|_{\H^{s+1}(\beta,\Lambda_\L)}$ for an extension~$\theta_\L\in\Cc(\La_\L,s,\beta)$, by the intertwining formulas in~\cite[Proposition~5.15]{SkodInnerFormI}, and it is an element of~$\K^s_2(\beta,\La_\L)$.
 Thus by the proof of~\cite[Proposition~5.8]{St05} the element~$k$ lies in~$\J^1(\beta,\La_\L)$ and therefore in~$\J^1(\beta,\La)$. This finishes the proof.
\end{proof}

\begin{proof}[Proof of Proposition~\ref{propCentralizerMove}] 
 The proof is similar to the proof of~\cite[Proposition~3.11]{SkodlerackCuspQuart} with the following modifications (see also: ~\cite[Proposition~3.15]{SecherreStevensIV}):
\begin{itemize}
 \item The character~$\theta_\L$ is just an extension of~$\theta$. (Note that we have no Glauberman correspondence anymore.) 
 \item We don't work in~$\A_\F=\End_\F(\V)$. We work directly in~$\A=\End_\D(\V)$, see~\cite{SecherreStevensIV} after Lemma~3.16. Note that there is no averaging function as in~\cite[\S 2.8]{SkodlerackCuspQuart} and we use the simple case~\cite[Paragraph after proof of Lemma~3.16]{SecherreStevensIV} and~\cite[Lemma~3.17]{SecherreStevensIV} to obtain the elements~$c'$ and~$x$ of Step 3 and 4 in~\cite[Proposition~3.11]{SkodlerackCuspQuart}.  
 \item There is no Cayley map. We conjugate in Step 4 with~$1+x$ instead.
\end{itemize}
\end{proof}

%%%%%%%%%%%%%%%%%%%%%%%%%%%%%
\subsection{Heisenberg representations for semisimple characters for general linear groups}
We continue in the special case~$h=0$, i.e.~$\G=\GL_\D(\V)$ and~$\R$ an algebraically closed field of characteristic~$\ell\neq p$.  We consider a semisimple character~$\theta\in\Cc(\beta,\Lambda)$ and a parabolic subgroup~$\P=\M\N$ which is adapted to a decomposition of~$\V$ properly subordinate to the stratum~$[\Lambda,-,0,\b]$, see~\cite[Definition 8.1 and before Lemma 8.4]{SkodlerackCuspQuart}.  By Lemma~\ref{lemkthetanondeg} there exists a unique (up to equivalence) irreducible~$\R$-representation~$\eta$ of~$\J^1$ containing~$\theta$, which we call the~\emph{Heisenberg representation} of~$\theta$.   The first statement is about the intertwining of~$\eta$.
\begin{proposition}\label{propetaIntertwiningtG}
For~$g\in \G$, the intertwining of~$\eta$ is given by
\begin{equation*}
\Hom_{\R[\J^1\cap (\J^1)^g]}(\eta,\eta^g)\simeq \begin{cases}
\R&\text{if }g\in \J^1\G_{{\beta}}\J^1;\\
0&\text{otherwise.}\end{cases}\end{equation*}
\end{proposition}

\begin{proof}
Let~$\theta_\L\in\Cc(\Lambda,\beta\otimes_\F 1)$ be an extension of~$\theta$, a semisimple character for~$\G\otimes\L$, and~$\eta_\L$ its Heisenberg~$\R$-representation.  The restriction of~$\eta_\L$ to~$\J^1$ is a direct sum of copies of~$\eta$.  Now the proof is similar to~\cite[Proposition~4.3]{SkodlerackCuspQuart}. 
\end{proof}

We also need Heisenberg representations for pairs of lattice sequences. Note that we have fixed a semisimple stratum~$\Delta=[\Lambda,n,0,\beta]$ with~$\E=\F[\beta]$ a product of fields in~$\End_\D(\V)$ and~$\Lambda$ an~$\o_{\E}$-$\o_\D$-lattice sequence in~$\V$, with hereditary order~$\mathfrak{a}_0=\mathfrak{a}_0(\Lambda)$.

\begin{proposition}[{cf.~\cite[Proposition~(5.1.14)]{BK93}}]\label{propHeisenbergPairs}
Let~$\Lambda'$ be an~$\o_{\E}$-$\o_\D$-lattice sequence in~$\V$ such that its hereditary order~$\mathfrak{a}_0(\Lambda')$ is contained in~$\mathfrak{a}_0$. Let~$\theta'$ be the transfer of~$\theta$ to~$\Lambda'$ with Heisenberg~$\R$-representation~{$\eta'$ of~$\J^1_{\Lambda'}$}. We denote by~$\J^1_{\Lambda',\Lambda}$ the group~$\P_1(\Lambda'_\beta)\J^1_{\Lambda}$. %(Note:~$\P_1(\Lambda'_\beta)$ is the intersection of~$\P_1(\Lambda')$ with~$\G_{{\varphi(\beta)}}$.) 
Then:
\begin{enumerate}
\item There is a unique (up to equivalence) extension~$(\eta_{\Lambda',\Lambda},\J^1_{\Lambda',\Lambda})$ of~$\eta$, such that~$\eta_{\Lambda',\Lambda}$ and~$\eta'$ induce isomorphic representations on~$\P_1(\Lambda')$. \label{propHeisenbergPairs-i}
\item The induced representation~$\ind_{\J^1_{\Lambda',\Lambda}}^{\P_1(\Lambda')}(\eta_{\Lambda',\Lambda})$ is irreducible.
\label{propHeisenbergPairs-ii}
\item The~$\G$-intertwining set of~$\eta_{\Lambda',\Lambda}$ is given by the set~$\J^1_{\Lambda',\Lambda}\G_\beta\J^1_{\Lambda',\Lambda}$.\label{propHeisenbergPairs-iii}
\end{enumerate}
\end{proposition}

\begin{proof}
To prove assertions \ref{propHeisenbergPairs-i} and~\ref{propHeisenbergPairs-ii} one can proceed in a similar fashion to the proof of~\cite[Proposition (5.1.14)]{BK93} if we can establish 
\begin{equation}\label{eq5_1_17}
(\P_1(\Lambda'):\J^1_{\Lambda'})\dim(\eta')=(\P_1(\Lambda'):\J^1_{\Lambda',\Lambda})\dim(\eta),
\end{equation}
see~\cite[(5.1.17)]{BK93}.   Assertion~\ref{propHeisenbergPairs-iii} follows as in~\cite[Proposition (5.1.19)]{BK93} if we can establish 
\begin{equation}\label{eq5_1_19}
(\P_1(\Lambda')x\P_1(\Lambda'))\cap \G_\beta=\P_1(\Lambda'_\beta)x\P_1(\Lambda'_\beta),
\end{equation} 
for all~$x\in\G_\beta$. (Note that we have~$\G=\tG$ here.) We now prove~\eqref{eq5_1_17} and~\eqref{eq5_1_19}.

For~\eqref{eq5_1_17}: We need to show the exactness of the sequence
\begin{equation}\label{eqDuke3_17}
0\rightarrow\mathfrak{b}_1(\Gamma)\rightarrow\mathcal{J}^1(\beta,\Gamma)\rightarrow (\mathcal{H}(\beta,\Lambda))^*\rightarrow\mathfrak{b}_0(\Gamma)\rightarrow 0
\end{equation}
in~$\A=\End_\D\V$ 
for~$\Gamma=\Lambda,~\Lambda'$ to then follow the proof of~\cite[Proposition~(5.1.2)]{BK93}. We just establish the exactness for~$\Gamma=\Lambda$. By~\cite[Lemma 3.17]{St05} the exactness of~\eqref{eqDuke3_17} is known if~$\D=\F$, in particular the corresponding sequence in~$\End_\L\V=\A\otimes_\F\L$ is exact for any maximal unramified field extension~$\L/\F$ in~$\D$. Note that the Galois group of~$\L/\F$ acts on the second factor of~$\A\otimes_\F\L$. Taking the~$\Gal(\L/\F)$-fixed points of the latter sequence we obtain the exactness of~\eqref{eqDuke3_17}.

A similar idea applies for~\eqref{eq5_1_19} which is known for~$\D=\F$, by~\cite[Lemma~4.6]{RKSS}. Thus both terms in~\eqref{eq5_1_19} become equal after tensoring with~$\o_\L$ over~$\o_\F$. Taking the~$\Gal(\L/\F)$-fixed points on both sides we obtain the result by~\cite[Lemma~2.1]{StevensDouble}.
\end{proof}

Let~{$\eta_\P$} be the natural representation of{$\J^1_\P=(\H^1(\J^1\cap\P))$} on the set of~$(\J^1\cap\N)$-fixed points of~$\eta$.  

\begin{proposition}[{cf.~\cite[Proposition~8.6]{SkodlerackCuspQuart}}]\label{propetaP}
The~$\R$-representation~$\eta_\P$ is irreducible, and~$\ind_{\J^1_\P}^{\J^1}\eta_\P\simeq\eta$.
\end{proposition}

 \begin{proof}
  The proof is similar to~\cite[Proposition~8.6]{SkodlerackCuspQuart}. We do not need to apply {the} Glauberman correspondence, because
  the representations~$\eta_\M\otimes 1_{\J^1\cap\N}$ induces on~$\J^1$ a semisimple representation of the same rank as~$\eta$ which contains~$\eta$. Therefore we only need to prove that~$\eta_\P$ and~$\eta_\M\otimes 1_{\J^1\cap\N}$ coincide. This follows, because the restriction of~$\eta$ to~$\J^1_\P$ contains all the representations~$\eta_\M\otimes\phi$, where~$\phi$ {runs} through all characters of~$(\J^1\cap\N)/(\H^1\cap\N)$. Therefore~$\eta_\M\otimes 1_{\J^1\cap\N}$ occurs with multiplicity one in~$\eta$ and is equal to~$\eta_\P$. 
\end{proof}

%%%%%%%%%%%%%%%%%%%%%%%%%%%%%
\section{Compact subgroups adapted to Levi subgroups of~$\G$. }\label{AppendixC}
%%%%%%%%%%%%%%%%%%%%%%%%%%%%%

In this section we recall the subgroups~$\J_\P$, see~\cite[after Lemma~8.4]{SkodlerackCuspQuart} and~\cite[\S 5]{St08}.
We arbitrarily fix a linear order $<$ on the set of idempotents of~$\A=\End_\D(\V)$. 
Let~$\M$ be a Levi subgroup of~$\G$. In this paragraph we introduce the notion of a semisimple stratum adapted to~$\M$ and we construct compact subgroups~$\J_\P(\beta,\Lambda),\ \J_\P^{\st,\M}(\beta,\Lambda)$ for a parabolic subgroup ~$\P$ with Levi subgroup~$\M$. The latter are used for the construction of Bushnell--Kutzko--Stevens types for Bernstein components.  

%%%%%%%%%%%%%%%%%%%%%%%%%%%%%
\subsection{Strata adapted to~$\M$}
At first we treat general linear groups and then classical groups. 

%%%%%%%%%%%%%%%%%%%%%%%%%%%%%
\subsubsection{The case~$h=0$}\label{secAppJhiszero}
For~$\M$ there is exactly one tuple~$\ee^\M=(\be^\M_j)_{j\in S^\M}$ of pairwise orthogonal non-zero idempotents  in~$\A$ which sum to~$1$ and such that 
\begin{itemize}
\item $S^\M=\{1,2,3,\ldots,l\}$ and~$\be^\M_j <\be_k^\M$ for~$j<k$. 
\item $\M=\{g\in\G\mid\ g\be_j^\M=\be^\M_jg\ \text{ for all }j\in S^\M\}$
\end{itemize}
Note that the condition~$g\be_j^\M=\be^\M_jg$ for all $j\in S^\M$ is equivalent to~$g=\sum_{j\in S^\M} \be^\M_j g\be^\M_j$, and to~$\be^\M_k g\be^\M_j=0$ for~$k\ne j$. We also set
\[
\Z_\M^\circ = \left\{\sum_{j\in S^\M} \l_j \be^\M_j \mid \l_j\in F_\so^\times\right\}
\]
which is clearly central in~$\M$; indeed, it is the maximal split torus in the centre of~$\M$, but we will not use this fact.

\begin{definition}\label{defstratumAdaptedToMtG}
A semisimple stratum~$\Delta=[\Lambda,n,0,\beta]$ in~$\A$ is called (exactly) adapted to~$\M$ if~$\ee^\M$ is (exactly) properly subordinate to~$\Delta$, see~\cite[Definition~8.1]{SkodlerackCuspQuart}.
\end{definition}

%%%%%%%%%%%%%%%%%%%%%%%%%%
\subsubsection{The case~$h\neq 0$}\label{secAppJhisnotzero}
At first we need to find a tuple of idempotents describing~$\M$. 

\begin{definition}
Let~$S$ be a non-empty set satisfying
\[ \{\pm 1,\pm 2,\pm 3, \ldots,\pm l\}\subseteq S\subseteq \{0, \pm 1,\pm 2,\pm 3, \ldots,\pm l\}\] 
for some non-negative integer $l$. For~$l=0$ we interpret this to mean~$S=\{0\}$. 
A family~$(\be_j)_{j\in S}$ of pairwise orthogonal non-zero idempotents is called an~($h$-)\emph{self-dual partition} if~$\sigma_h(\be_j)=\be_{-j}$ for all~$j\in S$.
We call a self-dual partition~$\ee=(\be_j)_{j\in S}$ \emph{admissible} for $h$ if either~$0\not\in S$ or~$h|_{\be_0\V}$ is not an isotropic orthogonal form over~$\F=\D$ of type~$(1,1)$, i.e., it doesn't define~$\O(1,1)(\F)$, 
\end{definition}
Given a self-dual partition~$\ee=(\be_j)_{j\in S}$, we also set
\[
\Z_\ee^\circ = \left\{\sum_{j\in S} \l_j \be_j \mid \l_j\in F_\so^\times,\ \l_0=1,\ \l_{-j}=\l_j^{-1}\right\},
\]
which is a subgroup of~$\G$. Note that, since~$F_\so^\times$ contains more than the two elements~$\pm 1$, any element of~$\A$ which commutes with every element of~$\Z_{\ee}^\so$ also commutes with~$\be_j$, for all~$j\in S$.

\begin{lemma}\label{lemAdmissPartforM}
There exists exactly one admissible partition~$\ee=(\be_j)_{j\in S}$ for~$h$ such that
\begin{equation}\label{eqM}
\M=\{g\in\G\mid\ \forall_{j\in S}:g\be_j=\be_jg\}
\end{equation}
and~$\be_1<\be_2<\ldots<\be_l$ and~$\be_{-j}<\be_j$ for all~$j>0$. .
\end{lemma}
The order condition on~$(\be_j)_{j\in S}$ is equivalent to
\[
\be_j=\max\{\be_i\mid\ i\in\{\pm 1,\pm 2,\pm 3, \ldots,\pm j\}\},\ j\in S,\ j>0.
\]
The partition in the lemma will be denoted by~$\ee^\M=(\be_j^\M)_{j\in S^\M}$. 

\begin{proof}
We need to prove two parts: existence and uniqueness. The existence is straightforward from the description of~$\M$ via a Witt-basis for~$h$. For proving uniqueness, suppose there are two admissible partitions~$\ee=(\be_j)_{j\in S}$ and~$\ff=(\rmf_k)_{k\in K}$ satisfying the assertions of the lemma.
We are going to prove~$\ee=\ff$. 

Because, for~$k\in K$, the idempotent~$\rmf_k$ commutes with all elements of~$\M\supseteq\Z_\ee^\circ$, we see that~$\rmf_k$ commutes with all~$\be_j$. 
We consider the set
\[
\mathfrak{M}(\ee,\ff):=\{(j,k)\in S\times K\mid\ \be_j \rmf_k\neq 0\}.
\]
We claim that for every index~$j\in S$ there is exactly one index~$k\in K$ such that~$(j,k)\in\mathfrak{M}(\ee,\ff)$. Then, by symmetry, we see that~$\mathfrak{M}(\ee,\ff)$ is the graph of a bijection from which we deduce~$\ee=\ff$ from the order condition.

At first we consider~$j\neq 0$. The element~$\be_j\rmf_k$ is an idempotent commuting with all elements of the simple algebra~$\End_\D(\be_j\V)$ and therefore
$\be_j\rmf_k$ is equal to~$\be_j$ or~$0$. But if~$k,k'\in K$ are such that $\be_j\rmf_k=\be_j=\be_j\rmf_{k'}$ then also~$\be_j\rmf_k\rmf_{k'}=\be_j$ so~$\rmf_k\rmf_{k'}\ne0$ and the pairwise orthogonality of~$\ff$ gives~$k=k'$.

For the case~$j=0$, by projecting onto~$\be_0\V$ we can assume~$S=\{0\}$, i.e.~$\be_0$ is the identity of~$\V$ and~$\M=\G$. Assume that~$K$ contains a positive element. Then~$h$ has to be orthogonal, because otherwise we can find a unipotent element in~$\G$ which does not commute with~$f_1$. If~$K$ has at least four elements then $\G$ contains an element~$g$ which satisfies~$f_2 g f_1\neq 0$. A contradiction. Therefore we obtain~$\dim_\D f_1\V\geq 2$ or~$|K|=3$ by admissibility of~$e$. In the first case we obtain a unipotent element which does not commute with~$f_1$ and in the second an element which swaps~$f_1$ with~$f_{-1}$. A contradiction. Therefore~$|K|=1$. 
This proves the claim and by symmetry we finish the proof of the lemma. 
\end{proof}

We now obtain {definitions} similar to those in~\S\ref{secAppJhiszero} 

\begin{definition}\label{defstratumAdaptedToMG}
A self-dual semisimple stratum~$\Delta=[\Lambda,n,0,\beta]$ in~$\A$ is called (exactly) adapted to~$\M$ if~$e^\M$ is (exactly) properly self-dual subordinate to~$\Delta$, see~\cite[Definition~8.2]{SkodlerackCuspQuart}.
\end{definition} 

. 
\newcommand{\tN}{\tilde{\N}}
\newcommand{\btN}{\overline{\tilde{\N}}}
\newcommand{\bN}{\overline{\N}}
\newcommand{\tGbL}{\tP(\Lambda_\beta)}
\newcommand{\GbL}{\P(\Lambda_\beta)}
\newcommand{\GbLst}{\P^{\st}(\Lambda_\beta)}
\newcommand{\Nrd}{\text{Nrd}}
\newcommand{\Gb}{\G_\beta}
\newcommand{\tGb}{\tG_\beta}
\newcommand{\mfB}{\mathfrak{B}}
\newcommand{\MbL}{\P(\Lambda_{\M,\beta})}
\newcommand{\MbLst}{\P^{\st}(\Lambda_{\M,\beta})}
\newcommand{\antidiag}{\text{antidiag}}
\newcommand{\mf}[1]{\mathfrak{#1}}

%%%%%%%%%%%%%%%%%%%%%%%%%%%%%
\subsection{Construction of the group~$\J_\P$}
Let~$\Delta$ be {an}  $h$-self-dual semisimple stratum adapted to~$\M$ and~$\P=\M\N$ be the parabolic subgroup of~$\G$ given by the order of the indices of~$S^\M$, i.e.
\[\P=\{ g\in\G\mid \  e^\M_k g e^\M_j=0\ \text{for all}\ j<k,\  j, k\in S^\M\},\]
with opposite unipotent radical~$\bar{\N}$. 
Then~$\J=\J(\beta,\Lambda)$ has an Iwahori decomposition with respect to~$\bar{\N}\M\N$, see~\cite[Lemma~8.4]{SkodlerackCuspQuart}, i.e., 
\[\J=(\J\cap\bar{\N})(\J\cap\M)(\J\cap\N).\]
We need a finer version:
\begin{lemma}\label{lemIwahoriGbeta}
The group~$\GbL$ has an Iwahori decomposition with respect to~$\bar{\N}\M\N$, i.e., 
\[\GbL=(\GbL\cap\bar{\N})(\GbL\cap\M)(\GbL\cap\N)\]
\end{lemma}

\begin{proof}
{We first consider the case}~$h=0$. An element~$g$ of~$\P(\Lambda_\beta)$ has an Iwahori decomposition~$g=j_{\bar{\N}} j_\M j_\N$ in~$\J$. Since $g$ commutes with~$\beta$ the elements~$j_{\bar{\N}},j_\M,j_\N$ commute with~$\beta$ too. 

For~$h\neq 0$ we also consider the ambient general linear group~$\tG$. First we note that the centralizers satisfy~$\G_\beta=\tG_\beta\cap\G$ and by the first part of the proof we obtain:
$\P(\Lambda_\beta)$ is contained in~$(\tGbL\cap\btN)^\Sigma(\tGbL\cap\tM)^\Sigma(\tGbL\cap\tN)^\Sigma,$ where~$\btN\M\tN$ is the Iwahori decomposition in~$\tG$ defined by~$e^\M$. Further the first and the third factor 
are contained in~$\tP_1(\Lambda_\beta)$, respectively,~because~$e^\M$ is properly subordinate to~$\Delta$, thus in the orthogonal case no elements of reduced norm~$-1$ occur in~$\tP_1(\Lambda_\beta)\cap\btN$ and~$\tP_1(\Lambda_\beta)\cap\tN$. 
Let~$g$ be an element of~$\GbL$ with Iwahori decomposition~$g=g_{\bN} g_\M g_\N$ then 
\[\Nrd_{\A/\F}(g)\equiv\Nrd_{\A/\F}(g_\M) \pmod{{1+\p_\F}}\]
 and therefore~$g_\M$ is an element of~$\tGbL\cap\tM\cap\G=\GbL\cap\M$.
\end{proof}

We now want to define subgroups of~$\J(\beta,\Lambda)$ which are adapted to~$\bN\M\N$. We  have two choices because we have two poly-simplicial structures on the buidling of~$\Gb$. The first is where facets are defined by self-dual lattice chains (the \emph{weak structure}) and the second one obtains from the weak structure by removing all thin panels (the \emph{strong structure}), see the oriflame construction in~\cite[\S 8.2]{abramenkoNebe:02}.

\begin{definition}\label{defstrongPointsandFacets}
 We call a point~$\Lambda$~in the building of~$\G_{\beta}$ \emph{strong} if the closure of its facet with respect to the weak structure is the intersection of closures of thick panels. Otherwise, the point is called~\emph{weak}. 
\end{definition}

%%%%%%%%%%%%%%%%%%%%%%%%%%%%%
\subsubsection{The case~$\G=\P=\M$}
We have two groups
\begin{enumerate}
\item$\J(\beta,\Lambda)=\GbL\J^1(\beta,\Lambda)$ and
\item$\J^{\st}(\beta,\Lambda):=\GbLst\J^1(\beta,\Lambda)$
\end{enumerate}
where~$\GbLst$ is the pointwise stabilizer of the facet of~$\Lambda_\beta$ with respect to the strong  poly-simplicial structure on the building~$\mfB(\Gb)$ of~$\Gb$. 
We are going to omit the argument~$(\beta,\Lambda)$ and write~$\J$ and~$\J^{\st}$ instead if there is no reason for confusion. 
We need a fine enough sufficient condition for the groups~$\J$ and~$\J^{\st}$ to coincide, i.e., for~$\GbL=\GbLst$. Let~$\M(\Delta)$ be the Levi subgroup defined by~$\Delta$ and let us suppose that~$h|_{\V^{\I_0}}$ does not define the group~$\O(1,1)(\F)$, {where}~$\V^{\I_0}$ {is} the sum of all~$\V^i,\ i\in\I_0$.
We have
\[\M(\Delta)\simeq \M_0\times\prod_{i\in \I_+}\GL_\D(\V^i)\] 
where~$\M_0$ is the set of~$\F$-rational points of the algebraic connected component of~$\U(\V^{\I_0},h|_{\V^{\I_0}})$.
We have~$\GbLst=\GbL$ if~$\Lambda_{\beta_{i}}$ is not a weak point of~$\B(\G_{\beta_{i}})$ for all~$i\in \I_0$. 
So we obtain the following: 

\begin{proposition}\label{propJsteqJ}
Suppose~$\Lambda_{\beta_{i}}$ is a strong point in~$\B(\G_{\beta_{i}})$ for all~$i\in\I_0$. Then~$\J=\J^{\st}$. 
\end{proposition}

%%%%%%%%%%%%%%%%%%%%%%%%%%%%%
\subsubsection{The general case}
In this paragraph we define and compare the more general groups~$\J_\P(\beta,\Lambda)$ and~$\J_\P^{\st,\M}(\beta,\Lambda)$.
We need a lemma for preparation. 
\begin{lemma}\label{lemJcapP}
\begin{enumerate}
\item $\J\cap\P=(\M\cap\GbL)(\J^1(\beta,\Lambda)\cap\P).$
\item $\J^{\st}\cap\P=(\M\cap\GbLst)(\J^1(\beta,\Lambda)\cap\P).$
\end{enumerate}
\end{lemma}

\begin{proof}
We prove the first assertion. The proof of the other one is similar. An element~$g\in\J\cap\P$ has the form~$g=g_\beta g_{\J^1}$, {with}~$g_\beta\in\GbL,\ g_{\J^1}\in\J^1$. We use the Iwahori decompositions:
\[g_\beta=g_{\beta,-}g_{\beta,0}g_{\beta,+},\qquad g_{\J^1}=g_-g_0 g_+,\]
using Lemma~\ref{lemIwahoriGbeta}. Note that~$g_{\beta,+}$ is an element of~$\GbL\cap\N$ which is a subset of~$\J^1$. We therefore can assume without loss of generality that~$g_{\beta,+}$ is trivial.  
We obtain 
\[g=g_{\b,-}(g_{\beta,0}g_- g_{\beta,0}^{-1})g_{\beta,0}g_{0}g_{+}.\]
Therefore, by the uniqueness of Iwahori decomposition, using that~$g$ is an element of~$\P$, we obtain that~$g_{\b,-}(g_{\beta,0}g_0 g_{\beta,0}^{-1})$ is trivial and therefore~$g=g_{\beta,0}g_{0}g_{+}$, showing the assertion. 
\end{proof} 

We now define:
\begin{enumerate}
\item $\J_\P^{\BK}(\beta,\Lambda)=\H^1(\beta,\Lambda)(\J(\beta,\Lambda)\cap\P)=(\H^1(\beta,\Lambda)\cap\bar{\N} )(\M\cap\GbL)(\J^1(\beta,\Lambda)\cap\P);$
\item $\J_\P(\beta,\Lambda)=\H^1(\beta,\Lambda)(\J^{\st}(\beta,\Lambda)\cap\P)=(\H^1(\beta,\Lambda)\cap\bar{\N} )(\M\cap\GbLst)(\J^1(\beta,\Lambda)\cap\P);$
\item $\J^{\st,\M}_\P(\beta,\Lambda)=(\H^1(\beta,\Lambda)\cap\bar{\N} )(\MbLst)(\J^1(\beta,\Lambda)\cap\P)$.
\end{enumerate}

We need a fine sufficient condition for  $\J_\P$ and~$\J_\P^{\st,\M}$ to coincide. For that we need to analyze the centralizer~$\M_\beta$. 
We have two decompositions of~$\V$. 
\begin{itemize}
\item one given by the idempotents~$ e^\M=(e^\M_j)_{j\in S}$ for~$\M$ and
\item one given by the idempotents~$\mathbbm{1}=(\mathbbm{1}_i)_{i\in\I}$
\end{itemize}

We get a possibly finer decomposition by the idempotents in
\[e_{j,i}:=e^\M_j\mathbbm{1}_i,\ (j,i)\in \mathfrak{M}(e^\M,\mathbbm{1})\]
where the latter set is the set of pairs~$(j,i)$ such that~$e^\M_j\mathbbm{1}_i$ is non-zero.
We compute~$\M_\beta$:

\begin{eqnarray*}
\M_\beta &=& \Gb\cap\prod_{(j,i)\in\mathfrak{M}(e^\M,\mathbbm{1})} \Aut_\D(e_{(j,i)}\V)\\
&\simeq& \prod_{(j,i)\in\mathfrak{M}(e^\M,\mathbbm{1}),j>0} \tG_{\beta_{(j,i)}}\times  \prod_{(0,i)\in\mathfrak{M}(e^\M,\mathbbm{1}),j\in\I_+} \tG_{\beta_{(0,i)}}\times \prod_{(0,i)\in\mathfrak{M}(e^\M,\mathbbm{1}),i\in\I_0} \G_{\beta_{(0,i)}}\
\end{eqnarray*}
Note that for~$\M=\M(\Delta)$ of $\G$ we get the decomposition of~$\Gb$ with respect to the associated splitting of~$\Delta$. For the stabilizer of~$\Lambda_\beta$ in~$\M$ we get\
 
\begin{eqnarray*}
\P(\Lambda_{\M,\beta}) &\simeq&\prod_{(j,i)\in\mathfrak{M}(e^\M,\mathbbm{1}),j>0} \tP(\La_{\beta_{(j,i)}})\times  \prod_{(0,i)\in\mathfrak{M}(e^\M,\mathbbm{1}),i\in\I_+}\tP(\La_{\beta_{(0,i)}})\times \prod_{(0,i)\in\mathfrak{M}(e^\M,\mathbbm{1}),i\in\I_0} \P(\La_{\beta_{(0,i)}})
\end{eqnarray*}
and for the stabilizer with respect to the strong poly-simplicial structure:
\begin{eqnarray*}
\P^{\st}(\Lambda_{\M,\beta}) &\simeq& \prod_{(j,i)\in\mathfrak{M}(e^\M,\mathbbm{1}),j>0} \tP(\La_{\beta_{(j,i)}})\times  \prod_{(0,i)\in\mathfrak{M}(e^\M,\mathbbm{1}),i\in\I_+}\tP(\La_{\beta_{(0,i)}})\times \prod_{(0,i)\in\mathfrak{M}(e^\M,\mathbbm{1}),i\in\I_0} \P^{\st}(\La_{\beta_{(0,i)}})
\end{eqnarray*}

Therefore we get
\begin{proposition}\label{propJPsteqJP}
The groups~$\J_\P^{\BK}$ and~$\J^{\st,\M}_\P$ coincide if~$\La_{\beta_{(0,i)}}$ is a strong point in~$\B(\G_{\beta_{(0,i)}})$ for all~$i\in\I_0$ such that~$(0,i)\in\mathfrak{M}(e^\M,\mathbbm{1})$.
\end{proposition}

The main goal is the following proposition:
\begin{proposition}\label{propJPst}
The group~$\J_\P^{\st,\M}(\beta,\Lambda)$ is a subgroup of $\J^{\st}(\beta,\Lambda)$ and we have
\[\J_\P^{\st,\M}(\beta,\Lambda)\subseteq\J_\P(\beta,\Lambda)\subseteq\J_\P^{\BK}(\beta,\Lambda).\]
In particular, under the condition of Proposition~\ref{propJPsteqJP}, the group~$\J_\P^{\BK}(\beta,\Lambda)$ is a subgroup of~$\J^{\st}(\beta,\Lambda)$. 
\end{proposition}

\begin{remark}\label{remJPst}
We will need to use this in Section~\ref{typesoverC}, where we start from a type for a cuspidal representation of~$\M$. In this case, the conditions of Proposition~\ref{propJPsteqJP} are indeed satisfied, by~\cite[Propositions~4.2,~4.4]{MiSt} when~$\D=\F$, by~\cite[Corollary 5.20]{SecherreStevensIV} for inner forms of general linear groups, and by~\cite[\S12.3]{SkodlerackCuspQuart} for inner forms of classical groups.
\end{remark}

Before proving Proposition~\ref{propJPst}, we illustrate the main insight for the above statement:
\begin{example}
We consider the group~$\G=\SO(3,3)(\F)$ defined by the symmetric form~$h$ with Gram matrix $\antidiag(1,1,1,1,1,1)$ {on}~$\V=\F^{6}$. Let~$\Lambda$ be the strict lattice sequence given by the hereditary order
\[\left(\begin{matrix}\o&\p&\p\\ \o&\o&\p \\ \p^{-1}&\o&\o\\ \end{matrix}\right)^{(2,2,2)}.\]
The stratum~$\Delta$ is the null-stratum of depth zero with lattice sequence~$\Lambda$. 
The Levi~$\M$ is given by the idempotents. 
\[e_0=\diag(1,1,0,0,1,1),\ e_1=\diag(0,0,1,0,0,0),\ e_{-1}=\diag(0,0,0,1,0,0).\]
We observe: 
\begin{itemize}
\item $h_{e_0\V}$ defines~$\SO(2,2)(\F)$, and
\item $\Lambda_0=e_0\Lambda$ is a strong vertex in~$\mathfrak{B}(\SO(2,2)(\F))$, even though~$\Lambda$ is a weak point in~$\mathfrak{B}(\G)$. 
\end{itemize} 
And we obtain
\[\P^{\st}(\Lambda_0)=\P(\Lambda_0)\subseteq \G\cap \left(\begin{matrix}\o& &\p\\ &1 & \\ \p^{-1}& &\o \\ \end{matrix}\right)^{(2,2,2)}\]
which is contained in~$\P^{\st}(\Lambda).$ 
\end{example}

\begin{proof}[Proof of Proposition~\ref{propJPst}]
We have to show:
\[\prod_{(0,i)\in\mathfrak{M}(e^\M,\mathbbm{1}),i\in\I_0} \P^{\st}(\La_{\beta_{(0,i)}})\subseteq \prod_{i\in\I_0} \P^{\st}(\La_{\beta_{i}})\]
Without loss of generality we can assume that~$\Delta$ is a null-stratum of depth zero with lattice sequence~$\Lambda$, in particular~$\I=\I_0$ is a singleton. 
We put~$\Lambda_0=e_0\Lambda$ and we have to prove
\[\P^{\st}(\Lambda_0)\subseteq \P^{\st}(\Lambda).\]
If~$\La$ is a strong point in~$\mathfrak{B}(\G)$ then~$\P(\Lambda)=\P^{\st}(\Lambda)$ and the inclusion follows because~$\P(\Lambda)$ contains~$\P^{\st}(\La_0)$.
Suppose that~$\Lambda$ is a weak point in~$\mathfrak{B}(\G)$. Then~$(\tP(\La)\cap\U(\V,h))/\P_1(\La)$ contains a direct factor isomorphic to~$\O(1,1)(k_\D)$. (Note that~$\P(\La)$ is a subgroup of~$\G$, so~$\tP(\La)\cap\U(\V,h)$ could be bigger.)
Let~$\mf{a}$ be the self-dual hereditary order defined by~$\Lambda$.  Then
\[\mf{a}/\mf{a}_1\simeq \left(\prod_{k\in \K_+}\M_{n_k}(k_\D)\right)\times \left(\prod_{k\in\K_0}\M_{n_k}(k_\D)\right)\]
where~$\K_0$ corresponds precisely to the blocks fixed by~$\sigma_h$. We lift the primitive idempotents of the right hand side to~$\mf{a}$ {to get}~$f=(f_k)_{k\in\K_0\cup\K_+\cup\K_-}$, such that idempotents of~$e^\M$ and~$f$ commute (Note that~$\Delta$ is adapted to~$\M$).
Assume that there is an element~$g\in\P(\Lambda_0)$ which is not an element of~$\P^{\st}(\Lambda)$. Then~$(\tP(\La)\cap\U(\V,h))/\P_1(\La)$ has a factor isomorphic to~$\O(1,1)(k_\D)$, say at some~$k_0\in\K_0$, such that  the projection of $g$ onto the~$k_0$th factor has determinant~$-1$. Thus~$f_{k_0}e_0=f_{k_0}$
and the~$k_0$th~$\O(1,1)(k_\D)$-factor occurs in~$(\tP(\La_0)\cap\U(e_0\V,h|_{e_0\V}))/\P_1(\La_0)$ and still the projection of~$g$ onto this factor has determinant~$-1$. Then~$g$ cannot be an element of~$\P^{\st}(\Lambda_0)$. 
\end{proof}

\bibliographystyle{plain}
\bibliography{Endosplitting}

\end{document}